\documentclass[10pt]{amsart}
\usepackage{fullpage,url,amssymb,amsmath,amsfonts,amsthm,mathrsfs}
\usepackage[all]{xy} \SelectTips{cm}{10}

\usepackage{mathpazo}

\usepackage{color}
\definecolor{webcolor}{rgb}{0.8,0,0.2}
\definecolor{webbrown}{rgb}{.6,0,0}
\usepackage[
        colorlinks,
        linkcolor=webbrown,  filecolor=webcolor,  citecolor=webbrown, 
        backref,
        pdfauthor={David Zywina}, 
       pdftitle={An effective open image theorem for abelian varieties},
]{hyperref}
\usepackage[alphabetic,backrefs,lite]{amsrefs} 

\numberwithin{equation}{section}

\newcommand{\CC}{\mathbb C}

\newcommand{\FF}{\mathbb F}
\newcommand{\GG}{\mathbb G}

\newcommand{\QQ}{\mathbb Q}

\newcommand{\ZZ}{\mathbb Z}

\newcommand{\OO}{\mathcal O}
\newcommand{\calG}{\mathcal G}  \newcommand{\calF}{\mathcal F}
\newcommand{\calH}{\mathcal H}

\newcommand{\calB}{\mathcal B}
\newcommand{\calS}{\mathcal S}

\newcommand{\calC}{\mathcal C}

\newcommand{\calW}{\mathcal W}

\newcommand{\calM}{\mathcal M}
\newcommand{\calT}{\mathcal T}

\newcommand{\calV}{\mathcal V}
\newcommand{\calY}{\mathcal Y}

\newcommand{\p}{\mathfrak p}
\newcommand{\m}{\mathfrak m}

\newcommand{\scrA}{\mathscr A}

\newcommand{\scrF}{\mathscr F}

\newcommand{\boldG}{\mathbf G}

\def\rs{{\operatorname{rs}}}

\def\un{{\operatorname{un}}}

\def\conn{{\operatorname{conn}}}
\def\rank{{\operatorname{rank}}}
\def\Res{{\operatorname{Res}}}

\def\Spec{\operatorname{Spec}} 
\def\Gal{\operatorname{Gal}}
 
\def \GL {\operatorname{GL}}  

\def \GSp {\operatorname{GSp}}

\def \PSL {\operatorname{PSL}}
\def \Sp {\operatorname{Sp}}
\def\Aut{\operatorname{Aut}} 
\def\End{\operatorname{End}}
\def\Frob{\operatorname{Frob}}

\def\T{\operatorname{T}}

\def\Ndim{\operatorname{Ndim}}
\def\sc{{\operatorname{sc}}}

\newcommand{\q}{\mathfrak q}
\newcommand{\MT}{\operatorname{MT}}

\newcommand{\defi}[1]{\textsf{#1}} 

\newcommand\blank[1]{}

\def\bbar#1{\setbox0=\hbox{$#1$}\dimen0=.2\ht0 \kern\dimen0 
\overline{\kern-\dimen0 #1}}
\newcommand{\Qbar}{{\overline{\mathbb Q}}} 
\newcommand{\Kbar}{{\bbar{K}}}

\newcommand{\FFbar}{\overline{\FF}} 

\newtheorem{thm}{Theorem}[section]
\newtheorem{lemma}[thm]{Lemma}
\newtheorem{cor}[thm]{Corollary}
\newtheorem{prop}[thm]{Proposition}
\newtheorem{conj}[thm]{Conjecture}

\theoremstyle{definition}
\newtheorem{definition}[thm]{Definition}

\theoremstyle{remark}
\newtheorem{remark}[thm]{Remark}

\newenvironment{romanenum}{\hfill \begin{enumerate} }{\end{enumerate}}
\newenvironment{alphenum}{\hfill \begin{enumerate} }{\end{enumerate}}

\begin{document}

\title[An effective open image theorem for abelian varieties]{An effective open image theorem for abelian varieties}
\author{David Zywina}
\address{Department of Mathematics, Cornell University, Ithaca, NY 14853, USA}
\email{zywina@math.cornell.edu}
\urladdr{http://www.math.cornell.edu/~zywina}

\date{{\today}}

\begin{abstract}
Fix an abelian variety $A$ of dimension $g\geq 1$ defined over a number field $K$.  For each prime $\ell$, the Galois action on the $\ell$-power torsion points of $A$ induces a representation $\rho_{A,\ell}\colon \Gal_K \to \GL_{2g}(\ZZ_\ell)$.   The $\ell$-adic monodromy group of $A$ is the Zariski closure $G_{A,\ell}$ of the image of $\rho_{A,\ell}$ in $\GL_{2g,\QQ_\ell}$.   The image of $\rho_{A,\ell}$ is open in $G_{A,\ell}(\QQ_\ell)$ with respect to the $\ell$-adic topology and hence the index $[G_{A,\ell}(\QQ_\ell)\cap \GL_{2g}(\ZZ_\ell): \rho_{A,\ell}(\Gal_K)]$ is finite.   We prove that this index can be bounded in terms of $g$ for all $\ell$ larger then some constant depending on certain invariants of $A$.
\end{abstract}

\maketitle

\section{Introduction} \label{S:intro}

Let $A$ be an abelian variety of dimension $g\geq 1$ defined over a number field $K$.   Fix an algebraic closure $\Kbar$ of $K$ and define the absolute Galois group $\Gal_K:=\Gal(\Kbar/K)$.

Take any rational prime $\ell$.  For each positive integer $n$, we denote by $A[\ell^n]$ the $\ell^n$-torsion subgroup of $A(\Kbar)$. The group $A[\ell^n]$ is a free $\ZZ/\ell^n \ZZ$-module of rank $2g$ and comes with a natural $\Gal_K$-action that respects the group structure.
The $\ell$-adic \defi{Tate module} of $A$ is 
\[
T_\ell(A):= \varprojlim_n A[\ell^n],
\]
where the transition maps $A[\ell^{n+1}]\to A[\ell^n]$ are multiplication by $\ell$; it is a free $\ZZ_\ell$-module of rank $2g$.    The induced Galois action on $T_\ell(A)$ can be expressed in terms of a continuous representation
\[
\rho_{A,\ell} \colon \Gal_K \to \Aut_{\ZZ_\ell}(T_\ell(A)) = \GL_{T_\ell(A)}(\ZZ_\ell),
\]   
where $\GL_{T_\ell(A)}$ is the group scheme over $\ZZ_\ell$ for which $\GL_{T_\ell(A)}(R)=\Aut_R(T_\ell(A)\otimes_{\ZZ_\ell} R)$ for all $\ZZ_\ell$-algebras $R$ with the obvious functoriality.    After choosing a basis for $T_\ell(A)$, one could have $\rho_{A,\ell}$ mapping to $\GL_{2g}(\ZZ_\ell)$; it will make our arguments easier not to make such an arbitrary choice.

Define $V_\ell(A):= T_\ell(A) \otimes_{\ZZ_\ell} \QQ_\ell$; it is a $\QQ_\ell$-vector space of dimension $2g$  and inherits the Galois action from $T_\ell(A)$.  We can define a group scheme $\GL_{V_\ell(A)}$ over $\QQ_\ell$ as above; it is the generic fiber of $\GL_{T_\ell(A)}$.   We can view $\GL_{T_\ell(A)}(\ZZ_\ell)$, and hence also $\rho_{A,\ell}(\Gal_K)$, as a subgroup of $\GL_{V_\ell(A)}(\QQ_\ell)=\Aut_{\QQ_\ell}(V_\ell(A))$.  

To study the group $\rho_{A,\ell}(\Gal_K)$, we consider a related algebraic group defined over $\QQ_\ell$.

\begin{definition}
The \defi{$\ell$-adic monodromy group} of $A$ is the algebraic group $G_{A,\ell}$ defined over $\QQ_\ell$ obtained by taking the Zariski closure of $\rho_{A,\ell}(\Gal_K)$ in $\GL_{V_\ell(A)}$.  Let $\calG_{A,\ell}$ be the algebraic group over $\ZZ_\ell$ obtained by taking the Zariski closure of $\rho_{A,\ell}(\Gal_K)$ in $\GL_{T_\ell(A)}$.
\end{definition}

Note that the group schemes $G_{A,\ell}$ and $\calG_{A,\ell}$ determine each other; $G_{A,\ell}$ is the generic fiber of $\calG_{A,\ell}$ and $\calG_{A,\ell}$ is the Zariski closure of $G_{A,\ell}$ in $\GL_{T_\ell(A)}$.   

The group $\rho_{A,\ell}(\Gal_K)$ is open in $G_{A,\ell}(\QQ_\ell)$ with respect to the $\ell$-adic topology, cf.~\cite{MR574307}.   In particular, $\rho_{A,\ell}(\Gal_K)$ is an open, and hence finite index, subgroup of $\calG_{A,\ell}(\ZZ_\ell)=G_{A,\ell}(\QQ_\ell)\cap \Aut_{\ZZ_\ell}(T_\ell(A))$.   An effective version would ask for effective bounds for $[\calG_{A,\ell}(\ZZ_\ell):\rho_{A,\ell}(\Gal_K)]$. As a special case of our work, we will see that
 \[
 [\calG_{A,\ell}(\ZZ_\ell): \rho_{A,\ell}(\Gal_K)] \ll_g 1
 \] 
 for all sufficiently large primes $\ell$.  The notation ``$\ll_g$'' indicates that the index can be bounded in terms of a constant that depends only on $g$, cf.~\S\ref{SS:notation}.   For future applications, in particular~\ref{Zywina-Large families}, we are interested in describing how large $\ell$ needs to be in terms of  invariants of $A$.

\subsection{Some quantities}

Before stating our main results, we need to define some quantities that will show up in the bounds.  See \S\ref{S:mondromy groups} for further details and references.

For each prime $\ell$, let $G_{A,\ell}^\circ$ be the neutral component of $G_{A,\ell}$, i.e., the algebraic subgroup of $G_{A,\ell}$ that is the connected component of the identity.   The algebraic group $G_{A,\ell}^\circ$ is reductive and its rank $r$ is independent of $\ell$.   Let $\calG_{A,\ell}^\circ$ be the $\ZZ_\ell$-group subscheme of $\calG_{A,\ell}$ that is the Zariski closure of $G_{A,\ell}^\circ$.   Let $K_A^\conn$ be the minimal extension of $K$ in $\Kbar$ for which $\rho_{A,\ell}(\Gal_{K_A^\conn})\subseteq G_{A,\ell}^\circ(\QQ_\ell)$.   The field $K_A^\conn$ is a number field that is independent of $\ell$.   The extension $K_A^\conn/K$ is unramified at all primes ideals for which $A$ has good reduction and the degree $[K_A^\conn:K]$ can be bounded in terms of $g$.

Let $\p$ be any non-zero prime ideal of $\OO_K$ for which $A$ has good reduction.   Denote by $P_{A,\p}(x)$ the \defi{Frobenius polynomial} of $A$ at $\p$; it is a monic polynomial of degree $2g$ with integer coefficients.   For a prime $\ell$ satisfying $\p\nmid \ell$, the representation $\rho_{A,\ell}$ is unramified at $\p$ and we have
\[
P_{A,\p}(x)= \det(xI - \rho_{A,\ell}(\Frob_\p)).
\]
Let  $\Phi_{A,\p}$ be the subgroup of $\CC^\times$ generated by the roots of $P_{A,\p}(x)$.   If $\Phi_{A,\p}$ is a free abelian group, then it has rank at most $r$.   Moreover, the set of primes $\p$ for which $\Phi_{A,\p}$ is a free abelian group of rank $r$ has positive density.

Denote by $h(A)$ the (logarithmic absolute semistable) Faltings height of $A$.

\subsection{Main results}
 We can now state the main theorem of the paper.   As before, $A$ is an abelian variety of dimension $g\geq 1$ defined over a number field $K$. Let $\q$ be a non-zero prime ideal of $\OO_K$ for which $\Phi_{A,\q}$ is a free abelian group of rank $r$, where $r$ is the common rank of the reductive groups $G_{A,\ell}^\circ$.   
 
\begin{thm} \label{T:main new}
 There are positive constants $c$ and $\gamma$, depending only on $g$, such that for any prime $\ell$ satisfying
 \begin{align} \label{E:main ell bound}
 \ell \geq c \cdot \max(\{[K:\QQ],h(A), N(\q)\})^\gamma
 \end{align}
 the following hold:
\begin{alphenum}
\item  \label{T:main new a}
     $[\calG_{A,\ell}(\ZZ_\ell): \rho_{A,\ell}(\Gal_K)] \ll_{g,[K:\QQ]} 1$,
\item \label{T:main new b}
if $\ell$ is unramified in $K_A^\conn$, then $[\calG_{A,\ell}(\ZZ_\ell): \rho_{A,\ell}(\Gal_K)] \ll_{g} 1$.
     
\item  \label{T:main new c}
      $\calG_{A,\ell}^\circ$ is a reductive group scheme over $\ZZ_\ell$,    
\item   \label{T:main new d}
      the commutator subgroups of $\rho_{A,\ell}(\Gal_{K_A^\conn})$ and $\calG^\circ_{A,\ell}(\ZZ_\ell)$ agree.
\end{alphenum}
\end{thm}

\begin{remark}
Building on the work of Serre, Wintenberger proved that parts (\ref{T:main new c}) and (\ref{T:main new d}) of Theorem~\ref{T:main new} hold for all primes $\ell \geq C$ with  $C$ a constant depending on $A$, cf.~\cite{MR1944805}*{\S2.1}.  Using results of Serre, it is then easy to show that $[\calG_{A,\ell}(\ZZ_\ell): \rho_{A,\ell}(\Gal_K)] \ll_A 1$ holds for all $\ell$.

In \S2.2 of \cite{MR1944805}, Wintenberger discusses how one could make the bound $\ell\geq C$ effective.    He mentioned that it should be possible to consider $\ell$ large enough in terms of the Faltings height $h(A)$, a related field $K$, the prime ideal $\q$, and the finite set $S$ of prime ideals of $\OO_K$ for which $A$ has bad reduction.   We tried to work this out following Wintenberger's approach but the dependencies on the set $S$ were not appropriate for our applications where we vary $A$ in a geometric family.

Note that, except for possibly part~(\ref{T:main new b}), the set $S$ of bad primes of $A$ do not occur in our theorem.  This was achieved by using the effective bounds of Masser--W\"ustholz (Theorem~\ref{T:MW}) to give a new streamlined proof.  In particular, our proof of (\ref{T:main new c}) and (\ref{T:main new d}) does not require inertia groups which is a key ingredient in the work of Serre and Wintenberger (for example, see \S\S3.1--3.3 of \cite{MR1944805}).
\end{remark}

  Assuming the Generalized Riemann Hypothesis (GRH) for number fields, we can give a version of Theorem~\ref{T:main new} that does not involve the prime ideal $\q$.
Let $D$ be the product of primes $p$ that ramify in $K$ or are divisible by a prime ideal for which $A$ has bad reduction. 

\begin{thm} \label{T:GRH bound for Nq}
Suppose that GRH holds.  Theorem~\ref{T:main new} holds with (\ref{E:main ell bound}) replaced by 
\[
\ell \geq c \cdot \max(\{[K:\QQ],h(A), \log D\})^\gamma,
\] 
where $c$ and $\gamma$ are positive constants  depending only on $g$.
\end{thm}

\begin{remark}  
The difficulty with bounding the minimal possible $N(\q)$ is that the most natural way to do this is via the Chebotarev density theorem, and this  requires knowledge about the image of a representation $\rho_{A,\ell}$ (which is the thing we are trying to study in the first place!). Our proof of Theorem~\ref{T:GRH bound for Nq} uses a variant of Theorem~\ref{T:main new} along with an effective version of the Chebotarev density theorem. 
Unfortunately, unconditional versions of the Chebotarev density theorem do not produce upper bounds for $N(\q)$.
\end{remark}

For $\ell$ sufficiently large, one can show that $T_\ell(A)$ has a basis over $\ZZ_\ell$ such that, with respect to this basis, $\rho_{A,\ell}(\Gal_K)\subseteq\GSp_{2g}(\ZZ_\ell)$.

 \begin{cor}  \label{C:maximal monodromy}
Suppose that $\End(A_{\Kbar})=\ZZ$ and that there is a non-zero prime ideal $\q$ of $\OO_K$ for which $A$ has good reduction and for which $\Phi_{A,\q}$ is a free abelian group of rank $g+1$.  Then there are positive constants $c$ and $\gamma$, depending only on $g$, such that 
  \[
\rho_{A,\ell}(\Gal_K)=\GSp_{2g}(\ZZ_\ell)
 \]
 holds for all primes $\ell\geq c \cdot \max(\{[K:\QQ],h(A), N(\q)\})^\gamma$ that are unramified in $K$.
 \end{cor}

Corollary~\ref{C:maximal monodromy} recovers the main result of Lombardo in \cite{explicit} except that he gives explicit numerical values of $c$ and $\gamma$  (he also gives another condition on a prime ideal $\q$ that implies ours).   In principle our constants can be made explicit, but working them out in the full generality of Theorem~\ref{T:main new} would be quite a chore. Note that Lombardo's methods are different than ours; he studies the maximal proper subgroups of $\GSp_{2g}(\FF_\ell)$ and makes use of inertia groups.  
 
 \begin{remark}
 Suppose that we have $\rho_{A,\ell}(\Gal_K)=\GSp_{2g}(\ZZ_\ell)$ for all sufficiently large $\ell$.   In this case, one can show that $\End(A_{\Kbar})=\ZZ$ and that $\Phi_{A,\p}$ is a free abelian group of rank $g+1$ for all prime ideals $\p\subseteq \OO_K$ away from a set of density $0$.  This justifies the assumptions of Corollary~\ref{C:maximal monodromy}.
 \end{remark}
 
\subsection{Notation} \label{SS:notation}

For two real quantities $f$ and $g$, we write that $f \ll_{\alpha_1,\ldots,\alpha_n} g$ if the inequality $|f| \leq C |g|$ holds for some positive constant $C$ depending only on $\alpha_1,\ldots,\alpha_n$.  In particular, $f\ll g$ means that the implicit constant is absolute.  We denote by $O_{\alpha_1,\ldots,\alpha_n}(g)$ a quantity $f$ satisfying $f \ll_{\alpha_1,\ldots,\alpha_n} g$.

Fix a number field $F$. We denote the ring of integers of $F$ by $\OO_F$.   For a non-zero prime ideal $\p$ of $\OO_F$, we denote by $F_\p$ the $\p$-adic completion of $F$ and let $\OO_\p$ be the valuation ring of $F_\p$.  The residue field of $\OO_\p$ agrees with $\FF_\p:=\OO_K/\p$.  For a field $F$, let $\bbar{F}$ be a fixed algebraic closure of $F$ and define $\Gal_F:=\Gal(\bbar{F}/F)$.

 For a scheme $X$ over a commutative ring $R$ and a commutative $R$-algebra $S$, we denote by $X_S$ the base extension of $X$ by $\Spec S$.  
 
 Let $M$ be a free module of finite rank over a commutative ring $R$.  Denote by $\GL_M$ the $R$-scheme such that $\GL_M(S)= \Aut_S(M\otimes_R S)$ for any commutative $R$-algebra $S$ with the obvious functoriality. 

For an algebraic group $G$ over a field $F$, we denote by $G^\circ$ the neutral component of $G$ (i.e., the connected component of the identity of $G$); it is an algebraic subgroup of $G$.  For an algebraic group $T$ of multiplicative type defined over $F$, let $X(T)$ be the group of characters $T_{\bbar{F}}\to \GG_{m,\bbar{F}}$; it has a natural $\Gal_F$-action.  If $T$ is a torus, then the group $X(T)$ is free abelian whose rank is equal to the dimension of $T$.

Consider a topological group $G$. Note that profinite groups, will always be considered with their profinite topology.   The \defi{commutator subgroup} of $G$ is the closed subgroup generated by the commutators of $G$; we denote it by $G'$.

\subsection{Overview}

In \S\ref{S:mondromy groups}, we recall several fundamental results about the Galois representations $\rho_{A,\ell}$ and their $\ell$-adic monodromy group $G_{A,\ell}$.   We also state the \emph{Mumford--Tate conjecture} for $A$ in \S\ref{SS:MT group} which says that $G_{A,\ell}^\circ$ arises from a certain reductive group defined over $\QQ$ that is independent of $\ell$.  In \S\ref{SS:reduction to connected case}, we show that it suffices to prove Theorem~\ref{T:main new} in the special case where all the groups $G_{A,\ell}$ are connected.

In \S\ref{S:proof of c and d}, we prove parts (\ref{T:main new c}) and (\ref{T:main new d}) of Theorem~\ref{T:main new}.   Our new proof makes key use of theorems of Masser--W\"ustholz and Larsen--Pink.    Some of the needed group theory is described in \S\ref{S:some group theory}.

In \S\ref{S:abelian}, assuming the $\ell$-adic monodromy groups $G_{A,\ell}$ are connected, we construct abelian representations $\beta_{A,\ell}\colon \Gal_K \to Y(\QQ_\ell)_c$, where $Y$ is a certain torus over $\QQ$ and $Y(\QQ_\ell)_c$ is the maximal compact subgroup of $Y(\QQ_\ell)$ with respect to the $\ell$-adic topology.  We show that the image of $\beta_{A,\ell}(\Gal_K)$ is open in $Y(\QQ_\ell)_c$.  Moreover, we show that the index $[Y(\QQ_\ell)_c:\beta_{A,\ell}(\Gal_K)]$ can be bounded in terms of $g$ if $\ell$ is unramified in $K$ and in terms of $g$ and $[K:\QQ]$ otherwise.    This bound is key in deducing Theorem~\ref{T:main new}(\ref{T:main new a}) and (\ref{T:main new b}) from Theorem~\ref{T:main new}(\ref{T:main new d}).

In \S\ref{S:end of main proof}, we complete the proof of Theorem~\ref{T:main new}.  In \S\ref{S:GRH bound for Nq} and \S\ref{S:proof of maximal monodromy}, we prove Theorem~\ref{T:GRH bound for Nq} and Corollary~\ref{C:maximal monodromy}, respectively.

\subsection{Acknowledgements}

 Special thanks to David Zureick-Brown; this article is spun off from earlier work with him.  Thanks also to Chun Yin Hui.

\section{Background: \texorpdfstring{$\ell$}{l}-adic monodromy groups} \label{S:mondromy groups}

Fix an abelian variety $A$ of dimension $g\geq 1$ defined over a number field $K$.   

\subsection{Compatibility}

Take any non-zero prime ideal $\p$ of $\OO_K$ for which $A$ has good reduction.  Denote by $A_\p$ the abelian variety over $\FF_\p$ obtained by reducing $A$ modulo $\p$.      There is a unique polynomial $P_{A,\p}(x) \in \ZZ[x]$ such that $P_{A,\p}(n)$ is the degree of the isogeny $n-\pi$ for each integer $n$, where $\pi$ is the Frobenius endomorphism of $A_\p/\FF_\p$.   The polynomial $P_{A,\p}(x)$ is monic of degree $2g$.  For each rational prime $\ell$ for which $\p\nmid \ell$, the representation $\rho_{A,\ell}$ is unramified at $\p$ and satisfies 
\[
\det(xI - \rho_{A,\ell}(\Frob_\p)) = P_{A,\p}(x).
\]
In the notation of \cite{MR1484415}, the Galois representations $\{\rho_{A,\ell}\}_\ell$ form a \emph{strictly compatible} family.  

Note that $\rho_{A,\ell}(\Frob_\p)$ is semisimple in $\GL_{V_\ell(A)}$; this can be seen by noting that $\pi$ acts semisimply on the $\ell$-adic Tate module of $A_\p$.    From Weil, we know that all of the roots of $P_{A,\p}(x)$ in $\CC$ have absolute value $N(\p)^{1/2}$.

\subsection{Neutral component} 
 
Let $G_{A,\ell}^\circ$ be the neutral component of $G_{A,\ell}$, i.e., the connected component of $G_{A,\ell}$ containing the identity.  Define $\calG_{A,\ell}^\circ$ to be the $\ZZ_\ell$-group subscheme of $\calG_{A,\ell}$ that is the Zariski closure of $G_{A,\ell}^\circ$.

Define $K_A^\conn$ to be the subfield of $\Kbar$ fixed by the kernel of the homomorphism 
\begin{align} \label{E:conn hom}
\Gal_K \xrightarrow{\rho_{A,\ell}} G_{A,\ell}(\QQ_\ell)\to G_{A,\ell}(\QQ_\ell)/G_{A,\ell}^\circ(\QQ_\ell).  
\end{align}
Equivalently, $K_{A}^\conn$ is the smallest extension of $K$ in $\Kbar$ that satisfies $\rho_{A,\ell}(\Gal_{K_A^\conn}) \subseteq G_{A,\ell}^\circ(\QQ_\ell)$.  
\begin{prop} \label{P:connected}
\begin{romanenum}
\item \label{P:connected i}
The field $K_A^\conn$ depends only on $A$, i.e., it is independent of $\ell$.   
\item \label{P:connected ii}
The degree $[K_A^\conn:K]$ can be bounded in terms of $g$ only.
\item \label{P:connected iii}
We have 
\[
[\calG_{A,\ell}(\ZZ_\ell):\rho_{A,\ell}(\Gal_K)] = [\calG_{A',\ell}(\ZZ_\ell):\rho_{A',\ell}(\Gal_{K_A^\conn})],
\] 
where $A'$ is the base change of $A$ to $K_A^\conn$.
\end{romanenum}
\end{prop}
\begin{proof}
Part (\ref{P:connected i}) was proved by Serre \cite{MR1730973}*{133}; see also \cite{MR1441234}.   From (\ref{P:connected i}), we find that $K_A^\conn$ is a subfield of $K(A[\ell^\infty])$ and hence $[K_A^\conn:K]$ divides $[K(A[\ell]):K]\ell^{e_\ell}$ for some integer $e_\ell$.  Since $[K(A[\ell]):K]$ divides $|\GL_{2g}(\FF_\ell)|$, we deduce that $[K_A^\conn :K]$ divides $|\GL_{2g}(\FF_\ell)| \ell^{e_\ell}$.   Therefore, $[K_A^\conn:K]$ must divide $|\GL_{2g}(\FF_2)|\cdot |\GL_{2g}(\FF_3)|$ which completes the proof of (\ref{P:connected ii}).

Note that the homomorphism (\ref{E:conn hom}) is surjective since $G_{A,\ell}$ is the Zariski closure of $\rho_{A,\ell}(\Gal_K)$.     Therefore,  $[\calG_{A,\ell}(\ZZ_\ell):\calG_{A,\ell}^\circ(\ZZ_\ell)] = [G_{A,\ell}(\QQ_\ell):G_{A,\ell}^\circ(\QQ_\ell)]=[K_A^\conn:K]$.  Part (\ref{P:connected iii}) follows since $[\rho_{A,\ell}(\Gal_K): \rho_{A,\ell}(\Gal_{K_A^\conn})]=[K_A^\conn:K]$.
\end{proof}

\subsection{Tate conjecture}

The following, which was conjectured by Tate, is an important result of Faltings, cf.~\cite{MR861971}.

\begin{thm}[Faltings] \label{T:Tate conjecture}
\begin{romanenum}
\item \label{T:Tate conjecture i}
The $\QQ_\ell[\Gal_K]$-module $V_\ell(A)$ is semisimple.
\item  \label{T:Tate conjecture ii}
The natural map $\End(A) \otimes_\ZZ \QQ_\ell \hookrightarrow \End_{\QQ_\ell[\Gal_K]}(V_\ell(A))$ is an isomorphism.
\end{romanenum}
\end{thm}

Here are some basic properties of the groups $G_{A,\ell}^\circ$.  

\begin{prop}  \label{P:conn facts}
\begin{romanenum}
\item  \label{P:conn facts i}
The  group $G_{A,\ell}^\circ$ is reductive.
\item  \label{P:conn facts ii} 
For any number field $L$ containing $K_A^\conn$, the group $\rho_{A,\ell}(\Gal_L)$ is Zariski dense in $G_{A,\ell}^\circ$.
\item  \label{P:conn facts iii}
The commutant of $G_{A,\ell}^\circ$ in $\End_{\QQ_\ell}(V_\ell(A))$ agrees with $\End(A_{\Kbar})\otimes_\ZZ \QQ_\ell$.
\item   \label{P:conn facts iv}
All of the endomorphisms of $A$ over $\Kbar$ are defined over $K_A^\conn$, i.e., $\End(A_{\Kbar})=\End(A_{K_A^\conn})$.
\end{romanenum}
\end{prop}
\begin{proof}
Part (\ref{P:conn facts i}) can be deduced from Theorem~\ref{T:Tate conjecture}.

Take any number field $L\supseteq K_A^\conn$.  We have $\rho_{A,\ell}(\Gal_L) \subseteq G_{A,\ell}^\circ(\QQ_\ell)$, so $G_{A_L,\ell}$ is a finite index subgroup of $G_{A,\ell}^\circ$.  Since $G_{A,\ell}^\circ$ is connected, we have $G_{A_L,\ell}=G_{A,\ell}^\circ$.  This proves (\ref{P:conn facts ii}).

From part (\ref{P:conn facts ii}) and Theorem~\ref{T:Tate conjecture}(\ref{T:Tate conjecture ii}), we find that $\End(A_L)\otimes_\ZZ \QQ_\ell$ is naturally isomorphic to the subring of $\End_{\QQ_\ell}(V_\ell(A))$ that commutes with $G_{A_L,\ell}=G_{A,\ell}^\circ$.   Since this holds for all $L\supseteq K_A^\conn$, we deduce that
\begin{equation} \label{E:conn facts}
\End(A_{K_A^\conn})\otimes_\ZZ \QQ_\ell = \End(A_{\Kbar})\otimes_\ZZ \QQ_\ell.
\end{equation}
This proves part (\ref{P:conn facts iii}).   Since $\End(A_{\Kbar})$ is a free abelian group and $\Gal_{K_A^\conn}$ acts trivially on $\End(A_{\Kbar}) \otimes_\ZZ \QQ_\ell$ by (\ref{E:conn facts}), we deduce that $\Gal_{K_A^\conn}$ acts trivially on $\End(A_{\Kbar})$. This proves (\ref{P:conn facts iv}).  
\end{proof}

We will also need the following effective modulo $\ell$ version of Faltings' theorem due to Masser and W\"ustholz.   We denote by $h(A)$ the (logarithmic absolute) Faltings height of $A$ obtained after base extending to any finite extension of $K$ over which $A$ has semistable reduction (see \S5 of \cite{MR861978}).   In particular, note that $h(A_L)=h(A)$ for any finite extension $L/K$.

\begin{thm}[Masser-W\"ustholz]  \label{T:MW}
Let $L$ be a finite extension of $K$.  There are positive constants $c$ and $\gamma$, depending only on the dimension of $A$, such that if $\ell \geq c(\max\{[L:\QQ], h(A)\})^{\gamma}$, then the following hold:
\begin{itemize}
\item 
the $\FF_\ell[\Gal_L]$-module $A[\ell]$ is semisimple,
\item 
the natural map $\End(A_L) \otimes_\ZZ \FF_\ell \to \End_{\FF_\ell[\Gal_L]}(A[\ell])$ is an isomorphism,
\item
$\End(A_L)\otimes_\ZZ \FF_\ell$ is a semisimple $\FF_\ell$-algebra.
\end{itemize}
\end{thm}
\begin{proof}
Since $h(A)=h(A_L)$, it suffices to prove the theorem in the special case $L=K$.   

The first two conclusions of the theorem follow from Corollaries 1 and 2 in \S1 of \cite{MR1336608}; see the last remark of their paper for the stated dependence on $K$.  By Lemmas~2.3 and 5.2 of \cite{MR1336608}, we deduce $\End(A)\otimes_\ZZ \FF_\ell$ is semisimple after suitable increasing $c$ and $\gamma$; as before, see the last remark of loc.~cit.~for the stated dependence on $K$. 
\end{proof}

\subsection{Computing \texorpdfstring{$\Phi_{A,\p}$}{Frobenius groups}}

Fix a non-zero prime ideal $\p$ of $\OO_K$ for which $A$ has good reduction.  Let  $\Phi_{A,\p}$ be the subgroup of $\CC^\times$ generated by the roots of $P_{A,\p}(x)$.  

Let $\pi_1,\ldots, \pi_{2g} \in \CC$ be the roots of $P_{A,\p}(x)$ with multiplicity.     Define the number field $L=\QQ(\pi_1,\ldots, \pi_{2g})$.   We have a surjective homomorphism 
\[
\varphi\colon \ZZ^{2g} \to \Phi_{A,\p},\quad e\mapsto \prod_{i=1}^{2g} \pi_i^{e_i}.
\]
To compute the group $\Phi_{A,\p}$, we need to describe the kernel of $\varphi$.  We first describe the $e \in \ZZ^{2g}$ for which $\varphi(e)$ is a root of unity.  For each non-zero prime $\lambda$ of $\OO_L$, let $v_\lambda\colon L^\times\twoheadrightarrow\ZZ$ be the $\lambda$-adic valuation.

\begin{lemma} \label{L:root of unity linear}
Take any $e\in \ZZ^{2g}$.  The following are equivalent:
\begin{alphenum}
\item \label{L:root of unity linear a}
$\varphi(e)$ is a root of unity,
\item \label{L:root of unity linear b}
$\sum_{i=1}^{2g}  v_\lambda(\pi_i) \cdot e_i= 0$ 
holds for all prime ideals $\lambda | N(\p)$ of $\OO_L$,
\end{alphenum}
\end{lemma}
\begin{proof}
For a fixed $e\in \ZZ^{2g}$, define $\alpha:=\varphi(e)=\prod_{i=1}^{2g} \pi_i^{e_i} \in L^\times$.    Observe that for a non-zero prime $\lambda$ of $\OO_L$, we have $v_\lambda(\alpha)=\sum_{i=1}^{2g} v_\lambda(\pi_i)\cdot e_i $.  If $\alpha$ is a root of unity, then we have $v_\lambda(\alpha)=0$ for all $\lambda$.   Therefore, (\ref{L:root of unity linear a}) implies (\ref{L:root of unity linear b}).

We now assume that (\ref{L:root of unity linear b}) holds, i.e., $v_\lambda(\alpha)=0$ for all prime ideals $\lambda|N(\p)$ of $\OO_L$.   We need to show that $\varphi(e)$ is a root of unity.

Take any non-zero prime ideal $\lambda \nmid N(\p)$ of $\OO_L$.  For each $\pi_i$, we have $\pi_i \bbar{\pi}_i = N(\p)$, where $\bbar{\pi}_i$ is the complex conjugate of $\pi_i$ under any complex embedding.   So $v_\lambda(\pi_i)+v_\lambda(\bbar\pi_i)=0$.   Since $\pi_i$ and $\bbar\pi_i$ are algebraic integers, we have $v_\lambda(\pi_i)\geq 0$ and $v_\lambda(\bbar\pi_i)\geq 0$, and hence $v_\lambda(\pi_i)=0$.   Therefore, $v_\lambda(\alpha)=0$.   Combing this with our assumption, we deduce that $v_\lambda(\alpha)=0$ for all non-zero prime ideals $\lambda$ of $\OO_L$.   This implies that $\alpha\in \OO_L^\times$.

Take any embedding $\iota\colon L\hookrightarrow \CC$.  From Weil, we know that each $\iota(\pi_i)$ has absolute value $N(\p)^{1/2}$.  Therefore, $|\iota(\alpha)|=   N(\p)^{(e_1+\cdots +e_{2g})/2}$ for any $\iota$ and hence $|N_{L/\QQ}(\alpha)| = N(\p)^{[L:\QQ]\,  (e_1+\cdots +e_{2g})/2}$.   We have $N_{L/\QQ}(\alpha)=\pm 1$ since $\alpha\in \OO_L^\times$, so $e_1+\cdots +e_{2g}=0$.    Therefore, $\alpha$ has absolute value $1$ under any embedding into $\CC$.   Since $\alpha$ is a unit in $\OO_L$ with absolute value $1$ under any embedding into $\CC$, we conclude that $\alpha$ is a root of unity.
\end{proof}

We now describe how to compute the kernel of $\varphi$.  
Let $M\subseteq \ZZ^{2g}$ be the group of $e\in \ZZ^{2g}$ for which $\sum_{i=1}^{2g} v_\lambda(\pi_i)\cdot e_i  = 0$ for all prime ideals $\lambda|N(\p)$ of $\OO_L$.   By Lemma~\ref{L:root of unity linear}, we have $\varphi^{-1}(\mu_L)=M$, where $\mu_L$ is the (finite) group of roots of unity in $L^\times$.    Define the homomorphism
\[
\varphi|_{M} \colon M \to \mu_L,\quad e\mapsto \prod_{i=1}^{2g} \pi^{e_i}.
\]
   Computing $\varphi$ on a basis of $M$, one can then explicitly compute $\ker \varphi =  \ker (\varphi|_M) \subseteq \ZZ^{2g}$.  

The following finiteness result will be used multiple times in our proofs.
\begin{prop} \label{P:Frob group possibilities}
Let $\p$ be a non-zero prime ideal of $\OO_K$ for which $A$ has good reduction.  Let $\pi_1,\ldots, \pi_{2g} \in \CC$ be the roots of $P_{A,\p}(x)$ with multiplicity.    Let $M$ be the group of $e \in \ZZ^{2g}$ for which $\prod_{i=1}^{2g} \pi_i^{e_i} = 1$.  There are a finite number of subgroups $M_1,\ldots, M_s$ of $\ZZ^{2g}$, depending only on $g$ (and not on $A$ and $\p$), such that $M$ equals one of the $M_i$.
\end{prop}
\begin{proof}
Define the number field $L:=\QQ(\pi_1,\ldots,\pi_{2g})$.  Let $N$ be the group of $e \in \ZZ^{2g}$ for which $\prod_{i=1}^{2g} \pi_i^{e_i}$ is a root of unity.  Take any $e \in \ZZ^{2g}$.  By Lemma~\ref{L:root of unity linear}, we have $e\in N$ if and only if $\sum_{i=1}^{2g} v_\lambda(\pi_i) \cdot e_i = 0$ holds for prime ideals $\lambda|\ell$ of $\OO_L$.     We have $[L:\QQ]\ll_g 1$ since $L$ is the splitting field of a degree $2g$ polynomial and hence $0\leq v_\lambda(\pi_i) \leq [L:\QQ]\ll_g 1$.   So $N$ is defined in $\ZZ^{2g}$ by a finite number of linear equations with integer coefficients, where the number of equations and size of the coefficients can be bounded in terms of $g$.     So $N$ is one of a finite number of subgroups $N_1,\ldots, N_m$ of $\ZZ^{2g}$ that depend only on $g$.

The group $M$ is the kernel of $N\to \mu_L$, $e\to \prod_{i=1}^{2g} \pi_i^{e_i}$, where $\mu_L$ is the group of roots of unity in $L^\times$.   We have $|\mu_L|\ll_g 1$ since $[L:\QQ]\ll_g 1$.   So we have $[N: M]\leq C$ for some constant $C$ depending only on $g$.   The  proposition thus holds where $M_1,\ldots, M_s$ are the subgroups of $N_1,\ldots, N_m$ of index at most $C$.
\end{proof}

\subsection{Common rank}
The results in this section are due to Serre and details can be found in \cite{MR1441234}.
Fix a prime $\ell$ and denote by $r$ the rank of the reductive group $G_{A,\ell}^\circ$.

\begin{lemma} \label{L:pre common rank}
\begin{romanenum}
\item \label{L:pre common rank i}
Let $\p$ be a non-zero prime ideal of $\OO_K$ for which $A$ has good reduction.  If $\Phi_{A,\p}$ is a free abelian group, then $\p$ splits completely in $K_A^\conn$ and $\Phi_{A,\p}$ has rank at most $r$.
\item \label{L:pre common rank ii}
There is a set $S$ of prime ideals of $\OO_K$ with density $0$ such that $\Phi_{A,\p}$ is a free abelian group of rank $r$ for all $\p\notin S$ that split completely in $K_A^\conn$.
\end{romanenum}
\end{lemma}
\begin{proof}
Take any non-zero prime ideal $\p\nmid \ell$ of $\OO_K$ for which $A$ has good reduction.   Let $T_\p$ be the Zariski closure in $G_{A,\ell}$ of the subgroup generated by the semisimple element $t_\p:=\rho_{A,\ell}(\Frob_\p)$.    Note that $T_\p$ is a commutative algebraic subgroup of $G_{A,\ell}$ and $T_\p^\circ$ is a torus.     Let $X(T_\p)$ be the group of characters $(T_\p)_{\Qbar_\ell} \to \GG_{m,\Qbar_\ell}$.  

We claim that the groups $X(T_\p)$ and $\Phi_{A,\p}$ are isomorphic.  Let $\Omega\subseteq X(T_\p)$ be the weights of $T_\p \subseteq \GL_{V_\ell(A)}$ acting on $V_\ell(A)$.    The set $\Omega$ generates $X(T_\p)$ since this action is faithful.  Since $T_\p$ is generated by $t_\p$, the homomorphism $f\colon X(T_\p) \to \Qbar_\ell^\times$, $\alpha\mapsto \alpha(t_\p)$ is injective.   The elements $\{\alpha(t_\p):\alpha\in \Omega\}$ are the roots of $P_{A,\p}(x)$ in $\Qbar_\ell$ and generate the image of $f$.   Therefore, we have an isomorphism $X(T_\p) \to \Phi_{A,\p}$, $\alpha\mapsto \tau(\alpha(t_\p))$, where $\tau\colon \Qbar_\ell \hookrightarrow \CC$ is a fixed embedding.   This proves the claim.

Suppose that $\Phi_{A,\p}$ is a free abelian group of rank $s$.   From the above claim, $X(T_\p)$ is free abelian of rank $s$ and hence $T_\p$ is a torus of rank $s$.    Since $T_\p$ is connected, we have $T_\p \subseteq G_{A,\ell}^\circ$.   We have $s\leq r$ since $G_{A,\ell}^\circ$ has rank $r$.   We have $\rho_{A,\ell}(\Frob_\p) \in T_\p(\QQ_\ell) \subseteq G_{A,\ell}^\circ(\QQ_\ell)$, so $\p$ splits completely in $K_A^\conn$ since $K_A^\conn$ is the fixed field in $\Kbar$ of (\ref{E:conn hom}).  This proves (\ref{L:pre common rank i}) for $\p\nmid \ell$.

The set of prime ideals $\p\subseteq\OO_K$ that split completely in $K_A^\conn$ for which $T_\p$ is \emph{not} a maximal torus of $G_{A,\ell}^\circ$ has density $0$; this follows from {Theorem~1.2} of \cite{MR1441234} which shows it under the assumption $K_A^\conn=K$.  Part (\ref{L:pre common rank ii}) is a direct consequence of the claim.

Finally, from (\ref{L:pre common rank ii}) we find that $r$ does not depend on the choice of $\ell$.  Therefore, (\ref{L:pre common rank i}) holds for the primes $\p|\ell$ that were excluded above.
\end{proof}

\begin{prop} \label{P:common rank}  
The rank $r$ of the reductive group $G_{A,\ell}^\circ$ does not depend on the prime $\ell$.
\end{prop}
\begin{proof}
This follows from Lemma~\ref{L:pre common rank}(\ref{L:pre common rank ii}) which gives a characterization of $r$ that does not depend on $\ell$.
\end{proof}

\subsection{The Mumford--Tate group} \label{SS:MT group}
Fix a field embedding $\Kbar \subseteq \CC$.  The homology group $V:=H_1(A(\CC),\QQ)$ is a vector space of dimension $2g$ over $\QQ$.   It is naturally endowed with a $\QQ$-Hodge structure of type $\{(-1,0),(0,-1)\}$ and hence a decomposition 
\[
V\otimes_\QQ\CC = H_1(A(\CC),\CC)=V^{-1,0} \oplus V^{0,-1}
\]
such that $V^{0,-1}=\bbar{V^{-1,0}}$.  Let
\[
\mu\colon \GG_{m,\CC} \to \GL_{V\otimes_\QQ\CC}  
\]
be the cocharacter such that for each $z\in\CC^\times=\GG_{m}(\CC)$, $\mu(z)$ is the automorphism of $V\otimes_\QQ\CC$ that is multiplication by $z$ on $V^{-1,0}$ and the identity on $V^{0,-1}$.

\begin{definition}
The \defi{Mumford--Tate group} of $A$ is the smallest algebraic subgroup of $\GL_V$ defined over $\QQ$ that contains $\mu(\GG_{m,\CC})$.  We will denote the Mumford--Tate group of $A$ by $\MT_{A}$.
\end{definition}

The endomorphism ring $\End(A_\CC)$ acts on $V$; this action preserves the Hodge decomposition, and hence commutes with $\mu$ and thus also $\MT_A$.  Moreover, the ring $\End(A_\CC)\otimes_\ZZ \QQ$ is naturally isomorphic to the commutant of $\MT_A$ in $\End_\QQ(V)$.   The group $\MT_A$ is reductive since the $\QQ$-Hodge structure for $V$ is pure and polarizable.  Using our fixed embedding $\Kbar\subseteq \CC$ and Proposition~\ref{P:conn facts}(\ref{P:conn facts iv}), we have a natural isomorphism $\End(A_\CC)\otimes_\ZZ \QQ = \End(A_{K_A^\conn})\otimes_\ZZ \QQ$.\\

The \defi{comparison isomorphism} $V_\ell(A)\cong V \otimes_\QQ \QQ_\ell$ induces an isomorphism $\GL_{V_\ell(A)} \cong \GL_{V,\,\QQ_\ell}$; we will use the comparison isomorphism as an identification.    The following conjecture says that $G_{A,\ell}^\circ$ and $(\MT_A)_{\QQ_\ell}$ are the same algebraic group under the comparison isomorphism, cf.~\cite{MR0476753}*{\S3}.

\begin{conj}[Mumford--Tate conjecture]   \label{C:MT}
For each prime $\ell$, we have $G_{A,\ell}^\circ= (\MT_A)_{\QQ_\ell}$.
\end{conj}
 
The following proposition says that one inclusion of the Mumford--Tate conjecture is known unconditionally, see Deligne's proof in \cite{MR654325}*{I, Prop.~6.2}.

\begin{prop}\label{P:MT inclusion}
For each prime $\ell$, we have $G_{A,\ell}^\circ \subseteq  (\MT_A)_{\QQ_\ell}$.
\end{prop}

One consequence of the Mumford--Tate conjecture is the central torus of the reductive group $G_{A,\ell}^\circ$ (i.e., the neutral component of the center) is independent of $\ell$.   This has been proved unconditionally.

\begin{prop} \label{P:same centers}
The central tori of $G_{A,\ell}^\circ$ and $(\MT_A)_{\QQ_\ell}$ agree for all $\ell$.
\end{prop}
\begin{proof}
See \cite{MR2400251}*{Theorem~1.3.1} or \cite{MR3117743}*{Corollary 2.11}.
\end{proof}

\subsection{Reduction to the connected case} \label{SS:reduction to connected case}

We now show that to prove our main theorem, we may assume that all the $\ell$-adic monodromy groups are connected.

\begin{lemma} \label{L:reduction to connected case}
To prove Theorem~\ref{T:main new}, it suffices to consider the case where all the groups $G_{A,\ell}$ are connected; equivalently, $K_A^\conn=K$.
\end{lemma}
\begin{proof}
Assume that Theorem~\ref{T:main new} holds in the case where all the $\ell$-adic monodromy groups are connected.

Define $L:=K_A^\conn$ and let $A'$ be the base change of $A$ to $L$.  Note that $K_{A'}^\conn=L$.  By Proposition~\ref{P:connected}(\ref{P:connected ii}), we have $[L:\QQ]\ll_g [K:\QQ]$.    We also have $h(A')=h(A)$.  The prime ideal $\q \subseteq \OO_K$ splits completely in $L$ by Lemma~\ref{L:pre common rank}(\ref{L:pre common rank i}).   Let $\q'$ be a prime ideal of $\OO_L$ that divides $\q$.   The natural map $\FF_\q \to \FF_{\q'}$ is an isomorphism and in particular $N(\q)=N(\q')$.  Under this isomorphism, the abelian varieties $A_\q$ and $A'_{\q'}$ agree and hence $\Phi_{A',\q'}$ is also a free abelian group of rank $r$.  The assumption (\ref{E:main ell bound}) thus implies that
\[
 \ell \geq c \cdot \max(\{[L:\QQ],h(A'), N(\q')\})^\gamma,
\]
where $c$ is a possibly larger positive constant that depends only on $g$.  By the assumed connected case of Theorem~\ref{T:main new} and Proposition~\ref{P:connected}(\ref{P:connected iii}), we have
\[
[\calG_{A,\ell}(\ZZ_\ell):\rho_{A,\ell}(\Gal_K)] = [\calG_{A',\ell}(\ZZ_\ell):\rho_{A',\ell}(\Gal_{L})] \ll_{g,[L:\QQ]} 1.
\] 
Since $[L:\QQ]\ll_g [K:\QQ]$, we have $[\calG_{A,\ell}(\ZZ_\ell):\rho_{A,\ell}(\Gal_K)] \ll_{g,[K:\QQ]} 1$.   This shows that Theorem~\ref{T:main new}(\ref{T:main new a}) holds for $A/K$.  

By the assumed connected case of Theorem~\ref{T:main new}, the group $\calG_{A',\ell}^\circ=\calG_{A',\ell}$ is reductive.   So $\calG_{A,\ell}^\circ = \calG_{A',\ell}$ is reductive which shows that Theorem~\ref{T:main new}(\ref{T:main new c}) holds for $A/K$.   By the assumed connected case of Theorem~\ref{T:main new}, we have $\rho_{A',\ell}(\Gal_L)'=\calG_{A',\ell}^\circ(\ZZ_\ell)'$.  Therefore, $\rho_{A,\ell}(\Gal_L)' = \calG_{A,\ell}^\circ(\ZZ_\ell)'$ and so Theorem~\ref{T:main new}(\ref{T:main new d}) holds for $A/K$.  

Now assume further that $\ell$ is unramified in $K_A^\conn$.   So $\ell$ is unramified in $K_{A'}^\conn=L=K_A^\conn$.    By the assumed connected case of Theorem~\ref{T:main new} and Proposition~\ref{P:connected}(\ref{P:connected iii}), we have
\[
[\calG_{A,\ell}(\ZZ_\ell):\rho_{A,\ell}(\Gal_K)] = [\calG_{A',\ell}(\ZZ_\ell):\rho_{A',\ell}(\Gal_{L})] \ll_{g} 1.
\] 
Therefore, Theorem~\ref{T:main new}(\ref{T:main new b}) holds for $A/K$.  
\end{proof}

\section{Some group theory}  \label{S:some group theory}

We now review group theoretic results of Nori, Larsen and Pink.  They will be needed for the proofs in the next section.

\subsection{Nori theory}  \label{SS:semisimple approximation}

Fix a positive integer $m$ and a prime $\ell \geq m$; we will later apply this theory with $m=2g$, where $g\geq 1$ is the dimension of our abelian variety.

Fix a subgroup $\Gamma$ of $\GL_m(\FF_\ell)$.   Let $\Gamma_u$ be the set of elements in $\Gamma$ of order $\ell$; it is also the set of non-trivial unipotent elements in $\Gamma$ since $\ell \geq m$.   Let $\Gamma^+$ be the subgroup of $\Gamma$ generated by $\Gamma_u$.   The group $\Gamma^+$ is normal in $\Gamma$ and the quotient $\Gamma/\Gamma^+$ has cardinality relatively prime to $\ell$.  

For each $x\in \Gamma_u$, define the homomorphism 
\[
\varphi_x \colon \GG_a \to \GL_{m},\quad t \mapsto \exp(t \cdot \log x)
\]
of algebraic groups over $\FF_\ell$ where we use the truncated series
\[
\exp z= {\sum}_{i=0}^{m-1}{z^i}/{i!} \quad \text{and} \quad  \log z = - {\sum}_{i=1}^{m-1} {(1-z)^i}/{i}.
\]
To show that $\varphi_x$ is a homomorphism, note that $\log x$ is nilpotent and that $\exp(y+z)=\exp(y)\exp(z)$ for commuting nilpotent $y,z\in M_m(\FF_\ell)$.  

Let $\boldG_\Gamma$ be the algebraic subgroup of $\GL_{m,\FF_\ell}$ generated by the groups $\varphi_x(\GG_a)$ with $x\in \Gamma_u$.    The following theorem says that, for $\ell$ sufficiently large, we can recover $\Gamma^+$ from $\boldG_\Gamma$.

\begin{thm}[Nori] \label{T:Nori}  \cite{MR880952}*{Theorem~B}
There  is a constant $c_1(m)\geq 2m-2$, depending only on $m$, such that if $\ell> c_1(m)$ and $\Gamma$ is a subgroup of $\GL_{m}(\FF_\ell)$, then $\Gamma^+=\boldG_\Gamma(\FF_\ell)^+$. 
\end{thm}

We now consider the case where $\Gamma$ is a semisimple subgroup of $\GL_{m}(\FF_\ell)$, i.e., the group $\Gamma$ acts semisimply on $\FF_\ell^{m}$ via the natural action.  The following lemma is shown during the proof of Corollary~B.4 in \cite{1008.3675}.  

\begin{lemma}  \label{L:Nori}  
Suppose that $\ell >  c_1(m)$ and that $\Gamma$ is a semisimple subgroup of $\GL_{m}(\FF_\ell)$.  Then $\Gamma^+ \subseteq \GL_{m}(\FF_\ell)$ is semisimple and $\boldG_\Gamma$ is a semisimple algebraic subgroup of $\GL_{m,\FF_\ell}$.
\end{lemma}

For a semisimple algebraic group $G$ defined over $\FF_\ell$, let $G^{\sc}\to G$ be its simply connected cover.

\begin{lemma} \label{L:finite lifting}
There is a finite collection $\{\varrho_i \colon G_i \to \GL_m\}_{i\in I}$ of $\ZZ$-representations of split simply connected Chevalley groups and a constant $c_2(m)\geq c_1(m)$ such that if $\ell > c_2(m)$ and $\Gamma \subseteq \GL_m(\FF_\ell)$ is a semisimple subgroup, then the representation
\[
(\boldG_{\Gamma}^{sc})_{\FFbar_\ell} \to (\boldG_{\Gamma})_{\FFbar_\ell} \hookrightarrow \GL_{m,\FFbar_\ell}
\]
is isomorphic to the fiber of $\varrho_i$ over $\FFbar_\ell$ for some $i\in I$.
\end{lemma}
\begin{proof}
This is Theorem~B.7 in \cite{1008.3675}; by Lemma~\ref{L:Nori}, there is no harm in replacing $\Gamma$ by $\Gamma^+$.
\end{proof}

Finally consider a subgroup $H$ of $\GL_m(\ZZ_\ell)$ that is closed in the $\ell$-adic topology.   Let $S$ be the Zariski closure of $H$ in $\GL_{m,\QQ_\ell}$ and let $S^\circ$ be the connected component of the identity.  We define the \defi{Nori dimension} of $H$, which we denote by $\Ndim(H)$, to be the dimension of $\boldG_\Gamma$, where $\Gamma$ is the image of $H$ in $\GL_m(\FF_\ell)$.  The following is a special case of a theorem of Larsen, cf.~\cite{MR2832632}*{Theorem~7}.  

\begin{thm}[Larsen]  \label{T:Larsen exponential}
There are constants $c_3(m)$ and $c_4(m)$, depending only $m$, such that the following hold if $\ell \geq c_3(m)$.
\begin{romanenum}
\item \label{T:Larsen exponential i}
We have $\Ndim(H) \leq \dim S$.
\item \label{T:Larsen exponential ii}
If $S^\circ$ is semisimple and $\Ndim(H) = \dim S$, then $[S(\QQ_\ell)\cap \GL_m(\ZZ_\ell) : H] \leq c_4(m)$.
\end{romanenum}
\end{thm}

\subsection{A theorem of Larsen and Pink}

A classic theorem of Jordan say that every finite subgroup $\Gamma$ of $\GL_m(\CC)$ has a normal abelian subgroup whose index can be bounded by some constant depending only on $m$.   Larsen and Pink have given a generalized version that holds for all finite subgroups of $\GL_m(k)$, where $k$ is an arbitrary field.  The following is  \cite{larsen-pink-finite_groups}*{Theorem~0.2} specialized to the subgroups of $\GL_{m}(\FF_\ell)$. 

\begin{thm}[Larsen-Pink]  \label{T:Larsen-Pink}
Let $\Gamma$ be a subgroup of $\GL_{m}(\FF_\ell)$.  Then there are normal subgroups $\Gamma_3\subseteq \Gamma_2\subseteq \Gamma_1$ of $\Gamma$ such that the following hold:
\begin{alphenum}
\item $[\Gamma:\Gamma_1]\leq J(m)$, where $J(m)$ is a constant depending only on $m$.
\item $\Gamma_1/\Gamma_2$ is a direct product of finite simple groups of Lie type in characteristic $\ell$.
\item $\Gamma_2/\Gamma_3$ is abelian and its order is relatively prime to $\ell$.
\item $\Gamma_3$ is an $\ell$-group.
\end{alphenum}
\end{thm}

\section{Proof of Theorem~\ref{T:main new}(\ref{T:main new c}) and (\ref{T:main new d})}
\label{S:proof of c and d}

Fix an abelian variety $A$ of dimension $g\geq 1$ defined over a number field $K$.  
For each prime $\ell$, we have a representation 
\[
\rho_{A,\ell}\colon \Gal_K \to \calG_{A,\ell}(\ZZ_\ell) \subseteq \GL_{V_\ell(A)}(\QQ_\ell).
\] 
By Lemma~\ref{L:reduction to connected case}, we may assume that all the $\ell$-adic monodromy groups $G_{A,\ell}$ are connected.    In particular, $\calG_{A,\ell}^\circ=\calG_{A,\ell}$.

We can identify the special fiber of the $\ZZ_\ell$-group scheme $\GL_{T_\ell(A)}$ with the $\FF_\ell$-group scheme $\GL_{A[\ell]}$.   Denote by
\[
\bbar\rho_{A,\ell}\colon \Gal_K \to \calG_{A,\ell}(\FF_\ell)= \GL_{A[\ell]}(\FF_\ell)
\] 
the representation obtain by composing $\rho_{A,\ell}$ with reduction modulo $\ell$; equivalently, it describes the natural Galois action on $A[\ell]$.    

We fix a non-zero prime ideal $\q\subseteq \OO_K$ for which $\Phi_{A,\q}$ is a free abelian group of rank $r$, where $r$ is the common rank of the reductive groups $G_{A,\ell}^\circ$.  Now consider any prime satisfying
\[
\ell\geq c \cdot \max(\{[K:\QQ],h(A), N(\q)\})^\gamma,
\]
where $c\geq 1$ and $\gamma \geq 1$ are positive constants that depend only on $g$.    Throughout this section, we will repeatedly increase the constants $c$ and $\gamma$ to make sure that certain condition hold while always maintaining that they depend only on $g$.  In particular, the statement of all lemmas and propositions may require increasing $c$ and $\gamma$.

\subsection{Proof of Theorem~\ref{T:main new}(\ref{T:main new c})} \label{SS:new reductive proof}

We will make use of the following criterion of Wintenberger.

\begin{lemma} \cite{MR1944805}*{Th\'er\`eome~1} \label{L:Wintenberger reductive criterion}
Let $L$ be a free $\ZZ_\ell$-module of rank at most $\ell$ and define the $\QQ_\ell$-vector space $V:=L\otimes_{\ZZ_\ell} \QQ_\ell$.   Let $G$ be a reductive group over $\QQ_\ell$, $G\hookrightarrow \GL_V$ a faithful representation and $T$ a maximal torus of $G$.  Let $\calT$ and $\calG$ be the schematic closure of $T$ and $G$, respectively, in $\GL_L$.     Suppose further that $\calT$ is a torus over $\ZZ_\ell$.    Then $\calG$ is a smooth group scheme over $\ZZ_\ell$.  If also $\calG_{\FF_\ell}$ acts semisimply on $L/\ell L$, then $\calG$ is a reductive group scheme over $\ZZ_\ell$.
\end{lemma}

We will apply Lemma~\ref{L:Wintenberger reductive criterion} with $L:=T_\ell(A)$, $V:=V_\ell(A)$ and $G:=G_{A,\ell} \subseteq \GL_{V_\ell(A)}$.  The group $G$ is connected by assumption and it is reductive by Proposition~\ref{P:conn facts}(\ref{P:conn facts i}).   Let $T$ be the Zariski closure in $G_{A,\ell}$ of the subgroup generated by the semisimple element $\rho_{A,\ell}(\Frob_\q)$.     
Observe that $T$ is a maximal torus of $G_{A,\ell}$ if and only if $\Phi_{A,\q}$  is a free abelian group whose rank equals the rank of $G_{A,\ell}$.    Therefore, $T$ is a maximal torus of $G_{A,\ell}$ by our choice of $\q$.

 Denote by $\calT$ the Zariski closure of $T$ in $\GL_{T_\ell(A)}$; it is a $\ZZ_\ell$-group scheme.  It is straightforward to verify that $\calT$ and $\calG:=\calG_{A,\ell}$ are also the scheme-theoretic closures of $T$ and $G_{A,\ell}$, respectively, in $\GL_{T_\ell(A)}/\ZZ_\ell$.  
 
 \begin{lemma}
The group $\calG_{\FF_\ell}$ acts semisimply on $L/\ell L =A[\ell]$. 
\end{lemma}
\begin{proof}   
Theorem~\ref{T:MW} implies (after appropriately increasing $c$ and $\gamma$) that $A[\ell]$ is a semisimple $\FF_\ell[\Gal_{K}]$-module and that the commutant $R_1$ of $\bbar\rho_{A,\ell}(\Gal_{K})$ in $\End_{\FF_\ell}(A[\ell])$ is naturally isomorphic to $\End(A) \otimes_\ZZ \FF_\ell$.  Let $R_2$ be the commutant of $\calG_{\FF_\ell}$ in $\End_{\FF_\ell}(A[\ell])$.

We have $\bbar\rho_{A,\ell}(\Gal_{K}) \subseteq \calG(\FF_\ell)$, so to prove that $\calG_{\FF_\ell}$ acts semisimply on $A[\ell]$, it suffices to show that $R_2=R_1$.     We have $R_2 \subseteq R_1 =\End(A) \otimes_\ZZ \FF_\ell$ since $\bbar\rho_{A,\ell}(\Gal_{K}) \subseteq \calG(\FF_\ell)$.   To prove the other inclusion, it thus suffices to show that $\calG_{\FF_\ell}$ and $\End(A) \otimes_\ZZ \FF_\ell$ commute.   This is immediate since $\calG$ is the Zariski closure of $\rho_{A,\ell}(\Gal_{K})$ and the actions of $\Gal_{K}$ and $\End(A)$ on $T_\ell(A)$ commute.
\end{proof}

By increasing $c$, we may assume that $\ell > 2g = \rank_{\ZZ_\ell} T_\ell(A)$.   By the criterion of Lemma~\ref{L:Wintenberger reductive criterion}, it thus suffices to show that $\calT$ is a torus.

Let $F$ be a splitting field of $P_{A,\q}(x)$ over $\QQ_\ell$ and denote by $\lambda_1,\ldots,\lambda_m \in F$ the distinct roots of $P_{A,\q}(x)$.  Let $R$ be the valuation ring of the local field $F$ and denote its residue field by $\FF$.   Define $D:=\prod_{1\leq i < j \leq m} (\lambda_i-\lambda_j)^2\neq 0$; it is an integer since $P_{A,\q}(x)$ has coefficients in $\ZZ$. 
Since the roots of $P_{A,\q}(x)$ have absolute value $N(\q)^{1/2}$ under any embedding $F\hookrightarrow \CC$, we find that $|D|< c N(\q)^{\gamma}$ after increasing $c$ and $\gamma$ appropriately.   If $\q$ divides $\ell$, then $\ell \leq N(\q)$.      So by increasing $c$ and $\gamma$ appropriately, we may assume that $\q\nmid \ell$ and $D\not\equiv 0 \pmod{\ell}$.   This implies that $F/\QQ_\ell$ is an unramified  extension.

Since $\Spec R \to \Spec \ZZ_\ell$ is faithfully flat, it suffices to prove that $\calT_R$ is a torus.  To show this we will use the following lemma.

\begin{lemma} \label{L:diagonalize}
Take any matrix $B \in \GL_{2g}(R)$ that is semisimple in $\GL_{2g,F}$ and has characteristic polynomial $P_{A,\q}(x)$. Then the Zariski closure in $\GL_{2g,R}$ of the subgroup generated by $B$ is a split torus over $R$.
\end{lemma}
\begin{proof}
  Take any $1\leq i \leq m$.   Let $L_i$ be the $R$-submodule of $R^{2g}$ consisting of those $v\in R^{2g}$ that satisfy $Bv=\lambda_i v$.   Since $B$ is semisimple, it acts on $L_i$ as multiplication by $\lambda_i$.       Let $d_i$ be the rank of the $R$-module $L_i$.  By the assumption of the lemma, $B$ is diagonalizable in $\GL_{2g}(F)$ and hence $F^{2g}=\oplus_{i=1}^m L_i\otimes_R F$.  In particular, we have $2g=\sum_{i=1}^m d_i$ by taking dimensions.
  
Let $\varphi\colon R^{2g} \to \FF^{2g}$ be the reduction map.    For each $1\leq i \leq m$, define the $\FF$-vector space $V_i:= \varphi(L_i)$.

We claim that $\dim_\FF V_i = d_i$.  By the structure theorem for finitely generated modules over a PID, there is a basis $e_1,\ldots, e_{2g}$ of the $R$-module $R^{2g}$ such that
\[
L_i = R\pi^{a_1} \cdot e_1 \oplus \cdots \oplus R\pi^{a_{d_i}}\cdot e_{d_i}
\]
with integers $a_j \geq 0$, where $\pi$ is a uniformizer of $R$.   Note that if $\pi^a b$ is in $L_i$ for some $b\in R^{2g}$ and $a\geq 0$, then we have $b\in L_i$.    Therefore, $a_1=\cdots = a_{d_i}=0$.  The claim follows since we find that $V_i=\varphi(L_i)$ is a vector space over $\FF$ with basis $\varphi(e_1),\ldots, \varphi(e_{d_i})$.  

Let $\bbar{B} \in \GL_{2g}(\FF)$ be the reduction of $B$.  The matrix $\bbar{B}$ acts on $V_i$ as scalar multiplication by $\bbar{\lambda}_i$, where $\bbar{\lambda}_i$ is the image of $\lambda_i$ in $\FF$ (we have $\lambda_i \in R$ since it is a root of the monic polynomial $P_{A,\q}(x)\in \ZZ[x]$).   For distinct $1\leq i < j \leq m$, we have $\bbar{\lambda}_i\neq \bbar{\lambda}_j$ since otherwise the image of $D=\prod_{1\leq i < j \leq m} (\lambda_i-\lambda_j)^2$ in $\FF$ is $0$ which contradicts that $\ell\nmid D$.   Since $\bbar{\lambda}_1,\ldots, \bbar{\lambda}_m$ are distinct eigenvalues of $\bbar{B}$ and $\sum_{i=1}^m \dim_\FF V_i = \sum_{i=1}^m d_i = 2g$, we deduce that $\oplus_{i=1}^m V_i =\FF^{2g}$.    Therefore, $\varphi(\oplus_{i=1}^m L_i) = \FF^{2g}$.   By Nakayama's lemma, we have $\oplus_{i=1}^m L_i =R^{2g}$.  Therefore, $B$ is conjugate in $\GL_2(R)$ to a diagonal matrix.   

So without loss of generality, we may assume that $B$ is a diagonal matrix.  In particular, we can view $B$ as an $R$-point of the diagonal subgroup $\GG_{m,R}^{2g}$ of $\GL_{2g,R}$.   Let $\calT$ be the Zariski closure in $\GG_{m,R}^{2g}$ of the subgroup generated by $B$.    The fiber $\calT_F$ is a torus of rank $r$ since the eigenvalues in $F^\times$ of $B$ generate a free abelian group of rank $r$ (this uses that the characteristic polynomial of $B$ is $P_{A,\q}(x)$ and our choice of $\q$).     There is thus a set of equations of the form $\prod_{i=1}^{2g} x_i^{a_{i}}=1$, with integers $a_i$, that cut out $\calT_F$ in $\GG_{m,F}^{2g}$.   These equations thus define $\calT$ and hence it is a subtorus of $\GG_{m,R}^{2g}$.
\end{proof}

We now complete the proof by verifying that $\calT_R$ is a torus.   Set $B:=\rho_{A,\ell}(\Frob_\q)$; it semisimple and has characteristic polynomial $P_{A,\q}(x)$.  From our choice of $\q$, the Zariski closure in $\GL_{V_\ell(A)}$ of the group generated by $B$ is a maximal torus of $G_{A,\ell}$.   The group $\calT_R$ is the Zariski closure in $(\GL_{T_\ell(A)})_ R=\GL_{T_\ell(A)\otimes_{\ZZ_\ell} R}$ of the group generated by $B$.   By choosing an $R$-basis of $T_\ell(A)\otimes_{\ZZ_\ell} R$, Lemma~\ref{L:diagonalize} implies that $\calT_R$ is a torus.

\subsection{The reductive group \texorpdfstring{$H_{A,\ell}$}{} of Serre}
Recall that $\calG_{A,\ell}$ is reductive by \S\ref{SS:new reductive proof}.  Let  $\calC_{A,\ell}$ be the central torus $\calG_{A,\ell}$  and denote by $C_{A,\ell} \subseteq \GL_{A[\ell]}$ the torus obtained by base changing $\calC_{A,\ell}$ to $\FF_\ell$.  Since $\calC_{A,\ell}$ commutes with $\rho_{A,\ell}(\Gal_K) \subseteq \calG_{A,\ell}(\ZZ_\ell)$, we find that $C_{A,\ell}$ commutes with $\bbar\rho_{A,\ell}(\Gal_K)$.   
Let $\calS_{A,\ell}$ be the derived subgroup of $\calG_{A,\ell}$; it is semisimple group scheme over $\ZZ_\ell$.

We now consider the algebraic subgroup $\boldG_\Gamma \subseteq \GL_{A[\ell]}$ constructed as in  \S\ref{SS:semisimple approximation} with $\Gamma:=\bbar\rho_{A,\ell}(\Gal_K)$.   Note that after possibly increasing the constants $c$ and $\gamma$, Theorem~\ref{T:MW} implies that $A[\ell]$ is a semisimple $\Gamma$-module.   

\begin{lemma} \label{L:S is semisimple}
The group $\boldG_\Gamma$ is semisimple and commutes with $C_{A,\ell}$.  
\end{lemma}
\begin{proof}
After possibly increasing $c$, Lemma~\ref{L:Nori} implies that $\boldG_\Gamma$ is a semisimple subgroup of $\GL_{A[\ell]}$.   That $\boldG_\Gamma$ commutes with $C_{A,\ell}$ is an easy consequence of $\Gamma$ commuting with $C_{A,\ell}$.
\end{proof}

Define the algebraic subgroup 
\[
H_{A,\ell} := C_{A,\ell} \cdot \boldG_\Gamma
\] 
of $\GL_{A[\ell]}$.  The group $H_{A,\ell}$, for $\ell$ sufficiently large, was first defined by Serre \cite{MR1730973}*{136}.

The group $H_{A,\ell}$ is reductive by Lemma~\ref{L:S is semisimple}.  From the image of the representation $\rho_{A,\ell}$, we have constructed two reductive subgroups $H_{A,\ell}$ and $(\calG_{A,\ell})_{\FF_\ell}$ of $\GL_{A[\ell]}$.   We now prove that they agree; this is the main result of this section.

\begin{thm} \label{T:equal groups}
We have $H_{A,\ell}=(\calG_{A,\ell})_{\FF_\ell}$.  In particular, we have $\boldG_\Gamma=(\calS_{A,\ell})_{\FF_\ell}$.
\end{thm}

We will use the following criterion to verify that the reductive groups $H_{A,\ell}$ and $(\calG_{A,\ell})_{\FF_\ell}$ are equal.   

\begin{lemma}  \cite{MR1944805}*{Lemma~7} \label{L:reductive inclusion}
Let $F$ be a perfect field whose characteristic is $0$ or at least $5$.   Let $G_1 \subseteq G_2$ be reductive groups defined over $F$ that have the same rank.   Suppose we have a faithful linear representation $G_2 \hookrightarrow \GL_V$, where $V$ is a finite dimension $F$-vector space, such that the centers of the commutants of $G_1$ and $G_2$ in $\End_F(V)$ are the same $F$-algebra $R$.  Suppose further that the commutative $F$-algebra $R$ is semisimple.   Then $G_1=G_2$. 
 \qed
\end{lemma}

We first show that one of our reductive groups is a subgroup of the other.

\begin{lemma} \label{L:rank inequality}
There is an inclusion $H_{A,\ell} \subseteq (\calG_{A,\ell})_{\FF_\ell}$.
\end{lemma}
\begin{proof}
The reductive groups $H_{A,\ell}$ and $(\calG_{A,\ell})_{\FF_\ell}$ both have the same central torus $C_{A,\ell}$, so it suffices to prove that $\boldG_\Gamma\subseteq (\calS_{A,\ell})_{\FF_\ell}$.   After possibly increasing $c$, we may assume that $\ell \geq 2g$.  Take any $x\in \Gamma$ of order $\ell$ and let $\varphi_x \colon \GG_a \to \GL_{A[\ell]}$ be the homomorphism $t\mapsto \exp(t\cdot \log x)$ as in \S\ref{SS:semisimple approximation}.     Using that $\calS_{A,\ell}$ is semisimple and by possibly increasing $c$, Lemma~6 of \cite{MR1944805} implies that the image of $\varphi_x$ is contained in $(\calS_{A,\ell})_{\FF_\ell}$.    Since $x$ is an arbitrary element of $\Gamma$ of order $\ell$, we deduce that $\boldG_{\Gamma}$ is contained in $(\calS_{A,\ell})_{\FF_\ell}$.    
\end{proof}

The following lemma, which we will prove in \S\ref{SS:hard inclusion}, shows that $H_{A,\ell}(\FF_\ell)$ contains a large subgroup of $\bbar\rho_{A,\ell}(\Gal_K)$.

\begin{lemma} \label{L:hard inclusion}
There is a constant $b \geq 1$, depending only on $g$, such that $\bbar\rho_{A,\ell}(\Gal_{L}) \subseteq H_{A,\ell}(\FF_\ell)$ for some extension $L/K$ with $[L:K]\leq b$.
\end{lemma}

The following lemma, which we will prove in \S\ref{SS:same rank proof}, shows that our groups have the same rank.  

\begin{lemma} \label{L:same rank}
The reductive groups $H_{A,\ell}$ and $(\calG_{A,\ell})_{\FF_\ell}$ have the same rank.
\end{lemma}

\begin{lemma} \label{L:same commutant}
The commutants of $H_{A,\ell}$ and $(\calG_{A,\ell})_{\FF_\ell}$ in $\End_{\FF_\ell}(A[\ell])$ agree and their common center is a semisimple $\FF_\ell$-algebra.
\end{lemma}
\begin{proof}
Let $L$ be the extension of $K$ from Lemma~\ref{L:hard inclusion}.   After possibly increasing $c$, Theorem~\ref{T:MW} and $[L:K]\ll_g 1$ implies that the commutant $R$ of $\bbar\rho_{A,\ell}(\Gal_L)$ in $\End(A[\ell])$ is naturally isomorphic to $\End(A_L)\otimes_\ZZ \FF_\ell$ and that $R$ is a semisimple $\FF_\ell$-algebra.   In fact, we have a natural isomorphism $R=\End(A)\otimes_\ZZ \FF_\ell$ since $\End(A)=\End(A_{\Kbar})$ by Proposition~\ref{P:conn facts} and our  assumption that $G_{A,\ell}$ is connected.  Since $R$ is semisimple, the center of $R$ is semisimple.

Let $R_1$ and $R_2$ be the commutants of $H_{A,\ell}$ and $(\calG_{A,\ell})_{\FF_\ell}$, respectively, in $\End_{\FF_\ell}(A[\ell])$.   We have inclusion $R_2 \subseteq R_1 \subseteq R$ since $\bbar\rho_{A,\ell}(\Gal_L) \subseteq  H_{A,\ell}(\FF_\ell)$ by our choice of $L$ and since $H_{A,\ell} \subseteq (\calG_{A,\ell})_{\FF_\ell}$ by Lemma~\ref{L:rank inequality}.  So it suffices to show that $R=\End(A)\otimes_\ZZ \FF_\ell$ is a subring of $R_2$. Equivalently, it suffices to show that $(\calG_{A,\ell})_{\FF_\ell}$ commutes with $\End(A)\otimes_\ZZ \FF_\ell$; this is immediate since the actions of $\Gal_K$ and $\End(A)$ on $T_\ell(A)$ commute. 
\end{proof}

\begin{proof}[Proof of Theorem~\ref{T:equal groups}]
We may assume that $\ell\geq 5$ after possibly increasing $c$.   Set $G_1:=H_{A,\ell}$ and $G_2:=(\calG_{A,\ell})_{\FF_\ell}$.   Using Lemma~\ref{L:reductive inclusion} with Lemmas~\ref{L:rank inequality}, \ref{L:same rank} and \ref{L:same commutant}, we deduce that $G_1=G_2$.
\end{proof}

\subsection{Proof of Lemma~\ref{L:hard inclusion}}  \label{SS:hard inclusion}

Let $\Gamma_3\subseteq \Gamma_2\subseteq \Gamma_1$ be the normal subgroups of $\Gamma:=\bbar\rho_{A,\ell}(\Gal_K)$ as in Theorem~\ref{T:Larsen-Pink} with $m=2g$.    By increasing $c$, we may assume that $\ell >\max\{J(2g),2g,7\}$.  

Since $\Gamma=\bbar\rho_{A,\ell}(\Gal_K)$ acts semisimply on the $\FF_\ell$-vector space $A[\ell]$ and $\Gamma_3$ is a normal subgroup of $\Gamma$, we deduce by Clifford's theorem \cite{MR632548}*{11.1} that $\Gamma_3$ also acts semisimply on $A[\ell]$.  Since $\ell>2g$ and $\Gamma_3$ is an $\ell$-group, the action of $\Gamma_3$ on $A[\ell]$ is unipotent.  Since the action of $\Gamma_3$ on $A[\ell]$ is unipotent and semisimple, we must have $\Gamma_3=1$.   In particular, $\Gamma_2$ is an abelian group with cardinality relatively prime to $\ell$.

\begin{lemma} \label{L:gen bound of 2g}
The group $\Gamma_2$ has a set of generators with cardinality at most $2g$. 
\end{lemma}  
\begin{proof}
Since $\Gamma_2$ is abelian and has order relatively prime to $\ell$, it must lie in a maximal torus $\T$ of $\GL_{A[\ell]} \cong \GL_{2g,\FF_\ell}$.  Let $\FF$ be a finite extension of $\FF_\ell$ over which $\T$ is split.  So $\Gamma_2$ is a subgroup of $\T(\FF) \cong (\FF^\times)^{2g}$.    Since $\FF^\times$ is cyclic, $\Gamma_2$ lies in a finite abelian group generated by $2g$ elements.  The lemma is now easy from the structure theorem for finite abelian groups.
\end{proof}

\begin{lemma}  \label{L:center of Gamma1}
The center of $\Gamma_1$ is $\Gamma_2$.
\end{lemma}
\begin{proof}
Since $\Gamma_1/\Gamma_2$ is a product of non-abelian simple groups, we find that the cardinality of the center of $\Gamma_1$ divides $|\Gamma_2|$.  It thus suffices to show that $\Gamma_2$ lies in the center of $\Gamma_1$; equivalently, that the homomorphism $\varphi\colon \Gamma_1/\Gamma_2 \to \Aut(\Gamma_2)$ arising from conjugation is trivial.  

Suppose that $\varphi\neq 1$.     Then for some prime $p$ dividing $|\Gamma_2|$, there is a finite simple group $S$ of Lie type in characteristic $\ell$ that acts non-trivially on the $p$-Sylow subgroup $W$ of $\Gamma_2$.    We view $W$ as an additive group.   Define $\calW:= p^i W$, where $i\geq 0$ is the largest integer such that $S$ acts nontrivially on $\calW$ and trivially on $p\calW$.

Suppose that $S$ acts trivially on the quotient group $\calW/p\calW$.    Take any $w\in \calW$.    We thus have a well-defined map $\xi_w\colon S \to p\calW$, $h \mapsto h(w)-w$.   Using that $S$ acts trivially on $p\calW$, we find that 
\[
\xi_w(h_1 h_2) = h_1h_2(w) - w= h_1(w+\xi_w(h_2)) - w= h_1(w)-w + \xi_w(h_2)= \xi_w(h_1) + \xi_w(h_2)
\]
for all $h_1,h_2\in S$.   Therefore, $\xi_w\colon S \to p\calW$ is a homomorphism.    Since the group $S$ is nonabelian and simple and $p\calW$ is abelian, we must have $\xi_w = 0$.   Since $w\in \calW$ was arbitrary, we deduce that $S$ acts trivially on $\calW$ which contradicts our choice of $\calW$.     

Therefore, $S$ acts non-trivially on the $\FF_p$-vector space $\calW/p \calW$.  From Lemma~\ref{L:gen bound of 2g}, the dimension of $\calW/p\calW$ as an $\FF_p$-vector space is at most $2g$.   Since $S$ is simple, we deduce that $S$ is isomorphic to a subgroup of $\GL_{2g}(\FF_p)$.  
   
Since $\ell$ divides $|S|$ and $\ell > J(2g)$, Theorem~\ref{T:Larsen-Pink} (with $m=2g$ and $\ell$ replaced by $p$) implies that $S$ is a finite simple group of Lie type in characteristic $p$.    However, there is no finite simple group $S$ that is of Lie type in two distinct characteristics $p$ and $\ell>7$.   (We need $\ell>7$ to avoid the exceptional isomorphism $\PSL_2(\FF_7)\cong \PSL_3(\FF_2)$.)   This contradiction ensures that $\varphi=1$.
\end{proof}

Let $\psi\colon \Gal_{K} \to \bbar\rho_{A,\ell}(\Gal_K)/\Gamma_1$ be the homomorphism obtained by composing $\bar\rho_{A,\ell}$ with the obvious quotient map.   We define $L$ to be the fixed field in $\Kbar$ of $\ker(\psi)$; we have $\bbar\rho_{A,\ell}(\Gal_L)=\Gamma_1$.  The field $L$ is a Galois extension of $K$ which satisfies $[L:K]\leq J(2g)$.  Since $\ell >J(2g)$, we have $\bbar\rho_{A,\ell}(\Gal_L)^+=\Gamma^+$. 

\begin{lemma} \label{L:Gamma gen}
The group $\Gamma_1$ is generated by $\Gamma_2$ and $\Gamma^+$.
\end{lemma}
\begin{proof}
The group $\Gamma^+=\bbar\rho_{A,\ell}(\Gal_L)^+$ is a contained in $\Gamma_1=\bbar\rho_{A,\ell}(\Gal_L)$.   The homomorphism $\Gamma^+ \to \Gamma_1/\Gamma_2$ is surjective since $\ell\nmid|\Gamma_2|$ and since any finite simple group of Lie type in characteristic $\ell$ is generated by its elements of order $\ell$.   Therefore, $\Gamma_1$ is generated by $\Gamma^+$ and $\Gamma_2$.   
\end{proof}

Let $Z$ be the center of $(\calG_{A,\ell})_{\FF_\ell}$; we have $Z^\circ=C_{A,\ell}$.    

\begin{lemma} \label{L:in Z}
The group $\Gamma_2$ is contained in $Z(\FF_\ell)$.
\end{lemma}
\begin{proof}
After possibly increasing $c$ and $\gamma$, Theorem~\ref{T:MW} implies that the centralizer of $\Gamma_1=\bbar\rho_{A,\ell}(\Gal_L)$ in $\End_{\FF_\ell}(A[\ell])$ is $\End(A_L)\otimes_\ZZ \FF_\ell = \End(A)\otimes_\ZZ \FF_\ell$ (recall that $[L:k]\leq J(2g)$).   The group $\Gamma_2$ is contained in the center of $\Gamma_1$ by Lemma~\ref{L:center of Gamma1}, so we can identify $\Gamma_2$ with a subgroup of $(\End(A)\otimes_\ZZ \FF_\ell)^\times$.    The group $(\End(A)\otimes_\ZZ \FF_\ell)^\times$, and hence also $\Gamma_2$, commutes with $(\calG_{A,\ell})_{\FF_\ell}$.  Since $\Gamma_2 \subseteq \bbar\rho_{A,\ell}(\Gal_k) \subseteq \calG_{A,\ell}(\FF_\ell)$, we deduce that $\Gamma_2$ lies in the center of $(\calG_{A,\ell})_{\FF_\ell}$.
\end{proof}

Let $\widetilde{H}$ be the algebraic subgroup of $\GL_{A[\ell]}$ generated by $Z$ and $\boldG_\Gamma$; note that $Z$ and $\boldG_\Gamma$ commute since $\bbar\rho_{A,\ell}(\Gal_k)$ commutes with $Z$.   Observe that the neutral component of $\widetilde{H}$ is precisely our reductive group $H_{A,\ell}$. By Lemma~\ref{L:in Z},  we have $\Gamma_2 \subseteq \widetilde{H}(\FF_\ell)$.    Since $\boldG_\Gamma(\FF_\ell)^+=\Gamma^+$ by Theorem~\ref{T:Nori}, after possibly increasing $c$, we find that $\Gamma^+$ is contained $\widetilde{H}(\FF_\ell)$.   Therefore, $\Gamma_1$ is a subgroup of $\widetilde{H}(\FF_\ell)$ by Lemma~\ref{L:Gamma gen}.

\begin{lemma} \label{L:center bound}
The index of $H_{A,\ell}$ in $\widetilde{H}$ can be bounded by a constant depending only on $g$.
\end{lemma}
\begin{proof}
Equivalently, we need to bound the index of $C_{A,\ell}$ in $Z$.  It suffices to bound the cardinality of the center of the semisimple group $G:=(\calG_{A,\ell})_{\FFbar_\ell}/(C_{A,\ell})_{\FFbar_\ell}$.   This is clear since $G$ is a semisimple group whose rank is bounded in terms of $g$; the cardinality of the center can be bounded in terms of the root datum of $G$ and there are only finitely many root datum for semisimple groups of bounded rank.  (Moreover, the cardinality of the center of a semisimple group of rank $r$ is bounded above by $2^r$.)
\end{proof}  

We have shown that $\Gamma_1=\bbar\rho_{A,\ell}(\Gal_L)$ is a subgroup of $\widetilde{H}(\FF_\ell)$.  Let $L'$ be the smallest extension of $L$ for which $\bbar\rho_{A,\ell}(\Gal_{L'})\subseteq H_{A,\ell}(\FF_\ell)$.  By Lemma~\ref{L:center bound}, the degree $[L':L]$ can be bounded in terms of $g$.   We have already seen that $[L:K]$ can be bounded in terms of $g$.  Therefore, we have $\bbar\rho_{A,\ell}(\Gal_{L'})\subseteq H_{A,\ell}(\FF_\ell)$ with $[L':K] \ll_g 1$.

\subsection{Complexity\texorpdfstring{ of $H_{A,\ell}$}{}}

Define the $\ZZ$-torus $D:=\GG_m^{2g}$; we will identify it with the diagonal torus of $\GL_{2g,\ZZ}$.      Let $X(D)$ be the group of characters $D\to \GG_m$.  We have an isomorphism
\[
\ZZ^{2g} \xrightarrow{\sim} X(D),\quad m \mapsto \alpha_m,
\]
where $\alpha_m$ is the character given by $(x_1,\ldots, x_{2g}) \mapsto \prod_i x_i^{m_i}$.  \\

Consider a field $F$ and a connected and reductive subgroup $G$ of $\GL_{2g,F}$.   Choose an algebraically closed extension $F'/F$ and a torus $T$ in the diagonal torus $D_{F'}\subseteq \GL_{2g,F'}$ that is conjugate to a maximal torus of $G_{F'}$ by some element of $\GL_{2g}(F')$.  
We define the \defi{complexity} of $G$ to be the smallest integer $\scrF(G)\geq 1$
satisfying
\[
T=\bigcap_{m\in M} \ker \alpha_m
\]
in $D_{F'}$ for some subset $M\subseteq \ZZ^{2g}$ with $|m_i|\leq \scrF(G)$ for all $m\in M$ and $1\leq i \leq 2g$.   We can identify the character group of $T$ with $\ZZ^{2g}/\ZZ M$.  Note that $\scrF(G)$ does not depend on the choice of $F'$ or $T$.    

The goal of this section is to show that $\scrF(H_{A,\ell})\ll_g 1$.
 
\begin{lemma} \label{L:complexity G}
We have $\scrF(G_{A,\ell}) \ll_g 1$.
\end{lemma}
\begin{proof}
Let $T$ be the Zariski closure in $G_{A,\ell}$ of the group generated by $\rho_{A,\ell}(\Frob_\q)$.  As observed in \S\ref{SS:new reductive proof}, $T$ is a maximal torus of $G_{A,\ell}$.  Take $t\in D(\Qbar_\ell)$ that is conjugate to $\rho_{A,\ell}(\Frob_\q)$.   We have $t=(\pi_1,\ldots, \pi_{2g})$, where the $\pi_i$ are the roots of $P_{A,\q}(x)$ in $\Qbar_\ell$ with multiplicity.   
Let $N\subseteq \ZZ^{2g}$ be the group of $m\in \ZZ^{2g}$ for which $\prod_{i=1}^{2g} \pi_i^{m_i}=1$.  Choose a finite set $M$ of generators for $N$ with $B:=\max\{|m_i|: m\in M, 1\leq i \leq 2g\}$ minimal.  We have $\scrF(G_{A,\ell})=\scrF(T)\leq B$.   Proposition~\ref{P:Frob group possibilities} says that there are only finitely many possibilities for the group $N$ in terms of $g$.   We can thus bound $B$ in terms of $g$.  
\end{proof}

\begin{lemma} \label{L:complexity C}
We have $\scrF(C_{A,\ell}) \ll_g 1$.
\end{lemma}
\begin{proof}
Define $C:=(\calC_{A,\ell})_{\QQ_\ell}$; it is the central torus of $G_{A,\ell}$.  Since $\calC_{A,\ell}$ is a torus with special fiber $C_{A,\ell}$, we have $\scrF(C_{A,\ell})=\scrF(C)$.  So it suffices to prove $\scrF(C)\ll_g 1$.

The center of $G_{A,\ell}$ is the intersection of all of its maximal tori.   Since $G_{A,\ell}$ has rank $r$, there are maximal tori $T_0,\ldots, T_r$ of $G$ such that the neutral component of $Z:=T_0\cap \cdots \cap T_r$ is $C$.   

Using that $\scrF(T_i)=\scrF(G_{A,\ell})\ll_g 1$ by Lemma~\ref{L:complexity G}, we find that there an integer $B\geq 1$ depending only on $g$ and subset $M\subseteq \ZZ^{2g}$ such that
\[
Z= \bigcap_{m\in M} \ker \alpha_m
\]
and $|m_i|\leq B$ for all $m\in M$ and $1\leq i \leq 2g$.    We thus have $C=\cap_{m\in M'} \ker \alpha_m$, where $M'$ is any finite set that generates the smallest group $\ZZ M \subseteq N \subseteq \ZZ^{2g}$ for which $\ZZ^{2g}/N$ is torsion-free.     We can choose $M'$ so that $B':=\max\{|m_i|: m \in M',\, 1\leq i \leq 2g\}$ is minimal.  Since there are only finitely many possible $M$ for a given $g$, we find that $B'\ll_g 1$.  Therefore, $\scrF(C)\leq B'\ll_g 1$.
\end{proof}

\begin{lemma} \label{L:bounded formal character ss}
We have $\scrF(\boldG_\Gamma)\ll_g 1$.
\end{lemma}
\begin{proof}
After possibly increasing $c$, we may assume that $\ell > c_2(2g)$, with $c_2(2g)$ as in Lemma~\ref{L:finite lifting}.    Let $\{\varrho_i \colon G_i \to \GL_{2g}\}_{i\in I}$ be the representations of Lemma~\ref{L:finite lifting} with $m=2g$; they depend only on $g$.    

Fix a split maximal torus $T_i$ of $G_i$ and define $H_i := \varrho_i(T_i)_{\FFbar_\ell}$.     We have $\scrF(H_i)=\scrF(\varrho_i(G_i)_{\Qbar})=:f_i$, which does not depend on $\ell$.  Lemma~\ref{L:finite lifting} implies that any maximal torus $T\subseteq (\boldG_\Gamma)_{\FFbar_\ell}$ is conjugate in $\GL_{2g,\FFbar_\ell}$ to $H_i$ for some $i\in I$.  Therefore,
\[
\scrF(\boldG_\Gamma)=\scrF(T)=\scrF(H_i)=f_i.
\]   
Since $I$ is finite, we have $\scrF(\boldG_\Gamma) \leq \max \{ f_i : i \in I \} \ll_g 1$.
\end{proof}

Combining the previous two lemmas, we deduce the following.

\begin{lemma} \label{L:new bounded formal character}
We have $\scrF(H_{A,\ell}) \leq B$, where $B$ is a positive integer depending only on $g$.
\end{lemma}
\begin{proof}
Let $T$ be a subtorus of $D_{\FFbar_\ell}$ that is conjugate in $\GL_{2g,\FFbar_\ell}\cong \GL_{A[\ell],\FFbar_\ell}$ to a maximal torus of $(H_{A,\ell})_{\FFbar_\ell}$.     We have $T=T_1\cdot T_2$, where $T_1$ and $T_2$ are conjugate to $(C_{A,\ell})_{\FFbar_\ell}$ and to a maximal torus of $(\boldG_\Gamma)_{\FFbar_\ell}$, respectively.    By Lemmas~\ref{L:complexity C} and \ref{L:bounded formal character ss}, we have $\calF(T_1)\ll_g 1$ and $\calF(T_2)\ll_g 1$.  

So there are subsets $M_1$ and $M_2$, with $\max\{|m_i| : m \in M_1\cup M_2, 1\leq i \leq 2g\} \ll_g 1$ such that $T_i = \cap_{m \in M_i } \ker \alpha_m$ for $1\leq i  \leq 2$.    We then have $T= \cap_{m \in M} \ker \alpha_m$, where $M$ is a finite set of generators of the subgroup $\ZZ M_1 \cap \ZZ M_2$ of $\ZZ^{2g}$.   We can choose $M$ so that $B:=\max\{|m_i|: m \in M,\, 1\leq i \leq 2g\}$ is minimal.  Since there are only finitely many possible $M_1$ and $M_2$ for a given $g$, we find that $B\ll_g 1$.  Therefore, $\scrF(H_{A,\ell})=\scrF(T)\leq B\ll_g 1$.
\end{proof}

\subsection{Proof of Lemma~\ref{L:same rank}}
\label{SS:same rank proof}

Take $B$ as in Lemma~\ref{L:new bounded formal character}.  With our fixed prime $\q$, let $\pi_1,\ldots, \pi_{2g} \in \Qbar$ be the roots of $P_{A,\q}(x)$ with multiplicity and define the number field $F:=\QQ(\pi_1,\ldots, \pi_{2g})$.  Let $\calM$ be the (finite) set of subsets $M\subseteq \ZZ^{2g}$ such that
\[
\max \{ |m_i| : m \in M, \, 1\leq i \leq 2g\} \leq B
\]
and such that $\rank_\ZZ (\ZZ M)>2g-r$.   For each $M\in \calM$, define
\[
\beta_M:=\sum_{m\in M} N_{F/\QQ}\Big( \prod_{1\leq i \leq 2g,\, m_i>0}\, \pi_{i}^{m_i} -  \prod_{1\leq i \leq 2g,\, m_i<0}\, \pi_{i}^{-m_i}\Big)^2;
\]
it is an integer since all the $\pi_i$ are algebraic integers.     

\begin{lemma} \label{L:betaM congruences}
If $\rank (H_{A,\ell}) \neq r$, then $\beta_M \equiv 0 \pmod{\ell}$ for some $M\in \calM$.
\end{lemma}
\begin{proof}
Suppose that $\rank (H_{A,\ell}) \neq r$.   From Lemma~\ref{L:rank inequality}, we have an inclusion  $H_{A,\ell} \subseteq (\calG_{A,\ell})_{\FF_\ell}$ in $\GL_{A[\ell]}$ of reductive groups.    Therefore, 
\[
\rank(H_{A,\ell}) \leq \rank((\calG_{A,\ell})_{\FF_\ell})=\rank(G_{A,\ell})=r.
\]
Our assumption implies that $\rank (H_{A,\ell}) < r$.

Let $T$ be a subtorus of $D_{\FFbar_\ell}$ that is conjugate in $\GL_{2g,\FFbar_\ell}\cong \GL_{A[\ell],\FFbar_\ell}$ to a maximal torus of $(H_{A,\ell})_{\FFbar_\ell}$.  By Lemma~\ref{L:new bounded formal character}, there is a set $M\subseteq \ZZ^{2g}$ such that $T=\cap_{m\in M} \ker \alpha_m$ and such that $|m_i|\leq B$ for all $m\in M$ and $1\leq i \leq 2g$.  We have $X(T)\cong \ZZ^{2g}/M$ and hence $\dim T = 2g - \rank_{\ZZ}(\ZZ M)$.  Therefore,
 \[
 \rank_\ZZ(\ZZ M) = 2g-\dim T=2g-\rank (H_{A,\ell}) >2g-r,
 \]  
and hence $M\in \calM$.

After possibly increasing $c$ and $\gamma$, we may assume that $\ell > N(\q)$ and hence $\q\nmid \ell$.   So $\bbar\rho_{A,\ell}$ is unramified at $\q$ and $\bbar\rho_{A,\ell}(\Frob_\q)$ has characteristic polynomial $P_{A,\q}(x)$ modulo $\ell$.   The semisimple component of $\bbar\rho_{A,\ell}(\Frob_\q)$ is conjugate in $\GL_{2g}(\FFbar_\ell)$ to an element $t\in T(\FFbar_\ell)$.  Let $b_1,\ldots, b_{2g} \in \FFbar_\ell$ be the diagonal entries of $t$.  Let $\lambda$ be a prime ideal of $\OO_F$ dividing $\ell$ and choose an embedding $\FF_\lambda \hookrightarrow \FFbar_\ell$.    After first possibly conjugating $T$ and $t$ by some element of $\GL_{2g}(\FFbar_\ell)$, we may assume that the image of $\pi_i$ modulo $\lambda$ is $b_i$.
  
Take any $m\in M$.  We have $t\in T(\FFbar_\ell)$, so $\prod_{i} b_i^{m_i}=1$ and hence
\[
\prod_{i,\, m_i>0} b_i^{m_i} - \prod_{i,\, m_i<0} b_i^{-m_i} = 0.
\]
Therefore,  
\[
\prod_{1\leq i \leq 2g,\, m_i>0}\, \pi_{i}^{m_i} -  \prod_{1\leq i \leq 2g,\, m_i<0}\, \pi_{i}^{-m_i} \equiv 0 \pmod{\lambda}
\]
for all $m\in M$.   The lemma is now clear since $\beta_M$ is the sum of integers divisible by $\ell$.
\end{proof}

\begin{lemma} \label{L:betaM nonzero}
The integer $\beta_M$ is nonzero for all $M\in \calM$.
\end{lemma}
\begin{proof}
Suppose that there is a set $M\in\calM$ such that $\beta_M=0$.  We thus have $\prod_i \pi_i^{m_i} = 1$ for all $m\in M$.  Therefore, the subgroup $\Phi_{A,\q}$ of $\Qbar^\times$ generated by $\pi_1,\ldots, \pi_{2g}$ is an abelian group of rank at most $2g - \rank_\ZZ (\ZZ M)$.   By our definition of $\calM$, $\Phi_{A,\q}$ has rank strictly less than $r$.    This is a contradiction since $\Phi_{A,\q}$ has rank $r$ by our initial choice of $\q$. Therefore, $\beta_M$ is nonzero for all $M\in \calM$.
\end{proof}

From Weil, we know that each $\pi_i$ has complex absolute value $N(\q)^{1/2}$ under any embedding $F\subseteq \CC$.    We can bound $[F:\QQ]$ in terms of $g$, so there are positive constants $c'$ and $\gamma'$, depending only on $g$, such that 
\[
|\beta_M| < c' N(\q)^{\gamma'}
\]
for all $M\in \calM$.   We may thus assume that $c$ and $\gamma$ are taken so that $|\beta_M| < c N(\q)^{\gamma}$ holds for all $M\in \calM$.

Now suppose that $\rank(H_{A,\ell})<r$.    By Lemmas~\ref{L:betaM congruences} and \ref{L:betaM nonzero}, we have $\ell \leq |\beta_M|$ for some $M\in \calM$.   Therefore, $\ell < c N(\q)^{\gamma}$.   However, this contradicts $\ell \geq c \cdot \max(\{[K:\QQ],h(A), N(\q)\})^\gamma \geq c N(\q)^{\gamma}$, so we must have $\rank(H_{A,\ell})=r$.

\subsection{The derived subgroup\texorpdfstring{ of $\calG_{A,\ell}$}{}} \label{SS:derived calG}

We have already proved that the $\ZZ_\ell$-group scheme $\calG_{A,\ell}$ is reductive.   Let $\calS_{A,\ell}$ be the derived subgroup of $\calG_{A,\ell}$; it is semisimple group scheme over $\ZZ_\ell$.

By Theorem~\ref{T:equal groups}, there is a subgroup $\Gamma$ of $\GL_{A[\ell]}(\FF_\ell)$ that acts semisimply on $A[\ell]$ such that $\boldG_{\Gamma} = (\calS_{A,\ell})_{\FF_\ell}$, where $\boldG_{\Gamma}$ is defined as in \S\ref{SS:semisimple approximation}.   By Theorem~\ref{T:Nori}, we have $\Gamma^+= \calS_{A,\ell}(\FF_\ell)^+$.  Since $\Gamma^+$ is a normal subgroup of $\Gamma$, we find that $\Gamma^+$ acts semisimply on $A[\ell]$ by Clifford's theorem \cite{MR632548}*{11.1}.   We have $\boldG_{\Gamma}=\boldG_{\Gamma^+}$, so we may take  $\Gamma := \calS_{A,\ell}(\FF_\ell)^+$.

\begin{lemma} \label{L:calS mod ell basics}
\begin{romanenum}
\item  \label{L:calS mod ell basics i}
Every simple quotient of $\calS_{A,\ell}(\FF_\ell)^+$ is a finite simple group of Lie type in characteristic $\ell$.  In particular, $\calS_{A,\ell}(\FF_\ell)^+$ is perfect.
\item \label{L:calS mod ell basics ii}
The quotient group $\calS_{A,\ell}(\FF_\ell)/\calS_{A,\ell}(\FF_\ell)^+$ is abelian and its cardinality can be bounded in terms of $g$.
\item \label{L:calS mod ell basics iii}
The group $\calS_{A,\ell}(\FF_\ell)^+$ is the commutator subgroup of $\calS_{A,\ell}(\FF_\ell)$.
\end{romanenum}
\end{lemma}
\begin{proof}
Let $G$ be a finite simple quotient of $\calS_{A,\ell}(\FF_\ell)^+$.  By increasing $c$, we may assume that $\ell > J(2g)$ with $J(2g)$ as in Theorem~\ref{T:Larsen-Pink}.     The group $\calS_{A,\ell}(\FF_\ell)^+$, and hence also $G$, is generated by elements of order $\ell$, so Theorem~\ref{T:Larsen-Pink} implies that $G$ is either a finite (nonabelian) simple group of Lie type in characteristic $\ell$ or a cyclic group of order $\ell$.

Suppose that $G$ is a cyclic group of order $\ell$.   Theorem~\ref{T:Larsen-Pink}  implies that $\Gamma=\calS_{A,\ell}(\FF_\ell)^+$ contains a normal subgroup $U\neq 1$ that is an $\ell$-group.  So $U$ is a unipotent subgroup of $\GL_{A[\ell]}(\FF_\ell)$.     Since $U$ is a normal subgroup of $\Gamma$, we find that $U$ acts semisimply on $A[\ell]$ by Clifford's theorem \cite{MR632548}*{11.1}.   Since the action of $U$ on $A[\ell]$ is unipotent and semisimple, we must have $U=1$.  This contradiction proves part (\ref{L:calS mod ell basics i}).

Since $(\calS_{A,\ell})_{\FF_\ell}$ is of the form $\boldG_\Gamma$, part (\ref{L:calS mod ell basics ii}) follows from statement 3.6(v) of \cite{MR880952}.

From (\ref{L:calS mod ell basics ii}), we find that $\calS_{A,\ell}(\FF_\ell)' \subseteq \calS_{A,\ell}(\FF_\ell)^+$ and that $\calS_{A,\ell}(\FF_\ell)^+/\calS_{A,\ell}(\FF_\ell)'$ is abelian.  Since $\calS_{A,\ell}(\FF_\ell)^+$ has no abelian quotients by (\ref{L:calS mod ell basics i}), we deduce that $\calS_{A,\ell}(\FF_\ell)^+=\calS_{A,\ell}(\FF_\ell)'$.   This proves (\ref{L:calS mod ell basics iii}). 
\end{proof}

Let $\calB$ be the inverse image of $\calS_{A,\ell}(\FF_\ell)^+$ under the reduction modulo $\ell$ map $\calS_{A,\ell}(\ZZ_\ell) \to \calS_{A,\ell}(\FF_\ell)$.

\begin{lemma}  \label{L:subgroup exp} 
Let $H$ be a closed subgroup of $\calS_{A,\ell}(\ZZ_\ell)$ whose image in $\calS_{A,\ell}(\FF_\ell)$ contains $\calS_{A,\ell}(\FF_\ell)^+$.   Then $H\supseteq \calB$.  
\end{lemma}
\begin{proof}
Without loss of generality, we may assume that the image of $H$ in $\calS_{A,\ell}(\FF_\ell)$ is $\calS_{A,\ell}(\FF_\ell)^+=\Gamma$.   We have $\boldG_\Gamma=(\calS_{A,\ell})_{\FF_\ell}$, so $\Ndim(H) = \dim (\calS_{A,\ell})_{\FF_\ell}$.

Let $S$ be the Zariski closure of $H$ in $G_{A,\ell}$.   We have $S \subseteq (\calS_{A,\ell})_{\QQ_\ell}$.   By Theorem~\ref{T:Larsen exponential}(\ref{T:Larsen exponential i}), we have $\Ndim(H) \leq \dim S \leq \dim (\calS_{A,\ell})_{\QQ_\ell}$.   So we have
\[
\Ndim(H) \leq \dim S \leq  \dim (\calS_{A,\ell})_{\QQ_\ell} = \dim (\calS_{A,\ell})_{\FF_\ell}=\Ndim(H).
\]
Therefore,  $\dim S =  \dim (\calS_{A,\ell})_{\QQ_\ell}$ and hence $S = (\calS_{A,\ell})_{\QQ_\ell}$ since $(\calS_{A,\ell})_{\QQ_\ell}$ is connected.   In particular, $S$ is connected and semisimple.   Applying Theorem~\ref{T:Larsen exponential}(\ref{T:Larsen exponential ii}), we find that
\[
[S(\QQ_\ell)\cap \calG_{A,\ell}(\ZZ_\ell) : H] \leq c_4(2g).
\]
After possibly increasing $c$, we may assume that $\ell > c_4(2g)$ and thus $[\calS_{A,\ell}(\ZZ_\ell) : H] < \ell$.  Therefore, $H$ contains the pro-$\ell$ group that is the kernel of the reduction map $\calS_{A,\ell}(\ZZ_\ell) \to \calS_{A,\ell}(\FF_\ell)$.     It is now clear that $H$ is equal to $\calB$.
\end{proof}

\begin{lemma} \label{L:der prop}
The group $\calB$ is perfect and is equal to $\calS_{A,\ell}(\ZZ_\ell)'$.
\end{lemma}
\begin{proof}
Let $H$ be the group $\calB'$ or $\calS_{A,\ell}(\ZZ_\ell)'$.  By Lemma~\ref{L:calS mod ell basics}, we find that the image of $H$ in $\calS_{A,\ell}(\FF_\ell)$ is $\calS_{A,\ell}(\FF_\ell)^+$.   Therefore, $H=\calB$ by Lemma~\ref{L:subgroup exp}.  This proves the lemma.
\end{proof}

  We now summarize some important properties of the groups $\calS_{A,\ell}(\ZZ_\ell)$.
\begin{prop} \label{P:calS}
\begin{romanenum}
\item  \label{P:calS i}
The group $\calS_{A,\ell}(\ZZ_\ell)'$ agrees with the inverse image of $\calS_{A,\ell}(\FF_\ell)'$ under the reduction modulo $\ell$ homomorphism $\calS_{A,\ell}(\ZZ_\ell)\to \calS_{A,\ell}(\FF_\ell)$.
\item  \label{P:calS ii}
The groups $\calS_{A,\ell}(\ZZ_\ell)'$ and $\calS_{A,\ell}(\FF_\ell)'$ are perfect.
\item \label{P:calS iii}
The cardinality of the group $\calS_{A,\ell}(\ZZ_\ell)/\calS_{A,\ell}(\ZZ_\ell)'\cong \calS_{A,\ell}(\FF_\ell)/\calS_{A,\ell}(\FF_\ell)'$ can be bounded in terms of $g$.
\item  \label{P:calS iv}
The finite simple quotients of $\calS_{A,\ell}(\ZZ_\ell)'$ are all groups of Lie type in characteristic $\ell$.
\item  \label{P:calS v}
We have $\calG_{A,\ell}(\ZZ_\ell)'=\calS_{A,\ell}(\ZZ_\ell)'$.  
\end{romanenum}
\end{prop}
\begin{proof}
By Lemma~\ref{L:calS mod ell basics}, $\calS_{A,\ell}(\FF_\ell)'$ is perfect and equals $\calS_{A,\ell}(\FF_\ell)^+$.   Parts (\ref{P:calS i}) and (\ref{P:calS ii}) are immediate consequences of Lemma~\ref{L:der prop}.   Since $\calS_{A,\ell}$ is smooth, the reduction modulo $\ell$ map $\calS_{A,\ell}(\ZZ_\ell) \to \calS_{A,\ell}(\FF_\ell)$ is surjective; part (\ref{P:calS iii}) now follows from part (\ref{P:calS i}) and Lemma~\ref{L:calS mod ell basics}(\ref{L:calS mod ell basics ii}).

Let $G$ be a finite simple quotient of $\calS_{A,\ell}(\ZZ_\ell)'$; it is nonabelian since $\calS_{A,\ell}(\ZZ_\ell)'$ is perfect.   The kernel of the reduction modulo $\ell$ map $\calS_{A,\ell}(\ZZ_\ell)'\twoheadrightarrow \calS_{A,\ell}(\ZZ_\ell)'$ is a pro-$\ell$ group.  Since pro-$\ell$ groups are prosolvable, we find that $G$ must be a quotient of $\calS_{A,\ell}(\FF_\ell)'$.  By Lemma~\ref{L:calS mod ell basics}(\ref{L:calS mod ell basics i}), the group $G$ is a finite simple group of Lie type in characteristic $\ell$.

Define $H:=\calG_{A,\ell}(\ZZ_\ell)'$; it is a closed subgroup of $\calS_{A,\ell}(\ZZ_\ell)$.   With $G:=(\calG_{A,\ell})_{\FF_\ell}$, the quotient $G(\FF_\ell)/G(\FF_\ell)^+$ is abelian by \cite{MR3566639}*{Proposition~1.1}.  This implies that $H$ modulo $\ell$ is a subgroup of $\calG_{A,\ell}(\FF_\ell)^+=\calS_{A,\ell}(\FF_\ell)^+$.   Since $H\supseteq \calS_{A,\ell}(\ZZ_\ell)'$, we find by Lemma~\ref{L:der prop} that $H$ modulo $\ell$ contains the group $\calS_{A,\ell}(\FF_\ell)^+$.    Therefore, $H=\calB$.  By Lemma~\ref{L:der prop}, we deduce that $\calG_{A,\ell}(\ZZ_\ell)'=\calS_{A,\ell}(\ZZ_\ell)'$. 
\end{proof}

\begin{remark}
There is an alternate description of the commutator subgroups of $\calS_{A,\ell}(\ZZ_\ell)$ and $\calS_{A,\ell}(\FF_\ell)$.  Let $\pi\colon \calS_{A,\ell}^\sc \to \calS_{A,\ell}$ be the simply connected cover of $\calS_{A,\ell}$.   We have $\calS_{A,\ell}(\ZZ_\ell)' = \pi( \calS_{A,\ell}^\sc(\ZZ_\ell) )$ and $\calS_{A,\ell}(\FF_\ell)' = \pi( \calS_{A,\ell}^\sc(\FF_\ell) )$; as noted in \S1.2 of \cite{MR1944805}, $\pi( \calS_{A,\ell}^\sc(\FF_\ell) )$ is the commutator subgroup of $\calS_{A,\ell}(\FF_\ell)$ and $\pi( \calS_{A,\ell}^\sc(\ZZ_\ell) )$ is the subgroup of $\calS_{A,\ell}(\ZZ_\ell)$ whose reduction modulo $\ell$ lies in $\pi( \calS_{A,\ell}^\sc(\FF_\ell) )$.
\end{remark}

\subsection{Proof of Theorem~\ref{T:main new}(\ref{T:main new d})}
With $\Gamma:=\bbar\rho_{A,\ell}(\Gal_K)$, we have $\boldG_\Gamma= (\calS_{A,\ell})_{\FF_\ell}$ by Theorem~\ref{T:equal groups}.  By Theorem~\ref{T:Nori}, we have $\Gamma^+= \boldG_\Gamma(\FF_\ell)^+=\calS_{A,\ell}(\FF_\ell)^+$ and hence $ \calS_{A,\ell}(\FF_\ell)^+$ is a subgroups of $\Gamma= \bbar\rho_{A,\ell}(\Gal_K)$.   

Since $\calS_{A,\ell}(\FF_\ell)^+$ is perfect by Lemma~\ref{L:calS mod ell basics}(\ref{L:calS mod ell basics i}), we have $\calS_{A,\ell}(\FF_\ell)^+ \subseteq \bbar\rho_{A,\ell}(\Gal_K)'$. Therefore, $\bbar\rho_{A,\ell}(\Gal_K)' \subseteq \calG_{A,\ell}(\FF_\ell)' = \calS_{A,\ell}(\FF_\ell)'$, where the last equality follows from Lemma~\ref{P:calS}(\ref{P:calS v}).   By Lemma~\ref{L:calS mod ell basics}(\ref{L:calS mod ell basics iii}), we deduce that
\[
\bbar\rho_{A,\ell}(\Gal_K)' = \calS_{A,\ell}(\FF_\ell)' = \calS_{A,\ell}(\FF_\ell)^+.
\]

So $\rho_{A,\ell}(\Gal_K)'$ is a closed subgroup of $\calS_{A,\ell}(\ZZ_\ell)$ whose reduction modulo $\ell$ is equal to $\calS_{A,\ell}(\FF_\ell)^+$.   By Lemmas \ref{L:subgroup exp}  and \ref{L:der prop}, we have $\rho_{A,\ell}(\Gal_K)' = \calS_{A,\ell}(\ZZ_\ell)'$.   Therefore, $\rho_{A,\ell}(\Gal_K)'=\calG_{A,\ell}(\ZZ_\ell)'$ by Lemma~\ref{P:calS}(\ref{P:calS v}).

\section{Abelian representations} \label{S:abelian}
Fix an abelian variety $A$ of dimension $g\geq 1$ defined over a number field $K$.  Assume that all the $\ell$-adic monodromy groups $G_{A,\ell}$ are connected; equivalently, $K_A^\conn=K$.     

Fix an isogeny 
\[
\iota\colon A\to{\prod}_{i=1}^s A_i,
\] 
where $A_i=B_i^{e_i}$ with $e_i\geq 1$ and the $B_i$ are simple abelian varieties over $K$ that are pairwise non-isogenous.  By Proposition~\ref{P:conn facts}(\ref{P:conn facts iv}), each $B_i$ is geometrically simple and they are pairwise non-isogenous.   The isogeny $\iota$ induces an isomorphism $\End(A)\otimes_\ZZ \QQ = \prod_{i=1}^s \End(A_i)\otimes_\ZZ \QQ$.  Note that none of the results of this section will depend on the choice of $\iota$.

Let $L$ be the center of $\End(A)\otimes_\ZZ \QQ$.   Let $L_i$ be the center of $\End(A_i)\otimes_\ZZ \QQ$; we can also identify it with the center of $\End(B_i)\otimes_\ZZ \QQ$.   We have $L=\prod_{i=1}^s L_i$ and each $L_i$ is a number field.   

Let $T_L$ be the torus defined over $\QQ$ for which $T_L(R)=(L\otimes_\QQ R)^\times$ for any $\QQ$-algebra $R$ with the obvious functoriality.    We have $T_L = \prod_{i=1}^s T_{L_i}$, where $T_{L_i}$ is equal to the restriction of scalars $\Res_{L_i/\QQ}(\GG_{m,L_i})$. 

\subsection{The homomorphism \texorpdfstring{$\beta_{A,\ell}$}{beta}}

Take any prime $\ell$. The isogeny $\iota$ induces an isomorphism $V_\ell(A)=\oplus_{i=1}^s V_\ell(A_i)$ of $\QQ_\ell[\Gal_K]$-modules.   Note that each $V_\ell(A_i)$ is a free $L_i\otimes_\QQ \QQ_\ell$-module of rank $d_i:=2\dim A_i/[L_i:\QQ]$, cf.~\cite{MR0457455}*{II Theorem~2.1.1}.  Since the $\Gal_K$ and $L$ actions  on $V_\ell(A)$ commute, we have
\[
\rho_{A,\ell}(\Gal_K) \subseteq   \Aut_{L\otimes_\QQ \QQ_\ell}(V_\ell(A)) =  {\prod}_{i=1}^s \Aut_{L_{i} \otimes_{\QQ} \QQ_\ell}(V_\ell(A_i))\cong \prod_{i=1}^s \GL_{d_i}(L_i\otimes_\QQ \QQ_\ell).
\] 
By taking determinants, we obtain from $\rho_{A,\ell}$ a homomorphism
\[
\beta_{A,\ell}\colon \Gal_K \to \prod_{i=1}^s (L_i \otimes_\QQ \QQ_\ell)^\times= \prod_{i=1}^s T_{L_i}(\QQ_\ell)= T_L(\QQ_\ell).
\]
Using $\prod_{i=1}^s T_{L_i}= T_L$, we clearly have $\prod_{i=1}^s \beta_{A_i,\ell}=\beta_{A,\ell}$.   The homomorphism $\beta_{A,\ell}\colon \Gal_K\to T_L(\QQ_\ell)$ is unramified at all non-zero prime ideals $\p\nmid \ell$ of $\OO_K$ for which $A$ has good reduction since $\rho_{A,\ell}$ has this property. 

 \begin{lemma} \label{L:potential good}
There is a field extension $F/K$ with $[F:K]\ll_g 1$ such that for all primes $\ell$, the homomorphism $\beta_{A,\ell}|_{\Gal_{F}}\colon \Gal_F \to T_L(\QQ_\ell)$ is unramified at all non-zero prime ideals $\p\nmid \ell$ of $\OO_F$.
\end{lemma}
\begin{proof}
Define the number field $F=K(A[12])$; we have $[F: K] \leq 12^{(2g)^2}\ll_g 1$.  By a criterion of Raynaud (see Proposition 4.7 of \cite{MR0354656}*{Expose IX}), the abelian variety $A_F$ has semistable reduction at all non-zero prime ideals of $\OO_F$.  

Take any prime $\ell$.  Take any non-zero prime ideal $\p\nmid\ell$ of $\OO_F$ and denote by $I_\p$ an inertia subgroup of $\Gal_F$ at $\p$.  Since $A$ has semistable reduction at $\p$, we know from Grothendieck (Proposition 3.5 of \cite{MR0354656}*{Expose IX}) that the group $\rho_{A,\ell}(I_\p)$ consists of unipotent elements in $\GL_{V_\ell}(\QQ_\ell)$.   Since $T_L$ is a torus, we deduce that $\beta_{A,\ell}(I_\p) =1$.    Therefore, $\beta_{A,\ell}|_{\Gal_F}$ is unramified at $\p$.
\end{proof}

\subsection{The torus \texorpdfstring{$Y$}{Y}}\label{S:abelian setup}

Fix notation as in \S\ref{SS:MT group} and let $V_i$ be the homology group of $A_i$.  The isogeny $\iota$ induces an isomorphism $V= \oplus_{i=1}^s V_i$.    Since $\End(A)\otimes_\ZZ \QQ$ acts faithfully on $V$, we have a faithful action of $T_L$ on $V$ and we may thus identify $T_L$ with an algebraic subgroup of $\GL_V$.

 Let $\mathscr G$ be the algebraic subgroup of $\GL_V$ over $\QQ$ that satisfies
\[
\mathscr G(R)=  {\prod}_{i=1}^s \Aut_{L_{i} \otimes_{\QQ} R}(V_i \otimes_{\QQ} R)
\]
for any $\QQ$-algebra $R$ with the obvious functoriality.  Since $V=\oplus_{i=1}^s V_i$, we find that $T_L \subseteq \mathscr G \subseteq \GL_V$.   The $L_i$-vector space $V_i$ has dimension $d_i=2\dim A_i /[L_i:\QQ]$.  We define
\[
{\det}_{L} \colon \mathscr G \to T_{L}
\]
to be the homomorphism that for each $\QQ$-algebra $R$ takes $(g_1,\ldots, g_s) \in \mathscr G(R)\cong \prod_{i=1}^s \GL_{d_i}(L_i \otimes_\QQ R)$ to $(\det g_1,\ldots, \det g_s) \in \prod_{i=1}^s (L_i \otimes_\QQ R)^\times = (L\otimes_\QQ R)^\times  = T_L(R)$.    In particular, $\det_L$ gives an isogeny $T_L\to T_L$ of tori of degree $d_1\cdots d_s$.

Observe that $\MT_A\subseteq \calG$ since the Mumford--Tate group $\MT_A$ commutes with the action of $L$ on $V$.  Let $C_A$ be the central torus of the reductive group $\MT_A$.  One knows that the commutant of $\MT_A$ in $\End_\QQ(V)$ is naturally isomorphic to $\End(A)\otimes_\ZZ\QQ$.  We have $C_A\subseteq T_L$ since $C_A$ commutes with $\MT_A$ and $\End(A)\otimes_\ZZ\QQ$.   Define the algebraic group \[
Y:={\det}_L(C_A);
\] 
it is a subtorus of $T_L$ defined over $\QQ$.  Observe that $\det_L|_{C_A}\colon C_A \to Y$ is an isogeny of degree at most $d_1\cdots d_s$.    For later, note that $d_1\cdots d_s \ll_g 1$.

The following gives some important information on the image of $\beta_{A,\ell}$.

\begin{prop} \label{P:Zariski dense in Y}
We have $\beta_{A,\ell}(\Gal_K) \subseteq Y(\QQ_\ell)$.  Moreover,  $\beta_{A,\ell}(\Gal_K)$ is Zariski dense in $Y_{\QQ_\ell}$ and ${\det}_L(G_{A,\ell})=Y_{\QQ_\ell}$.
\end{prop}
\begin{proof}
 Using comparison isomorphisms, we have
\[
\mathscr G(\QQ_\ell)=  {\prod}_{i=1}^s \Aut_{L_{i} \otimes_{\QQ} \QQ_\ell}(V_\ell(A_i)) = \Aut_{L\otimes_\QQ \QQ_\ell}(V_\ell(A))
\] 
and hence $\rho_{A,\ell}(\Gal_K)\subseteq \mathscr G(\QQ_\ell)$.  In particular, $G_{A,\ell}\subseteq \mathscr{G}_{\QQ_\ell}$.  The homomorphism
\[
\beta_{A,\ell}\colon \Gal_K \to T_L(\QQ_\ell)
\]
can be obtained by composing $\rho_{A,\ell}\colon \Gal_K \to \mathscr G (\QQ_\ell)$ with $\det_L \colon \mathscr{G}(\QQ_\ell)\to T_L(\QQ_\ell)$.  

We have $\beta_{A,\ell}(\Gal_K)=\det_L(\rho_{A,\ell}(\Gal_K)) \subseteq \det_L(G_{A,\ell}(\QQ_\ell))$.   By Proposition~\ref{P:MT inclusion}, we have  $\beta_{A,\ell}(\Gal_K) \subseteq \det_L(\MT_{A}(\QQ_\ell))$.   Since $\MT_A$ is reductive and $T_L$ is a torus, we have $\det_L(\MT_A)=\det_L(C_A)=Y$.  Therefore, $\beta_{A,\ell}(\Gal_K)\subseteq Y(\QQ_\ell)$.

By Proposition~\ref{P:same centers},  $(C_{A})_{\QQ_\ell}$ is the central torus of the reductive group $G_{A,\ell}$.   Since $G_{A,\ell}$ is reductive, we have ${\det}_L(G_{A,\ell})={\det}_L((C_{A})_{\QQ_\ell})=Y_{\QQ_\ell}$.  Since $\rho_{A,\ell}(\Gal_K)$ is Zariski dense in $G_{A,\ell}$, we deduce that $\beta_{A,\ell}(\Gal_K)=\det_L(\rho_{A,\ell}(\Gal_K))$ is Zariski dense in $Y_{\QQ_\ell}$.
\end{proof}

Since $\beta_{A,\ell}$ is continuous and $\Gal_K$ is compact, we have $\beta_{A,\ell}(\Gal_K) \subseteq Y(\QQ_\ell)_c$, where $Y(\QQ_\ell)_c$ is the maximal compact subgroup of $Y(\QQ_\ell)$ with respect to the $\ell$-adic topology.   That $Y(\QQ_\ell)$ has a unique maximal compact subgroup uses that $Y(\QQ_\ell)$ is abelian.  The following theorem, which we will prove in \S\ref{SS:proof of main abelian}, is the main result of \S\ref{S:abelian}.

\begin{thm} \label{T:main abelian}
\begin{romanenum}
\item \label{T:main abelian i}
For any prime $\ell$, we have $[Y(\QQ_\ell)_c:\beta_{A,\ell}(\Gal_K)] \ll_{g,[K:\QQ]} 1$.
\item \label{T:main abelian ii}
For any prime $\ell$ that is unramified in $K$, we have $[Y(\QQ_\ell)_c:\beta_{A,\ell}(\Gal_K)] \ll_{g} 1$.
\end{romanenum}
\end{thm}

\subsection{\texorpdfstring{$\lambda$}{lambda}-adic representations} \label{SS:lambda-adic reps}

Throughout \S\ref{SS:lambda-adic reps}, we shall further assume that $A$ is a power of a simple abelian variety and hence $L$ is a number field.

We have $L_\ell:=L\otimes_\QQ \QQ_\ell = \prod_{\lambda | \ell} L_\lambda$, where $\lambda$ runs over the prime ideals of $\OO_L$ dividing $\ell$ and $L_\lambda$ is the $\lambda$-adic completion of $L$.  For each $\lambda$, define $V_\lambda(A):= V_\ell(A) \otimes_{L_\ell}  L_\lambda$; it is a $L_\lambda$-vector space of dimension $d=2g/[L:\QQ]$.   We have an isomorphism $V_\ell(A)=\oplus_{\lambda|\ell} V_\lambda(A)$ of $\QQ_\ell[\Gal_K]$-algebras.   The Galois action on $V_\lambda(A)$ gives a representation
\[
\rho_{A,\lambda} \colon \Gal_K\to \Aut_{L_\lambda}(V_\lambda(A)) \cong \GL_d(L_\lambda).
\]
By composing $\rho_{A,\lambda}$ with the determinant map, we obtain a homomorphism \[
\beta_{A,\lambda}\colon \Gal_K\to L_\lambda^\times.
\]
Using $T_L(\QQ_\ell)=(L\otimes_\QQ \QQ_\ell)^\times=\prod_{\lambda|\ell} L_\lambda^\times$, we find that $\beta_{A,\ell}=\prod_{\lambda|\ell} \beta_{A,\lambda}$.

\begin{lemma} \label{L:lambda compatible}
Take any non-zero prime ideal $\p$ of $\OO_K$ for which $A$ has good reduction.    Then there is an element $F_\p \in L^\times$ such that $\beta_{A,\lambda}$ is unramified at $\p$ and satisfies
\[
\beta_{A,\lambda}(\Frob_\p)=F_\p
\]
for all non-zero prime ideals $\lambda$ of $\OO_F$ that do not divide the characteristic of $\FF_\p$.
\end{lemma}
\begin{proof}
Let $\lambda$ be a non-zero prime ideal of $\OO_F$ that does not divide the characteristic of $\FF_\p$.  Theorem~2.1.1 of \cite{MR0457455}*{II~\S1} says that there is a polynomial $Q_\p(x)\in L[x]$ such that $\det(xI -\rho_{A,\lambda}(\Frob_\p)) = Q_\p(x)$; the polynomial $Q_\p(x)$ does not depend on the choice of $\lambda$.  Therefore, $\beta_{A,\lambda}(\Frob_\p)= \det(\rho_{A,\lambda}(\Frob_\p))=(-1)^d Q_\p(0)$ which is an element of $L^\times$ that does not depend on $\lambda$.
\end{proof}

For each prime $\ell$, let $\chi_\ell\colon \Gal_K \to \ZZ_\ell^\times$ be the $\ell$-adic cyclotomic character;  we have $\sigma(\zeta)=\zeta^{\chi_\ell(\sigma) \bmod{\ell^n}}$ for any $\ell^n$-th root of unity $\zeta\in \Kbar$ and $\sigma\in \Gal_K$.    Reducing modulo $\ell$, we obtain a homomorphism $\bbar{\chi}_\ell\colon \Gal_K \to \FF_\ell^\times$.

For any non-zero prime ideal $\lambda$ of $\OO_L$, the image of $\beta_{A,\lambda}$ is compact so it lies in $\OO_\lambda^\times$.   Reducing gives a homomorphism 
\[
\bbar\beta_{A,\lambda}\colon \Gal_K \to \FF_\lambda^\times.
\] 

\begin{lemma} \label{L:tame inertia weights}
Fix a prime $\ell\geq 5$ that splits completely in $L$ and is unramified in $K$ and let $\lambda$ be a prime ideal of $\OO_L$ dividing $\ell$.    Let $\p|\ell$ be a prime ideal of $\OO_K$ for which $A$ has good reduction at $\p$ and let $I_\p$ be an inertia subgroup of $\Gal_K$ for the prime $\p$.   Then there is an integer $0\leq b \leq 2g/[L:\QQ]$ such that $\bbar{\beta}_{A,\lambda}(\sigma)=\bbar\chi_\ell(\sigma)^b$ holds for all $\sigma\in I_\p$.
\end{lemma}
\begin{proof}
Since $\ell$ splits completely in $L$, we have $\FF_\lambda=\FF_\ell$ and hence $\bbar\beta_{A,\lambda}\colon \Gal_K \to \FF_\lambda^\times=\FF_\ell^\times$.

There is a $\OO_\lambda$-submodule $M$ of $V_\lambda(A)$ of rank $\dim_{L_\lambda} V_\lambda(A)$ that is stable under the action of $I_\p$.   Let $W$ be the semi-simplification of $M/\lambda M$ as a module over $\FF_\lambda[I_\p]=\FF_\ell[I_\p]$.  The character $\bbar\beta_{A,\lambda}|_{I_\p}$ thus arises by taking the determinant of the action of $I_\p$ on the vector space $W$ over $\FF_\lambda=\FF_\ell$.

Take any irreducible $\FF_\ell[I_\p]$-submodule $\calV$ of $W$ and set $n=\dim_{\FF_\ell}\calV$.   Let $Z$ be the ring of endomomorphisms of $\calV$ as an $\FF_\ell[I_\p]$-module.  Since $\calV$ is irreducible, $Z$ is a division algebra of finite dimension over $\FF_\ell$.  Therefore, $Z$ is a finite field and $\calV$ is a vector space of dimension $1$ over $Z$.   Choose an isomorphism $Z\cong \FF_{\ell^n}$ of fields.   The action of $I_\p$ on $\calV$ corresponds to a character $\alpha\colon I_\p \to Z^\times\cong \FF_{\ell^n}^\times$.  Let $\varepsilon_1,\ldots, \varepsilon_n \colon I_\p \to \FF_{\ell^n}^\times$ be the \defi{fundamental character of level $n$}, cf.~\cite{MR0387283}*{\S1.7}.    There are unique integers $e_i\in \{0,1,\ldots, \ell-1\}$ such that $\alpha= \varepsilon_1^{e_1}\cdots \varepsilon_n^{e_n}$.  These integers $e_1,\ldots,e_n$ are called the \defi{tame inertia weights} of $\calV$.     

Note that $\calV$ will arise in the semi-simplification of $A[\ell]$ as an $\FF_\ell[I_\p]$-module since $V_\lambda(A)\subseteq V_\ell(A)$.  From \cite{MR2372809}*{Th\'eor\`eme 1.2}, we find that all of the integers $e_i$ are either $0$ or $1$ (this uses that $A[\ell]$ is isomorphic as an $\FF_\ell[\Gal_K]$-module to the dual of $H^1_{\text{\'et}}(A_{\Kbar},\ZZ/\ell\ZZ)$ and the conditions on $\ell$ and $\p$ in the statement of the lemma). 

Take any $\sigma\in I_\p$.   The determinant of $\varepsilon_i(\sigma)\in \FF_{\ell^n}^\times \subseteq \Aut_{\FF_\ell}(\FF_{\ell^n})$ is equal to $N_{\FF_{\ell^n}/\FF_\ell}(\varepsilon_i(\sigma))=\bbar{\chi}_\ell(\sigma)$, where the last equality uses Proposition~8 of \cite{MR0387283}.  Viewing $\alpha(\sigma)$ as an endomorphism of $\calV$, we thus have $\det(\alpha(\sigma))= \bbar\chi_\ell(\sigma)^{e_1+\cdots+e_n}$.  Since $e_i\in \{0,1\}$, there is an integer $0\leq b \leq \dim_{\FF_\ell} \calV$ such that $\det(\sigma|\calV)=\det(\alpha(\sigma))=\bbar\chi_\ell(\sigma)^b$ for all $\sigma\in I_\p$.   

Therefore, there is an integer $0\leq b \leq \dim_{\FF_\ell} W = \dim V_\lambda(A)=2g/[L:\QQ]$ such that $\bbar{\beta}_{A,\lambda}(\sigma)=\bbar\chi_\ell(\sigma)^b$ for all $\sigma\in I_\p$.
\end{proof}

\subsection{Serre tori} \label{SS:Serre tori}

We now recall the families of compatible abelian Galois representations described by Serre in Chapter II of \cite{MR1484415}; for statements generalized to $\lambda$-adic representations see section I of \cite{MR0457455}.  In Lemma~\ref{L:locally algebraic}, we will see how these representations give rise to our $\beta_{A.\ell}$.   

Define the torus 
\[
T_K:=\Res_{K/\QQ}(\GG_{m,K}),
\] 
where we are taking restriction of scalars from $K$ to $\QQ$.  Let $I_K$ be the group of ideles of $K$.  

Fix a \emph{modulus} $\m$, i.e., a sequence $\{\m_v\}_v$ of non-negative integers indexed by the places $v$ of $K$ satisfying $\m_v=0$ for all but finitely many places $v$.   
For a finite place $v$ with $\m_v=0$, we define $U_{v,\m}=\OO_v^\times$. For a finite place $v$ with $\m_v>0$, we define $U_{v,\m}$ to be the subgroup of $x\in \OO_v^\times$ for which the $v$-adic valuation of $1-x$ is at least $\m_v$.   For an infinite place $v$, let $U_{v,\m}$ be the connected component of $K_v^\times$ containing the identity if $\m_v\geq 1$ and $K_v^\times$ otherwise.  Define $U_\m:=\prod_v U_{v,\m}$; it is an open subgroup of $I_K$.   Set $I_\m:=I_K/U_\m$.   

  Let $E_\m$ be the group of $x\in K^\times$ for which $x \in U_\m$.   We let $T_\m$ be the quotient of $T_K$ by the Zariski closure of $E_\m \subseteq K^\times = T_K(\QQ)$.  In \cite{MR1484415}*{II}, Serre constructs a commutative algebraic group $S_\m$ over $\QQ$ whose neutral component is $T_\m$ and for which the quotient $S_\m/T_\m$ equals the finite group $C_\m:=I_K/(I_\m K^\times)$.   He also constructs a homomorphism 
  \[
  \varepsilon\colon I_\m \to S_\m(\QQ)  
  \]   
for which we have a commutative diagram
\[
\xymatrix{
1 \ar[r] & K^\times/E_\m \ar[r]\ar[d] & I_\m \ar[r]\ar[d]^\varepsilon & C_\m \ar@{=}[d] \ar[r] &1 \\
1 \ar[r] & T_\m(\QQ) \ar[r] & S_\m(\QQ) \ar[r] & C_\m \ar[r] & 1.
}
\]

 They are characterized by the following universal property: for any field extension $k/\QQ$, homomorphism $f'\colon T_{\m,k}\to A$ of commutative algebraic groups over $k$ and homomorphism $\varepsilon'\colon I_\m \to A(k)$ such that the following diagram commutes
\[
\xymatrix{
K^\times/E_\m \ar[d] \ar[r] &I_\m\ar[d]^{\varepsilon'}\\
T_\m(k) \ar[r]^{f'} & A(k),}
\]
 there is a unique homomorphism $g\colon S_{\m,k}\to A$ which induces $f'$ and $\varepsilon'$ (i.e., $f'$ is obtained by composing $T_\m \hookrightarrow S_\m$ and $g$ and $\varepsilon'$ is obtained by composing $\varepsilon$ with $S_\m(\QQ) \hookrightarrow S_\m(k) \xrightarrow{g} A(k)$). \\
 
Take any prime $\ell$.  Let $\alpha_\ell \colon I_K \to S_\m(\QQ_\ell)$ be the homomorphism obtained by composing the natural projection $I_K\to \prod_{v|\ell}K_v^\times = (K\otimes_\QQ \QQ_\ell)^\times = T_K(\QQ_\ell)$ with the homomorphism $T_K(\QQ_\ell)\to S_\m(\QQ_\ell)$.   Composing the quotient map $I_K\to I_\m$ with $\varepsilon$ gives a homomorphism $I_K\to S_\m(\QQ)$ that we shall also denote by $\varepsilon$.     For all $x\in K^\times$, we have $\alpha_\ell(x)=\varepsilon(x)$, cf.~\cite{MR1484415}*{II~2.3}.     We thus have a continuous homomorphism 
\[
\varepsilon_\ell \colon I_K \to S_\m(\QQ_\ell),\quad x\mapsto \varepsilon(x)\alpha_\ell(x)^{-1}
\]
that vanishes on $K^\times$.    We shall also denote by 
\[
\varepsilon_\ell \colon \Gal_K \to S_\m(\QQ_\ell)
\]
the homomorphism arising from $\varepsilon_\ell$ via class field theory.  

We now observe that the homomorphisms $\varepsilon_\ell$ are compatible.  Take any non-zero prime $\p\nmid \ell$ of $\OO_K$ such that $\m_v=0$, where $v$ is the corresponding place.    Let $\pi_\p$ be an element of $I_K$ that is a uniformizer at the place $v$ and is $1$ at the other places.   Therefore,
\[
\varepsilon_\ell(\Frob_\p) = \varepsilon_\ell(\pi_\p) = \varepsilon(\pi_\p) \alpha_\ell(\pi_\p)^{-1} = \varepsilon(\pi_\p);
\]
this is an element of $S_\m(\QQ)$ that is independent of $\ell$ and the choice of $\pi_\p$.\\

We now show that our homomorphisms $\beta_{A,\ell}$ arise from Serre's $\varepsilon_\ell$ for some modulus $\m$.

\begin{lemma} \label{L:locally algebraic}
There is a modulus $\m$ and a homomorphism $\phi\colon S_{\m} \to T_L$ of algebraic groups over $\QQ$ such that each $\beta_{A,\ell}$ is equal to the composition of $\varepsilon_\ell \colon \Gal_K\to S_\m(\QQ_\ell)$ with the homomorphism $S_\m(\QQ_\ell)\xrightarrow{\phi} T_L(\QQ_\ell)$.
\end{lemma}
\begin{proof}
Using $T_L=\prod_{i=1}^s T_{L_i}$ and $\beta_{A,\ell}=\prod_{i=1}^s \beta_{A_i,\ell}$, it suffices to prove the lemma for each $A_i$; we can find a common modulus $\m$ by increasing the values $\m_v$ appropriately.  We may thus assume that $A$ is a power of a simple abelian variety and hence $L$ is a field.    

Since $L$ is a number field, we have a natural isomorphism $L\otimes_\QQ \QQ_\ell = \prod_{\lambda|\ell} L_\lambda$, where the product is over the non-zero prime ideals of $\OO_L$ that divide $\ell$.   We have a homomorphism
\[
\beta_{A,\ell}\colon \Gal_K \to T_L(\QQ_\ell)=(L\otimes_\QQ \QQ_\ell)^\times = {\prod}_{\lambda|\ell} L_\lambda^\times.
\]
We thus have $\beta_{A,\ell} :=\prod_{\lambda|\ell} \beta_{A,\lambda}$, where $\beta_{A,\lambda}\colon \Gal_K \to L_\lambda^\times$ is obtained by composing $\beta_{A,\ell}$ with the obvious projection.    Note that these agree with the representations $\beta_{A,\lambda}$ defined in \S\ref{SS:lambda-adic reps}.  By Lemma~\ref{L:lambda compatible}, for every prime $\p$ for which $A$ has good reduction, there is an element $F_\p \in L^\times=T(\QQ)$ such that $\beta_{A,\lambda}(\Frob_\p)=F_\p$ for all $\lambda$ that do not divide the characteristic of $\FF_\p$.

Fix a non-zero prime ideal $\lambda'$ of $\OO_L$ and let $\ell'$ be the rational prime it divides.   Th\'eor\`em~2 of \cite{MR693314} now applies and says that $\beta_{A,\lambda'}$ is \emph{locally algebraic}.   The main theorem of section I.6 of \cite{MR0457455} then implies that there is a modulus $\m$ and a  homomorphism $\psi\colon S_{\m,L} \to \GG_{m,L}$ such that $\beta_{A,\lambda'}$ agrees with the composition
\[
\Gal_K\xrightarrow{\varepsilon_{\ell'}} S_\m(\QQ_{\ell'}) \subseteq S_\m(L_{\lambda'}) \xrightarrow{\psi} \GG_m(L_{\lambda'})=L_{\lambda'}^\times.
\] 
Now take any non-zero prime ideal $\lambda$ of $\OO_L$ dividing a prime $\ell$.   Take any non-zero prime $\p \nmid \ell \ell'$ for which $A$ has good reduction.    We have $\varepsilon_\ell(\Frob_\p) = \varepsilon_{\ell'}(\Frob_\p)$ in $S_\m(\QQ)$, so 
\[
\beta_{A,\lambda}(\Frob_\p) = F_\p = \beta_{A,\lambda'}(\Frob_\p)=\psi(\varepsilon_{\ell'}(\Frob_\p))=\psi(\varepsilon_{\ell}(\Frob_\p)).
\]
By the Chebotarev density theorem, we deduce that $\beta_{A,\lambda}$ is the composition of $\varepsilon_\ell$ with the homomorphism $S_\m(\QQ_{\ell}) \subseteq S_\m(L_{\lambda}) \xrightarrow{\psi} \GG_m(L_{\lambda})$.  We take $\phi$ to be the composition of the natural morphism $S_\m\to \Res_{L/\QQ}(S_{\m,L})$ with the morphism $\Res_{L/\QQ}(S_{\m,L})\to \Res_{L/\QQ}(\GG_{m,L})=T_L$ induced by $\psi$.  Since $\beta_{A,\ell}=\prod_\lambda \beta_{A,\lambda}$, one can now check that the lemma holds with this $\phi$.\end{proof}
 
We define 
\[
\varphi\colon T_K \to T_L
\]
to be the homomorphism obtained by composing the quotient map $T_K\to T_\m$ with the homomorphism $\phi$ in Lemma~\ref{L:locally algebraic}.

Take any prime $\ell$.  We have $T_K(\QQ_\ell)=(K\otimes_\QQ \QQ_\ell)^\times= \prod_{v|\ell} K_v^\times$ and $T_K(\QQ_\ell)_c=\prod_{v | \ell} \OO_v^\times$, where $T_K(\QQ_\ell)_c$ is the maximal compact subgroup of $T_K(\QQ_\ell)$ with respect to the $\ell$-adic topology.   Let $\pi_\ell\colon T_K(\QQ_\ell)\hookrightarrow I_K$ be the homomorphism that extends an element of $\prod_{v|\ell} K_v^\times$ to an idele by setting $1$ at the places of $K$ that do not divide $\ell$.  Since $\beta_{A,\ell}$ has abelian image, class field theory gives a homomorphism $I_K \to T_L(\QQ_\ell)$ corresponding to $\beta_{A,\ell}$; we will also denote it by $\beta_{A,\ell}$. 

\begin{prop} \label{P:U in Tc}
  There is an open subgroup $U\subseteq T_K(\QQ_\ell)_c$ with $[T_K(\QQ_\ell)_c:U]\ll_g 1$ such that 
\[
\beta_{A,\ell}(\pi_\ell(u))=\varphi(u)^{-1}
\] 
holds for all $u\in U$.   We can take $U=T_K(\QQ_\ell)_c$ when $\ell$ is not divisible by any non-zero prime ideal $\p\subseteq \OO_K$ for which $A$ has bad reduction.
\end{prop}
\begin{proof}
Define the group $G:=(\phi\circ \varepsilon)(\pi_\ell(T_K(\QQ_\ell)_c)) \subseteq T_L(\QQ)$.

Take any prime $\ell'\neq \ell$.   By Lemma~\ref{L:locally algebraic}, we have $\beta_{A,\ell'}=\phi\circ (\varepsilon \cdot \alpha_{\ell'}^{-1})$.  We have $\alpha_{\ell'}(\pi_\ell(T_K(\QQ_\ell)_c))=1$ since $\ell\neq \ell'$, so  \[
\beta_{A,\ell'}(\pi_\ell(T_K(\QQ_\ell)_c))= (\phi\circ \varepsilon)( \pi_\ell(T_K(\QQ_\ell)_c))=G.
\] 
By class field theory, $G$ is thus the group generated by $\beta_{A,\ell'}(I_\p)$ where $I_\p$ are inertia groups at primes $\p | \ell$.   By Lemma~\ref{L:potential good} and $\ell\neq \ell'$, we find that $G$ is finite and $|G|\leq [F:K] \ll_g 1$. 

There is thus an open subgroup $U\subseteq T_K(\QQ_\ell)_c$ with $(\phi\circ \varepsilon)(\pi_\ell(U))=1$ and $[T_K(\QQ_\ell)_c:U]=|G|\ll_g 1$.   For each $u\in U$, we have
\[
\beta_{A,\ell}(\pi_\ell(u))=(\phi\circ \varepsilon_\ell)(\pi_\ell(u))=(\phi\circ \varepsilon)(\pi_\ell(u))\cdot \phi(\alpha_\ell(\pi_\ell(u)))^{-1}= \phi(\alpha_\ell(\pi_\ell(u)))^{-1},
\]
where we have used Lemma~\ref{L:locally algebraic} in the first equality.  The first part of the lemma follows by noting that $\phi(\alpha_\ell(\pi_\ell(u)))=\varphi(u)$ for all $u\in T_K(\QQ_\ell)_c$.\\

Now suppose that $\ell$ is not divisible by any non-zero prime ideal $\p\subseteq \OO_K$ for which $A$ has bad reduction.     Take any prime $\ell'\neq \ell$.    By our choice of $\ell$, the representation $\rho_{A,\ell'}$, and hence also $\beta_{A,\ell'}$, is unramified at all primes that divide $\ell$.   By class field theory, this is equivalent to having $\beta_{A,\ell'}(\pi_1(T_K(\QQ_\ell)_c))=1$.  From the work above, this implies that 
$1=\beta_{A,\ell'}(\pi_\ell(T_K(\QQ_\ell)_c))= (\phi\circ \varepsilon)( \pi_\ell(T_K(\QQ_\ell)_c))=G$.  Therefore, $U=T_K(\QQ_\ell)_c$ since $[T_K(\QQ_\ell)_c:U]=|G|=1$.
\end{proof}

The following lemma, which we will prove in \S\ref{SS:X bases proof}, gives some additional information on $\varphi$.   

\begin{lemma} \label{L: X bases}
Set $d:=[K:\QQ]$ and $e:=[L:\QQ]$. 
There are bases $\alpha_1,\ldots, \alpha_{d}$ of $X(T_K)$ and $\gamma_1,\ldots, \gamma_{e}$ of $X(T_L)$ such that
\[
\gamma_j \circ \varphi = \prod_{i=1}^{d} \alpha_i^{n_{i,j}}
\]
holds for all $1\leq j \leq e$, where $n_{i,j}$ are integers satisfying $|n_{i,j}|\leq 2g$.   Moreover, the two bases are stable under the $\Gal_\QQ$-actions on $X(T_K)$ and $X(T_L)$.
\end{lemma}

\begin{lemma} \label{L:W cardinality}
Let $W$ be the kernel of $\varphi\colon T_K\to T_L$.   Then $W/W^\circ$ is a finite group scheme whose cardinality can be bounded in terms of $g$.
\end{lemma}
\begin{proof}
Set $d:=[K:\QQ]$ and $e:=[L:\QQ]$.   Define the homomorphism $\varphi^*\colon X(T_L)\to X(T_K)$, $\gamma\mapsto \gamma\circ \varphi$.  The group $X(W)$ is isomorphic to the cokernel of $\varphi^*$.  The cardinality $m$ of the finite group scheme $W/W^\circ$ is equal to the cardinality of the torsion subgroup of the cokernel of $\varphi^*$.

Let $A\in M_{d,e}(\ZZ)$ be a matrix that represents $\varphi^*$ with respect to the bases of $X(T_L)$ and $X(T_K)$ from Lemma~\ref{L: X bases}.   In particular, all the entries of $A$ have absolute value at most $2g$.   Let $D \in M_{d,e}(\ZZ)$ be the \emph{Smith Normal Form} of $A$.   There is an integer $ 0\leq k \leq \min\{d,e\}$ such that 
$D_{i,j}>0$ if $1\leq i=j \leq k$ and $D_{i,j}=0$ otherwise.  Moreover, the integer $D_{i,i}$ divides $D_{i+1,i+1}$ for all $1\leq i \leq k-1$.    Therefore, $m=\prod_{i=1}^k D_{i,i}$.

Note that $m=\prod_{i=1}^k D_{i,i}$ is equal to the greatest common divisor of all determinants of $k\times k$ minors of $A$, cf.~\cite{MR2544610}*{Proposition~8.1}.     Since the entries of $A$ are bounded in terms of $g$ and $k\leq e \leq 2g$, we conclude that $m\ll_g 1$.
\end{proof}

\subsection{Proof of Lemma~\ref{L: X bases}} \label{SS:X bases proof}

We first assume that $L$ is a field.  

For a prime $\ell$, we have an isomorphism $(T_K)_{\QQ_\ell} = \prod_{\p|\ell} \Res_{K_\p/\QQ_\ell}(\GG_{m,K_\p})$, where the product is over prime ideals $\p|\ell$ of $\OO_K$.  Similarly, we have $(T_L)_{\QQ_\ell} = \prod_{\lambda|\ell} \Res_{L_\lambda/\QQ_\ell}(\GG_{m,L_\lambda})$, where the product is over prime ideals $\lambda|\ell$ of $\OO_L$.

Now fix any prime $\ell\geq 5$ that splits completely in the number fields $K$ and $L$ and is not divisible by any prime for which $A$ has bad reduction.  For the rest of the proof $\p$ and $\lambda$ will denote prime ideals of $\OO_K$ and $\OO_L$, respectively, that divide $\ell$. 

The tori $(T_K)_{\QQ_\ell}$ and $(T_L)_{\QQ_\ell}$ are split since $\ell$ splits completely in $K$ and $L$.   In particular, we have isomorphisms $(T_K)_{\QQ_\ell} = \prod_{\p|\ell} \GG_{m,\QQ_\ell}$ and $(T_L)_{\QQ_\ell} = \prod_{\lambda|\ell} \GG_{m,\QQ_\ell}$.
Denote by $\alpha_\p \colon (T_K)_{\QQ_\ell} \to \GG_{m,\QQ_\ell}$ and $\gamma_\lambda \colon (T_L)_{\QQ_\ell} \to \GG_{m,\QQ_\ell}$ the character obtained by projecting onto the corresponding factor.   Note that $\{\alpha_\p\}_{\p|\ell}$ and $\{\gamma_\lambda\}_{\lambda|\ell}$ are bases of $X((T_K)_{\QQ_\ell})$ and $X((T_L)_{\QQ_\ell})$, respectively.   

The homomorphism $\varphi\colon (T_K)_{\QQ_\ell} \to (T_L)_{\QQ_\ell}$ is thus of the form
\[
\varphi\Big((x_\p)_{\p|\ell}\Big) = \Big(\prod_{\p|\ell} \,x_\p^{n_{\p,\lambda}} \Big)_{\lambda|\ell}
\]
for unique integers $n_{\p,\lambda}$.  In particular, we have
\begin{align}\label{E:gamma vs alpha 1}
\gamma_\lambda \circ \varphi = \prod_{\p|\ell} \alpha_\p^{n_{\p,\lambda}}
\end{align}
for all $\lambda|\ell$.   The following lemma gives some constraints on the integers $n_{\p,\lambda}$.

\begin{lemma} \label{L:mod ell-1}
Each integer $n_{\p,\lambda}$ is congruent modulo $\ell-1$ to an integer that has absolute value at most $2g$.
\end{lemma}
\begin{proof}
Take any prime ideal $\lambda|\ell$ of $\OO_L$.  We have defined a representation $\beta_{A,\lambda}\colon \Gal_K \to L_\lambda^\times$.  The image of $\beta_{A,\lambda}$ is contained in $\OO_\lambda$ since it is continuous and $\Gal_K$ is compact.   By class field theory, we may view $\beta_{A,\lambda}$ as a homomorphism $I_K\to \OO_\lambda^\times$.    By Proposition~\ref{P:U in Tc} and our choice of $\ell$, we have $\beta_{A,\ell}(\pi_\ell(u))=\varphi(u)^{-1}$ for all $u\in T_K(\QQ_\ell)_c$.    Therefore, $\beta_{A,\lambda}(\pi_\ell(u))=\gamma_\lambda(\varphi(u)^{-1})=\prod_{\p|\ell} \alpha_\p(u)^{-n_{\p,\lambda}}$ holds for all $u\in T_K(\QQ_\ell)_c$.

Now fix a prime ideal $\p|\ell$ of $\OO_K$.  Define the homomorphism 
\[
\widetilde{\beta}_{A,\lambda} \colon \OO_\p^\times \xrightarrow{i} I_K \xrightarrow{\beta_{A,\lambda}} \OO_\lambda^\times,
\]
where $i\colon \OO_\p^\times \hookrightarrow I_K$ is the inclusion on the $\p$-th term of $I_K$ and $1$ elsewhere.   We have 
\[
\widetilde{\beta}_{A,\lambda}(a)=a^{-n_{p,\lambda}}
\] 
for all $a\in \OO_\p^\times$; note that this does makes sense because $\ell$ splits completely in $K$ and $L$ and hence $\OO_\p^\times=\ZZ_\ell^\times$ and $\OO_\lambda^\times=\ZZ_\ell^\times$.   For an inertia subgroup $I_\p$ of $\Gal_K$ at the prime $\p$, class field theory now implies that $\beta_{A,\lambda}(\sigma)=\chi_\ell(\sigma)^{-n_{\p,\lambda}}$ holds for all $\sigma\in I_\p$, where $\chi_\ell\colon \Gal_K\to \ZZ_\ell^\times$ is the $\ell$-adic cyclotomic character.  By Lemma~\ref{L:tame inertia weights} and using that $\FF_\lambda^\times = \FF_\ell^\times$ is cyclic, we deduce that $-n_{\p,\lambda}$ is congruent modulo $\ell-1$ to an integer $0\leq b \leq 2g$.
\end{proof}

We now prove Lemma~\ref{L: X bases} (in the case where $L$ is a field).   Set $d:=[K:\QQ]$ and $e:=[L:\QQ]$. 

Let $\sigma_1,\ldots, \sigma_{d} \colon K \hookrightarrow \Qbar$ be the $d$ distinct embeddings of $K$.    Each $\sigma_i$ extends to a homomorphism $K\otimes_\QQ \Qbar \to \Qbar$ of $\Qbar$-algebras and defines a character $\alpha_i \in X(T_K)=\operatorname{Hom}((T_K)_{\Qbar},\GG_{m,\Qbar})$.   Observe that $X(T_K)$ is a free abelian group of rank $d$ with basis $\alpha_1,\ldots, \alpha_{d}$.  The natural $\Gal_\QQ$-action on $X(T_K)$ permutes the characters $\alpha_1,\ldots, \alpha_{d}$. 

Let $\tau_1,\ldots, \tau_{e} \colon L \hookrightarrow \Qbar$ be the $e$ distinct embeddings of $L$.  As above, each $\tau_i$ determines a character $\gamma_i \in X(T_L)$.   The group $X(T_L)$ is free abelian of rank $e$ with basis $\gamma_1,\ldots, \gamma_{e}$.  The natural $\Gal_\QQ$-action on $X(T_L)$ permutes the characters $\gamma_1,\ldots, \gamma_{e}$. 

We have $\varphi(T_K) \subseteq T_L$.  So for each $1\leq j \leq e$, we have
\begin{align} \label{E:gamma vs alpha 2}
\gamma_j \circ \varphi = \prod_{i=1}^d \alpha_i^{n_{i,j}}
\end{align}
for unique integers $n_{i,j}$. 

A fixed embedding $\Qbar \hookrightarrow \Qbar_\ell$ induces isomorphisms $X(T_K) = X((T_K)_{\QQ_\ell})$ and $X(T_L) = X((T_L)_{\QQ_\ell})$.  Observe that under these isomorphisms, we have $\{\alpha_\p: \p |\ell\} = \{\alpha_1,\ldots, \alpha_d\}$ and $\{\gamma_\lambda: \lambda |\ell\} = \{\gamma_1,\ldots, \gamma_e\}$.   By comparing (\ref{E:gamma vs alpha 1}) and (\ref{E:gamma vs alpha 2}), we deduce that each $n_{i,j}$ is equal to some $n_{\p,\lambda}$.  By Lemma~\ref{L:mod ell-1}, $n_{i,j}$ is congruent modulo $\ell-1$ to an integer that has absolute value at most $2g$.     

By the Chebotarev density theorem, there are infinitely many primes $\ell\geq 5$ that splits completely in the number fields $K$ and $L$ and are not divisible by any prime for which $A$ has bad reduction.   Since $n_{i,j}$ is congruent modulo $\ell-1$ to an integer with absolute value at most $2g$ for infinitely many $\ell$, we deduce that $|n_{i,j}|\leq 2g$.\\

We now consider the general case where $L$ need not be a field.  For each $1\leq i \leq s$, let $\varphi_i\colon T_K \to T_{L_i}$ be the homomorphism $\varphi$ from \S\ref{SS:Serre tori} for the abelian variety $A_i$.    Observe that $\varphi \colon T_K\to T_L =\prod_{i=1}^s T_{L_i}$ is given by $(\varphi_1,\ldots, \varphi_s)$; one way to show this is to note the Proposition~\ref{P:U in Tc} characterizes $\varphi$ and that $\beta_{A,\ell}=\prod_{i=1}^s \beta_{A_i,\ell}$.    

From the case of Lemma~\ref{L: X bases} we have already proved, we find that there is a basis $\alpha_1,\ldots, \alpha_d$ of $X(T_K)$ such that for all $1\leq i \leq s$, we have $\gamma_{i,j} \circ \varphi_i = \prod_{k=1}^d \alpha_k^{n_{i,j,k}}$, where $\gamma_{i,1},\ldots, \gamma_{i,[L_i:\QQ]}$ is a basis of $X(T_{L_i})$ and $n_{i,j,k}$ is an integer with absolute value at most $2g$.    Moreover, the bases $\alpha_1,\ldots, \alpha_d$ and $\gamma_{i,1},\ldots, \gamma_{i,[L_i:\QQ]}$ are stable under the natural $\Gal_\QQ$-action.  Lemma~\ref{L: X bases} now follows immediately with the basis $\alpha_1,\ldots, \alpha_d$ for $X(T)$ and $\{ \gamma_{i,j}: 1\leq i \leq s, \, 1\leq j \leq [L_i:\QQ]\}$ for $X(T_L)=\oplus_{i=1}^s X(T_{L_i})$.

\subsection{Proof of Theorem~\ref{T:main abelian}} \label{SS:proof of main abelian}

\begin{lemma}\label{L:initial abelian index}
\begin{romanenum}
\item
We have $\varphi(T_K)=Y$.
\item
For any prime $\ell$, we have $[Y(\QQ_\ell)_c \colon\beta_{A,\ell}(\Gal_K)] \ll_g [Y(\QQ_\ell)_c:\varphi(T_K(\QQ_\ell)_c)]$.
\end{romanenum}
\end{lemma}
\begin{proof}
Take the open subgroup $U\subseteq T_K(\QQ_\ell)_c$ as in Proposition~\ref{P:U in Tc}.  We have 
\[
\beta_{A,\ell}(\pi_\ell(U))=\varphi(U)^{-1}=\varphi(U).
\] 
The set $U$ is Zariski dense in $(T_K)_{\QQ_\ell}$ and $\varphi(U) = \beta_{A,\ell}(\pi_\ell(U)) \subseteq Y(\QQ_\ell)$, so $\varphi((T_K)_{\QQ_\ell}) \subseteq Y_{\QQ_\ell}$.  Therefore, $\varphi(T_K)\subseteq Y$.

We now prove $\varphi(T_K)= Y$.  
Let $\psi\colon \Gal_K \to Y(\QQ_\ell)/\varphi(T_K)(\QQ_\ell)$ be the homomorphism obtained by composing $\beta_{A,\ell}$ with the obvious quotient map.    By Lemma~\ref {L:potential good}, there is a finite extension $F/K$ such that $\psi|_{\Gal_F}$ is unramified at all prime ideals $\p\nmid \ell$ of $\OO_F$.    Since $\varphi(U)\subseteq \varphi(T_K)(\QQ_\ell)$ and $[T_K(\QQ_\ell)_c: U]\ll_g 1$, we deduce, after possibly replacing $F$ by a finite extension, that $\psi|_{\Gal_F}$ is unramified at all prime ideals $\p$ of $\OO_F$.    Therefore, $\psi$ has finite image and hence $\psi|_{\Gal_F}=1$ for some finite extension $F/K$.     For such a finite extension $F/K$, we have $\beta_{A,\ell}(\Gal_F) \subseteq \varphi(T_K)(\QQ_\ell)$.  The group $\beta_{A,\ell}(\Gal_F)$ is Zariski dense in $Y_{\QQ_\ell}$ by Proposition~\ref{P:Zariski dense in Y} and using that $Y$ is connected.   So $\varphi(T_K)(\QQ_\ell)$ is Zariski dense in $Y_{\QQ_\ell}$.   Therefore, $\varphi(T_K)_{\QQ_\ell}=Y_{\QQ_\ell}$ and hence $\varphi(T_K)= Y$ as desired.

Finally, we have
\[
[Y(\QQ_\ell)_c \colon\beta_{A,\ell}(\Gal_K)]
\leq 
[Y(\QQ_\ell)_c \colon\beta_{A,\ell}(\pi_\ell(U))]
=
[Y(\QQ_\ell)_c \colon\varphi(U)].
\ll_g [Y(\QQ_\ell)_c \colon\varphi(T(\QQ_\ell)_c)],
\]
where the last inequality uses that $[T(\QQ_\ell)_c:U]\ll_g 1$.
\end{proof}

Take any prime $\ell$.  Let $\rho\colon \Gal_{\QQ_\ell}\to \Aut_\ZZ(X( (T_K)_{\QQ_\ell}))$ be the Galois action on the character group of the torus $(T_K)_{\QQ_\ell}$.   Let $F'$ be the subfield of $\Qbar_\ell$ fixed by $\ker \rho$.   Let $F/\QQ_\ell$ be any subfield of $F'$ for which the extension $F/\QQ_\ell$ is Galois and the extension $F'/F$ is unramified.  Define the integer 
\[
e:=[F:\QQ_\ell]\cdot [Y(F)_c\colon \varphi(T_K(F)_c)],
\]
where $Y(F)_c$ and $T_K(F)_c$ are the maximal compact subgroups of $Y(F)$ and $T_K(F)$, respectively, with respect to the $\m$-adic topology.

\begin{lemma} \label{L:new e power}
We have $y^{e} \in \varphi(T_K(\QQ_\ell)_c)$ for each $y\in Y(\QQ_\ell)_c$.
\end{lemma}
\begin{proof}
Take any $y\in Y(\QQ_\ell)_c$.  We have $y^n = \varphi(t)$ for some $t\in T_K(F)_c$, where $n:=[Y(F)_c:\varphi(T_K(F)_c)]$.   Since $y$ and $\varphi$ are defined over $\QQ_\ell$, we have $y^n=\varphi(\sigma(t))$ for all $\sigma\in \Gal(F/\QQ_\ell)$.   Therefore, $y^e=y^{[F:\QQ] n} = \varphi(t')$ for $t' := \prod_{\sigma\in \Gal(F/\QQ_\ell)} \sigma(t)$.  We have $t'\in T_K(\QQ_\ell)$ since it stable under the $\Gal(F/\QQ_\ell)$-action. 

Each $\sigma \in \Gal(F/\QQ_\ell)$ is a continuous automorphism of $F$ and hence $\sigma(t)\in T_K(F)_c$.   Therefore, $t' \in T_K(F)_c\cap T_K(\QQ_\ell) =T_K(\QQ_\ell)_c$.
\end{proof}

\begin{lemma} \label{L:new ramified case}
We have $[Y(\QQ_\ell)_c: \varphi(T(\QQ_\ell)_c)]\leq [Y(\QQ_\ell): \gamma(Y(\QQ_\ell))]$, where $\gamma\colon Y \to Y$ is the $e$-th power map.
\end{lemma}
\begin{proof}
Let $\gamma\colon Y \to Y$ be the $e$-th power map; it is an isogeny.   We have a quotient homomorphism
\begin{align} \label{E:injectivity compact}
Y(\QQ_\ell)_c/\gamma(Y(\QQ_\ell)_c) \to Y(\QQ_\ell)/\gamma(Y(\QQ_\ell)).
\end{align}

We claim that (\ref{E:injectivity compact}) is injective.   Take any $y\in Y(\QQ_\ell)_c$ for which $y=x^e$ for some $x\in Y(\QQ_\ell)$.    Let $G$ be the group generated by $x$ and $Y(\QQ_\ell)_c$.  Since $x^e\in Y(\QQ_\ell)_c$, we find that $Y(\QQ_\ell)_c$ is a finite index subgroup of $G$ and hence $G$ is compact.   We have $G=Y(\QQ_\ell)_c$ since $Y(\QQ_\ell)_c$ is the maximal compact subgroup of $Y(\QQ_\ell)$. Therefore, $x\in Y(\QQ_\ell)_c$ which finishes the proof of the claim.

By the injectivity of (\ref{E:injectivity compact}), we have  $[Y(\QQ_\ell)_c: \gamma(Y(\QQ_\ell)_c)] \leq [Y(\QQ_\ell): \gamma(Y(\QQ_\ell))]$.  The lemma is now a consequence of Lemma~\ref{L:new e power} which says that $\varphi(T(\QQ_\ell)_c) \supseteq \gamma(Y(\QQ_\ell)_c)$.
\end{proof}

Let $\gamma\colon Y\to Y$ be the $e$-th power map.  The map $\gamma$ is an isogeny and hence $Z:= \ker \gamma$ is a finite group scheme.    Starting with the short exact sequence $1\to Z(\Qbar_\ell)\to Y(\Qbar_\ell) \xrightarrow{\gamma} Y(\Qbar_\ell) \to 1$ and taking Galois cohomology gives an injective homomorphism $Y(\QQ_\ell)/\gamma(Y(\QQ_\ell)) \hookrightarrow H^1(\Gal_{\QQ_\ell}, Z(\Qbar_\ell))$.  This injective homomorphism and Lemma~\ref{L:new ramified case} implies that
\begin{align} \label{E:H1 bound}
[Y(\QQ_\ell)_c : \varphi(T(\QQ_\ell)_c)] \leq |H^1(\Gal_{\QQ_\ell},Z(\Qbar_\ell))|.
\end{align}

\begin{lemma} \label{L:H1 finite}
Let $H$ be a finite abelian group with a $\Gal_{\QQ_\ell}$-action.  Then the cardinality of $H^1(\Gal_{\QQ_\ell},H)$ can be bounded in terms of $|H|$.
\end{lemma}
\begin{proof}
Set $n=|H|$ and $G:=\Gal_{\QQ_\ell}$.   There is an open normal subgroup $N\subseteq G$ of index at most $n!$ that acts trivially on $H$.  We have an inflation-restriction exact sequence
\[
0\to H^1(G/N, H^N) \to H^1(G,H) \to H^1(N,H)
\]
The cardinality of $H^1(G/N, H^N)$ can be bounded in terms of $|G/N|\leq n!$ and $|H^N|\leq n$.  So it suffices to bound $H^1(N,H)$ which is the group of continuous homomorphisms $N\to H$ since $N$ acts trivially on $H$.

Let $K$ be the extension of $\QQ_\ell$ corresponding to the subgroup $N$ of $G$.   By local class field theory, we can idenitify $H^1(N,H)$ with $\text{Hom}(K^\times/(K^\times)^n, H)$.  The lemma follows by noting that the cardinality of $K^\times/(K^\times)^n$ can be bounded in terms of $n$ and $[K:\QQ_\ell] \leq n!$.
\end{proof}

 The group scheme $Z$ is finite with cardinality $e^{\dim Y}$.   Since $\dim Y\leq 2g$, Lemma~\ref{L:H1 finite} and (\ref{E:H1 bound}) imply that $[Y(\QQ_\ell)_c : \varphi(T(\QQ_\ell)_c)] \ll_{g,e} 1$.   

\begin{lemma} \label{L:F bound}
We have $[Y(F)_c\colon \varphi(T_K(F)_c)]\ll_{g} 1$.
\end{lemma}
\begin{proof}
Let $R$ be the ring of integers of $F$ and denote its maximal ideal by $\m$.    Define the residue field $\FF=R/\m$ and denote its cardinality by $q$.
Let $F^{\un}$ be the maximal unramified extension of $F$ in $\bbar{F}$.

We now consider algebraic group schemes of multiplicative type; for background, see \cite{MR3362641}*{Appendix B}.  The category of algebraic group schemes over $R$ of multiplicative type is equivalent to the category of algebraic group schemes over $F$ of multiplicative type for which the action of $\Gal_F$ on the character group is unramified.  This can be seen by noting that both are anti-equivalent to the category of discrete $\Gal(F^\un/F)$-modules that are finitely generated abelian groups, see \cite{MR3362641}*{Corollary B.3.6}. Explicitly, the equivalence is given by base extension by $F$.

Let $\calH$ be a torus over $R$.   The group $\calH(F)$ has a natural $\m$-adic topology.   Observe that $\calH(R)$ agrees with the maximal compact subgroup $\calH(F)_c$ of $\calH(F)$ with respect to the $\m$-adic topology (one can prove this by extending $R$ to reduce to the split case).     One can use the smoothness of $\calH$ and Hensel's lemma to show that $\calH(R/\m^{n+1})\to \calH(R/\m^n)$ is surjective with kernel isomorphic to  $\FF^{\dim \calH_F}$; in particular, the kernel has cardinality $q^{\dim \calH_F}$.  Therefore, $|\calH(R/\m^n)|= |\calH(\FF)|\cdot q^{(n-1) \dim \calH_F}$ for all $n\geq 1$.\\

Define $W:=\ker(\varphi) \subseteq T_K$ and $B:=W/W^\circ$.  The group scheme $B$ is finite and denote it cardinality by $m$.   By Lemma~\ref{L:W cardinality}, we have $m\ll_g 1$.

By our choice of $F$, the action of $\Gal_F$ on $X((T_K)_F))$ is unramified, i.e., factors through $\Gal(F^\un/F)$.  The actions of $\Gal_F$ on $X(W_F)$ and $X(Y_F)$ are also unramified since they can be viewed as a quotient and subgroup, respectively, of $X((T_K)_F)$ stable under the $\Gal_F$ action.     So there are short exact sequences 
\[
1\to \calW \xrightarrow{\iota} \calT \xrightarrow{\psi} \calY \to 1 \quad\quad \text{ and }\quad \quad
1\to \calW_0 \to \calW \to \calB \to 1
\]
of $R$-groups of multiplicative type such that base extension by $F$ gives rise to the short exact sequences
\[
1 \to W_F \hookrightarrow (T_K)_F \xrightarrow{\varphi} Y_F \to 1  \quad\quad \text{ and }\quad \quad
1\to W_F^\circ \hookrightarrow W_F \to B_F \to 1.
\]
We have $\calY(R)=Y(F)_c$ and $\calT(R)=T_K(F)_c$, and hence $[\calY(R):\psi(\calT(R))]=[Y(F)_c:\varphi(T_K(F)_c)]$.  So it suffices to prove that $[\calY(R):\psi(\calT(R))]\ll_{g} 1$.   Since $\psi(\calT(R))$ is a closed subgroup of $\calY(R)$ in the $\m$-adic topology, it suffices to prove that $[\calY(R/\m^n):\psi(\calT(R/\m^n))]\ll_{g} 1$ holds for all $n\geq 1$.  \\

First suppose that $n>1$.  We have $|\calY(R/\m^n)|=|\calY(\FF)| \cdot q^{(n-1)\dim Y}$ and 
\begin{align*}
|\psi(\calT(R/\m^n))| &= \frac{|\calT(R/m^n)|}{ |\calW(R/m^n)|} \\
& \geq \frac{1}{m}  \frac{|\calT(R/m^n)|}{ |\calW_0(R/m^n)|}\\
& \gg_g \frac{|\calT(\FF)|\cdot q^{(n-1) \dim T_K}}{ |\calW_0(\FF)|\cdot q^{(n-1) \dim W_0}} \\
&\geq \frac{|\calT(\FF)|}{ |\calW(\FF)|} \cdot q^{(n-1) (\dim T_K-\dim W)}.
\end{align*}
Using that $\dim Y = \dim T_K -\dim W$, we have
\[
[\calY(R/\m^n):\psi(\calT(R/\m^n))]
\ll_g |\calY(\FF)|/(|\calT(\FF)|/|\calW(\FF)|) = [\calY(\FF): \psi(\calT(\FF))].
\]
So it suffices to prove the $n=1$ case, i.e., show that $[\calY(\FF): \psi(\calT(\FF))]\ll_{g} 1$.

From the short exact sequence $1\to \calW(\FFbar)\to \calT(\FFbar) \xrightarrow{\psi} \calY(\FFbar) \to 1$, taking Galois cohomology gives an injective homomorphism $\calY(\FF)/\psi(\calT(\FF)) \hookrightarrow H^1(\Gal_{\FF}, \calW(\FFbar))$.  From the short exact sequence $1\to \calW_0(\FFbar)\to \calW(\FFbar) \to \calB(\FFbar) \to 1$, we have an exact sequence $H^1(\Gal_\FF, \calW_0(\FFbar))\to H^1(\Gal_\FF, \calW(\FFbar)) \to H^1(\Gal_\FF, \calB(\FFbar))$.  Since $(\calW_0)_{\FF}$ is connected,  Lang's theorem implies that $H^1(\Gal_\FF, \calW_0(\FFbar))=0$ and hence
\[
[\calY(\FF):\psi(\calT(\FF))] \leq |H^1(\Gal_{\FF}, \calW(\FFbar))| \leq |H^1(\Gal_\FF, \calB(\FFbar))|.
\]
Since $\calB(\FFbar)$ is a finite group of cardinality at most $m\ll_g 1$ and $\Gal_\FF$ is pro-cyclic, we have $|H^1(\Gal_\FF, \calB(\FFbar))| \ll_g 1$ and hence $[\calY(\FF):\psi(\calT(\FF))]\ll_g 1$.
\end{proof}

Since  $[Y(\QQ_\ell)_c : \varphi(T(\QQ_\ell)_c)] \ll_{g,e} 1$, Lemma~\ref{L:F bound} implies that $[Y(\QQ_\ell)_c : \varphi(T(\QQ_\ell)_c)]\ll_{g,[F:\QQ_\ell]} 1$.  If $\ell$ is unramified in $K$, then we can choose $F=\QQ_\ell$ and hence $[Y(\QQ_\ell) : \varphi(T(\QQ_\ell)_c)] \ll_{g} 1$.  This proves part (\ref{T:main abelian i}).

To prove part (\ref{T:main abelian ii}), it suffices to show that $[F:\QQ_\ell] \ll_{[K:\QQ]} 1$.   By Lemma~\ref{L: X bases}, there is a basis $\alpha_1,\ldots, \alpha_{[K:\QQ]}$ of $X( (T_K)_{\QQ_\ell})$ that is permuted by the natural $\Gal_{\QQ_\ell}$-action.  Therefore, $[F:\QQ_\ell] \leq [K:\QQ]!$.

\section{Proof of Theorem~\ref{T:main new}(\ref{T:main new a}) and (\ref{T:main new b})}
\label{S:end of main proof}

By Lemma~\ref{L:reduction to connected case}, we may assume the groups $G_{A,\ell}$ are all connected.  

Fix a prime $\ell\geq c \cdot \max(\{[K:\QQ],h(A), N(\q)\})^\gamma$.  We have already proved part (\ref{T:main new c}) of Theorem~\ref{T:main new} in \S\ref{S:proof of c and d},  so by increasing the constants $c$ and $\gamma$ appropriately, we may assume that $\ZZ_\ell$-group scheme $\calG_{A,\ell}$ is reductive.   Let $\calS_{A,\ell}$ be the derived subgroup of $\calG_{A,\ell}$; it is a semisimple group scheme over $\ZZ_\ell$.    We have  proved part (\ref{T:main new d}) of Theorem~\ref{T:main new} in \S\ref{S:proof of c and d}, so by increasing the constants $c$ and $\gamma$ appropriately, we may also assume that $\rho_{A,\ell}(\Gal_K)$ contains the commutator subgroup  $\calG_{A,\ell}(\ZZ_\ell)'$ of $\calG_{A,\ell}(\ZZ_\ell)$.   In particular, $\rho_{A,\ell}(\Gal_K)\supseteq\calS_{A,\ell}(\ZZ_\ell)'$.  

Define $S=(\calS_{A,\ell})_{\QQ_\ell}$; it is the derived subgroup of the connected reductive group $G_{A,\ell}$.   

With notation as in \S\ref{S:abelian setup} and Proposition~\ref{P:Zariski dense in Y}, there is a 
homomorphism 
\[
\delta:={\det}_L \colon G_{A,\ell} \to Y_{\QQ_\ell}
\]
of algebraic groups over $\QQ_\ell$, where $Y$ is a torus defined over $\QQ$.  Define $H:=\ker(\delta)$.

\begin{lemma} \label{L:H conn}
We have $H^\circ = S$ and the cardinality of the group scheme $H/S$ can be bounded in terms of $g$.
\end{lemma}
\begin{proof}
We have $H \supseteq S$ since $S$ is semisimple and $Y_{\QQ_\ell}$ is a torus.  It thus suffices to show that the kernel of $\delta|_C=\det_L|_C$ is finite with cardinality bounded in terms of $g$, where $C$ is the central torus of $G_{A,\ell}$.  By Proposition~\ref{P:same centers},  $C=(C_{A})_{\QQ_\ell}$ where $C_A$ is the central torus of $\MT_A$.  As noted in \S\ref{S:abelian setup}, the homomorphism $\det_L|_{C_A}\colon C_A \to Y$ is an isogeny of degree $d_1\cdots d_s \ll_g 1$.
\end{proof}

Define the homomorphism $\beta_{A,\ell}\colon \Gal_K \to Y(\QQ_\ell)$ by $\beta_{A,\ell}= \det_L\circ\rho_{A,\ell}$.  We have $\beta_{A,\ell}(\Gal_K)\subseteq Y(\QQ_\ell)_c$, where $Y(\QQ_\ell)_c$ is the maximal compact subgroup of $Y(\QQ_\ell)$ with respect to the $\ell$-adic topology. 

We have inequalities
\begin{align*}
[\calG_{A,\ell}(\ZZ_\ell):\rho_{A,\ell}(\Gal_K)]
&=\,[{\det}_L(\calG_{A,\ell}(\ZZ_\ell))\colon {\det}_L(\rho_{A,\ell}(\Gal_K))]
\cdot [ H(\QQ_\ell)\cap \calG_{A,\ell}(\ZZ_\ell) : H(\QQ_\ell) \cap \rho_{A,\ell}(\Gal_K)]\\
&\leq \,[Y(\QQ_\ell)_c \colon\beta_{A,\ell}(\Gal_K)]
\cdot [ H(\QQ_\ell)\cap \calG_{A,\ell}(\ZZ_\ell) : H(\QQ_\ell) \cap \rho_{A,\ell}(\Gal_K)]\\
&\ll_g \,[Y(\QQ_\ell)_c \colon\beta_{A,\ell}(\Gal_K)]
\cdot [ S(\QQ_\ell)\cap \calG_{A,\ell}(\ZZ_\ell) : S(\QQ_\ell) \cap \rho_{A,\ell}(\Gal_K)],
\end{align*}
where we have used Lemma~\ref{L:H conn} in the last inequality and we have also used that ${\det}_L(\calG_{A,\ell}(\ZZ_\ell))$ is a compact subgroup of $Y(\QQ_\ell)$.     We have
\begin{align*} 
 [ S(\QQ_\ell)\cap \calG_{A,\ell}(\ZZ_\ell) : S(\QQ_\ell) \cap \rho_{A,\ell}(\Gal_K)] &= [ \calS_{A,\ell}(\ZZ_\ell) : \calS_{A,\ell}(\ZZ_\ell) \cap \rho_{A,\ell}(\Gal_K)]\\
 &\leq  [ \calS_{A,\ell}(\ZZ_\ell):\calS_{A,\ell}(\ZZ_\ell)']\\
 &\ll_g 1,
 \end{align*}
where the last inequality uses Proposition~\ref{P:calS}(\ref{P:calS iii}).   Combining our inequalities, we find that
\begin{align*}
[\calG_{A,\ell}(\ZZ_\ell):\rho_{A,\ell}(\Gal_K)] \ll_g  [Y(\QQ_\ell)_c \colon\beta_{A,\ell}(\Gal_K)].
\end{align*}
Theorem~\ref{T:main new}(\ref{T:main new a}) and (\ref{T:main new b}) now follow from Theorem~\ref{T:main abelian}.

\section{Proof of Theorem~\ref{T:GRH bound for Nq}}  \label{S:GRH bound for Nq}

We first give a slight generalization of Theorem~\ref{T:main new}.  

\begin{thm} \label{T:main new revised}
The conclusion of Theorem~\ref{T:main new} holds with (\ref{E:main ell bound}) replaced by the assumption that $\ell \geq c \cdot \max(\{[K:\QQ],h(A)\})^\gamma$  and $\ell \nmid n$, where $n$ is a positive integer satisfying $n < c N(\q)^\gamma$ that depends on $A$.
\end{thm}
\begin{proof}
There are only two times in the proof of Theorem~\ref{T:main new} that we used the assumption  
\begin{align} \label{E:old assumption}
\ell\geq c\, N(\q)^\gamma.
\end{align}
In \S\ref{SS:new reductive proof}, we used (\ref{E:old assumption}) to prove that $\q\nmid \ell$ and that $\ell$ does not divide a specific non-zero integer $D$ that satisfies $|D| <c\, N(\q)^\gamma$.  In \S\ref{SS:same rank proof}, we used (\ref{E:old assumption}) to prove that $\q\nmid \ell$ and that $\ell$ does not divide  non-zero integers $\beta_M$ satisfying $|\beta_M| <c\, N(\q)^\gamma$, where $M$ lie in a set $\calM$ with $|\calM|\ll_g 1$.

The theorem thus holds with $n:= N(\q) \cdot |D| \cdot \prod_{M\in \calM} |\beta_M |$ after possibly increasing the constants $c$ and $\gamma$ so that $n<cN(\q)^\gamma$.
\end{proof}

By the same arguments as in the proof of Lemma~\ref{L:reduction to connected case}, we may assume that all the groups $G_{A,\ell}$ are connected; note that the integer $D$ is unchanged if we replace $A/K$ with $A_{K_A^\conn}/K_A^\conn$.

\begin{lemma} \label{L:GRH Bell new}
Let $\ell$ be a prime for which $\calG_{A,\ell}$ is reductive.   For each maximal torus $T$ of $G$, there is a subset $U_T \subseteq T(\FF_\ell)$ such that the following conditions hold:
\begin{alphenum}
\item \label{I:Bell a}
Let $\p\nmid\ell$ be a non-zero prime ideal of $\OO_K$ for which $A$ has good reduction.  If $\bbar\rho_{A,\ell}(\Frob_\p)$ is conjugate in $G(\FF_\ell)$ to an element of $T(\FF_\ell)-U_T$, then $\Phi_{A,\p} \cong \ZZ^r$.
\item \label{I:Bell b}
We have $|U_T| \ll_g \ell^{r-1}$.
\item \label{I:Bell c}
For any $h \in G(\FF_\ell)$ and maximal torus $T$ of $G$, we have $U_{hTh^{-1}}=hU_T h^{-1}$.
\end{alphenum}
\end{lemma}
\begin{proof}
Fix a maximal torus $T$ of $G$.  Let $\alpha_1,\ldots, \alpha_{2g} \in X(T)$ be the weights of $T\subseteq \GL_{A[\ell]}$, with multiplicity, acting on $A[\ell]$.    Let $M\subseteq \ZZ^{2g}$ be the group consisting of $e\in \ZZ^{2g}$ satisfying $\prod_{i=1}^{2g}\alpha_i^{e_i}=1$.    By Theorem~\ref{T:equal groups} and Lemma~\ref{L:new bounded formal character}, there are only finite many possibilities for $M$ in terms of $g$.
 Note that we have an isomorphism $X(T)\cong \ZZ^{2g}/ M$; in particular, $\ZZ^{2g}/M$ is a free abelian group of rank $r$.      Define the finite set \[
\scrA:=\{m\in \ZZ^{2g} - M : \max_i |m_i| \leq C\},
\]
where $C$ is a positive constant depending only on $g$ that we will later impose an additional condition on.  For each $m\in \ZZ^{2g}$, define the character  $\beta_m:=\prod_{i=1}^{2g} \alpha_i^{m_i} \in X(T)$.   We have $\beta_m\neq 1$ for each $m\in \scrA$ since $m\notin M$.    Define $Y:= \cup_{m\in \scrA} \ker \beta_m$.  The set of characters $\{\beta_m: m \in \scrA\}\subseteq X(T)$ is stable under the action of $\Gal_{\FF_\ell}$ since $\alpha_1,\ldots,\alpha_{2g}$ is stable under this Galois action.  We may thus view $Y$ as a subvariety of $T$ defined over $\FF_\ell$.    Define the subset $U_T:=Y(\FF_\ell)$ of $T(\FF_\ell)$.    Note that while $M$ and $\scrA$ depend on our choice of ordering $\alpha_1,\ldots, \alpha_{2g}$ of weights, the set $U_T$ does not.

We now prove (\ref{I:Bell a}).  Take any prime $\p\nmid \ell$ for which $A$ has good reduction and $\bbar\rho_{A,\ell}(\Frob_\p)$ is conjugate in $G(\FF_\ell)$ to an element of $T(\FF_\ell)-U_T$.   We may assume that $t_\p:=\bbar\rho_{A,\ell}(\Frob_\p)$ lies in $T(\FF_\ell)-U$.   The roots of $P_{A,\p}(x)$ modulo $\ell$, with multiplicity, are $\alpha_1(t_\p),\ldots, \alpha_{2g}(t_\p) \in \FFbar_\ell$.   Let $R$ be the ring of integers of a splitting field $F/\QQ_\ell$ of $P_{A,\p}(x)$.    Let $\pi_1,\ldots,\pi_{2g} \in R$ be the roots of $P_{A,\p}(x)$ with multiplicity;  note that each $\pi_i$ lies in $R$ since it is an algebraic integer in $F$.  Let $\FF$ be the residue field of $R$.   Let $\bbar\pi_1,\ldots, \bbar\pi_{2g} \in \FFbar_\ell$ be the values obtained by reducing each $\pi_i$ and then applying a fixed embedding $\FF\hookrightarrow \FFbar_\ell$.  By  rearranging the $\pi_i$, we may assume that $\bbar{\pi}_i = \alpha_i(t_\p)$ holds for all $1\leq i \leq 2g$.  Let $M_\p$ be the group of $e\in \ZZ^{2g}$ that satisfy $\prod_{i=1}^{2g} \pi_i^{e_i}=1$. By Proposition~\ref{P:Frob group possibilities}, $M_\p$ is one of a finite number of subgroups of $\ZZ^{2g}$ that depend on $g$.

We claim that $M_\p \subseteq M$.   Suppose to the contrary that there is an $m \in M_\p - M$.    We may assume that $\max_i |m_i| \leq C$, where $C$ is our constant depending only on $g$; we can take $m$ to be in a fixed finite set of generators for each of the groups $M_i$ from Proposition~\ref{P:Frob group possibilities}.  We have $(\prod_{i=1}^{2g}\alpha^{m_i})(t_\p)=1$ since $m\in M_\p$.  In particular, by our choice of $C$, there is an $m\in \scrA$ such that $\beta_m(t_\p)=1$.   Therefore, $t_\p \in T(\FF_\ell)\cap Y(\FFbar_\ell) = Y(\FF_\ell)=U_T$.   However, this is a contradiction since we assumed that $t_\p \in T(\FF_\ell)-U_T$. This proves the claim.

Since $M_\p \subseteq M$, we have a surjective homomorphism $\Phi_{A,\p}\cong \ZZ^{2g}/ M_\p \to \ZZ^{2g}/ M \cong \ZZ^r$. Let $T_\p$ be the Zariski closure in $G_{A,\ell}$ of the subgroup generated by the semisimple element $\rho_{A,\ell}(\Frob_\p)$.    As explained in the proof of Lemma~\ref{L:pre common rank}, we have $X(T_\p) \cong \Phi_{A,\p}$.    We thus have $X(T_\p^\circ) = \Phi_{A,\p}/(\Phi_{A,\p})_{\operatorname{tors}}$, where the neutral component $T_\p^\circ$ is a torus and $(\Phi_{A,\p})_{\operatorname{tors}}$ is the torsion subgroup of $\Phi_{A,\p}$.   Since there is a surjective homomorphism $\Phi_{A,\p} \twoheadrightarrow \ZZ^r$, the torus $T_\p^\circ$ has dimension at least $r$.    Since $T_\p^\circ$ is contained in the reductive group $G_{A,\ell}$ of rank $r$, we find that $T_\p^\circ$ is a maximal torus of $G_{A,\ell}$ and has dimension $r$.    We have $T_\p=T_\p^\circ$ since a maximal torus of a connected reductive group is its own centralizer.    So $\Phi_{A,\p}\cong X(T_\p)$ is a free abelian group of rank $r$.  This completes the proof of (\ref{I:Bell a}). 

We now prove (\ref{I:Bell b}).   Since $|\scrA|\ll_g 1$, to verify $|U_T|\ll_g \ell^{r-1}$ it suffices to prove that $|\{t\in T(\FF_\ell): \beta_m(t)=1\}|\ll_g \ell^{r-1}$ for each $m\in \scrA$.     Take any $m\in \scrA$.    Let $\FF_{\ell^d}/\FF_\ell$ be the smallest extension over which $T$ splits; the character $\beta_m$ is defined over $\FF_{\ell^d}$.   We have $d\ll_g 1$.     For $t\in T(\FF_\ell)$ satisfying $\beta_m(t)=1$,  we have $\sigma(\beta_m)(t)=\sigma(\beta_m(t))=\sigma(1)=1$ for all $\sigma\in \Gal(\FF_{\ell^d}/\FF_\ell)$.   Define $W:=\bigcap_{\sigma \in \Gal(\FF_{\ell^d}/\FF_\ell)} \ker \sigma(\beta_m)$; it is a subvariety of $T$ defined over $\FF_\ell$ that contains all $t\in T(\FF_\ell)$ satisfying $\beta_m(t)=1$.   So to verify $|U_T|\ll_g \ell^{r-1}$ it suffices to prove that $|W(\FF_\ell)|\ll_g \ell^{r-1}$.   For any $\sigma\in \Gal(\FF_{\ell^d}/\FF_\ell)$, $\sigma$ permutes the characters $\alpha_1,\ldots, \alpha_{2g}$ (with multiplicity) and hence there is an $m_\sigma \in \scrA$ satisfying $\sigma(\beta_m)=\beta_{m_\sigma}$.  The algebraic group $W$ is diagonalizable (over $\FF_{\ell^d}$) and 
\[
X(W) \cong \ZZ^{2g}/\Big(M + \sum_{\sigma\in \Gal(\FF_{\ell^d}/\FF_\ell)}\ZZ  m_\sigma\Big)
\]
There are only finitely many possibilities (in terms of $g$) for the group $X(W)$ since  $d\ll_g 1$ and since there are only finitely many possibilities (in terms of $g$) for $M$ and each $m_\sigma$.  In particular, the torsion subgroup of $X(W)$ can be bounded in terms of $g$ and hence $|(W/W^\circ)(\FFbar_\ell)|\ll_g 1$.   Therefore, $|W(\FF_\ell)|\ll_g |W^\circ (\FF_\ell)|  \ll_g \ell^{r-1}$, where the last inequality uses that $W^\circ$ is a torus over $\FF_\ell$ of rank at most $r-1$ and $r$ can be bounded in terms of $g$.  We deduce that $|U_T|\ll_g \ell^{r-1}$.

It remains to prove (\ref{I:Bell c}).   Take any $h \in G(\FF_\ell)$ and define the maximal torus $T':=hTh^{-1}$ of $G$.  We have an isomorphism $\iota\colon T'\to T$, $t\mapsto h^{-1}th$ of tori and an isomorphism $X(T)\to X(T')$, $\alpha\mapsto \alpha\circ \iota$ of groups that respects the $\Gal_{\FF_\ell}$-actions.  For $1\leq i \leq 2g$, define $\alpha_i':=\alpha_i \circ \iota$.   Observe that $\alpha_1',\ldots, \alpha_{2g}'$ are the weights, with multiplicity, of $T'$ acting on $A[\ell]$.   The group $M$ is also the group consisting of $e\in \ZZ^{2g}$ satisfying $\prod_{i=1}^{2g}(\alpha_i')^{e_i}=1$.  So we have the same set $\scrA$ when defining $U_{T'}$ (with our ordering of weights $\alpha_1',\ldots, \alpha_{2g}'$).  For each $m\in \ZZ^{2g}$, define the character $\beta'_m:=\prod_{i=1}^{2g} (\alpha_i')^{m_i}$ of $T'$.   Note that $\beta'_m=\beta_m \circ \iota$ for all $m\in \ZZ^{2g}$.    For any $t\in T(\FF_\ell)$, we have 
\[
t \in U_T \iff \beta_m(t)=1 \text{ for some $m\in \scrA$} \iff \beta'_m(hth^{-1})=1 \text{ for some $m\in \scrA$} \iff  hth^{-1} \in U_{T'}.
\]
Since $\iota$ induces an isomorphism $T'(\FF_\ell)\to T(\FF_\ell)$, we have $U_{T'}=hU_T h^{-1}$.
\end{proof}

\begin{lemma} \label{L:GRH Bell}
Let $\ell$ be a prime for which $\calG_{A,\ell}$ is reductive.   There is a subset $\calB_\ell$ of $\calG_{A,\ell}(\FF_\ell)$ stable under conjugation satisfying $|\calB_\ell|/|\calG_{A,\ell}(\FF_\ell)| = 1 + O_g(1/\ell)$ such that if $\p\nmid \ell$ is a prime ideal of $\OO_K$ for which $A$ has good reduction and  $\bbar\rho_{A,\ell}(\Frob_\p) \in \calB_\ell$, then $\Phi_{A,\p}$ is a free abelian group of rank $r$. 
\end{lemma}
\begin{proof}
The group $G:=(\calG_{A,\ell})_{\FF_\ell}$ is connected and reductive.  
Let $\calB_\ell$ be the set of elements in $G(\FF_\ell)- \bigcup_T U_T$ that are semisimple and regular in $G$, where the union is over all maximal tori of $G$ and the sets $U_T$ are as in Lemma~\ref{L:GRH Bell new}.  Using property (\ref{I:Bell c}) of Lemma~\ref{L:GRH Bell new}, we find that $\calB_\ell$ is stable under conjugation by $G$.

Take any prime ideal $\p\nmid \ell$ of $\OO_K$ for which $A$ has good reduction and $\bbar\rho_{A,\ell}(\Frob_\p) \in \calB_\ell$.   In particular, $\bbar\rho_{A,\ell}(\Frob_\p)$ is conjugate in $G(\FF_\ell)$ to an element of $T(\FF_\ell)-U_T$ for some maximal torus $T$ of $G$.  Property (\ref{I:Bell a}) of Lemma~\ref{L:GRH Bell new} implies that $\Phi_{A,\p}\cong \ZZ^r$.  

It remains to prove that $|\calB_\ell|/|\calG_{A,\ell}(\FF_\ell)| = 1 + O_g(1/\ell)$.  Let $G(\FF_\ell)_{\rs}$ be the set of elements in $G(\FF_\ell)$ that are regular and semisimple in $G$.   For each maximal torus $T$ of $G$, define $T(\FF_\ell)_{\rs}=T(\FF_\ell)\cap G(\FF_\ell)_{\rs}$.   We have $|G(\FF_\ell)_{\rs}| = |G(\FF_\ell)| (1+O_g(1/\ell))$ and $|T(\FF_\ell)_{\rs}| = \ell^r + O_g(\ell^{r-1})$ for any maximal torus $T$ of $G$ by  the proof of Lemma~4.5 of \cite{MR3038553}; note that the proof of this lemma only uses that $G/\FF_\ell$ is reductive and there are only a finite number of possibilities, in terms of $g$, for the  Lie type of $G$.

Every element of $G(\FF_\ell)$ that is regular and semisimple element in $G$ lies in a unique maximal torus.  We thus have a disjoint union 
$\calB_\ell=\bigcup_T \big(T(\FF_\ell)_{\rs} - U_T\big)$, with the union being over all maximal tori $T$ of $G$.  Therefore,
\begin{align*}
|\calB_\ell| &\geq \sum_T \big(|T(\FF_\ell)_{\rs}| - |U_T|\big) = \sum_T \ell^r \cdot (1+ O_g(1/\ell)),
\end{align*}
where we have used property (\ref{I:Bell b}) of Lemma~\ref{L:GRH Bell new}.  We also have a disjoint union $G(\FF_\ell)_{\rs}=\bigcup_T T(\FF_\ell)_{\rs}$ and hence $|G(\FF_\ell)_{\rs}| = \sum_T \ell^r \cdot (1+O_g(1/\ell))$.  Since $|G(\FF_\ell)_{\rs}| = |G(\FF_\ell)| (1+O_g(1/\ell))$, we have inequalities
 $|G(\FF_\ell)| \geq |\calB_\ell| \geq |G(\FF_\ell)| (1+O_g(1/\ell))$.  Therefore, $|\calB_\ell|/|\calG_{A,\ell}(\FF_\ell)| = 1 + O_g(1/\ell)$.
\end{proof}

Let $\q$ be a non-zero prime ideal of $\OO_K$ for which $A$ has good reduction and $\Phi_{A,\q}$ is a free abelian group of rank $r$.   We can assume that $\q$ is chosen so that $N(\q)$ is minimal.   We have $D=\prod_{p\in V} p$, where $V$ is the set of primes $p$ that ramify in $K$ or are divisible by a prime ideal for which $A$ has bad reduction.  \\

By Theorem~\ref{T:main new revised}, there are positive constants $c$ and $\gamma$, depending only on $g$, and a positive integer $n < c N(\q)^\gamma$ such that  for all primes $\ell \nmid n D$ satisfying $\ell \geq c \cdot \max(\{[K:\QQ],h(A)\})^\gamma$, we have
   \begin{align*}
 [\calG_{A,\ell}(\ZZ_\ell): \rho_{A,\ell}(\Gal_K)] &\ll_{g} 1 
 \end{align*}
and the $\ZZ_\ell$-group scheme $\calG_{A,\ell}$ is reductive.   

\begin{lemma} \label{L:existence of ell}
Let $\ell\nmid nD$ be a prime satisfying $\ell \geq c \cdot \max(\{[K:\QQ],h(A)\})^\gamma$.    There is a subset $\calC_\ell$ of $\bbar\rho_{A,\ell}(\Gal_K)$ that is stable under conjugation such that the following hold:
\begin{alphenum}
\item \label{L:existence of ell a}
if $\p\nmid \ell$ is a prime ideal of $\OO_K$ for which $A$ has good reduction and $\bbar\rho_{A,\ell}(\Frob_p)\in \calC_\ell$, then $\Phi_{A,\p}$ is a free abelian group of rank $r$.
\item \label{L:existence of ell b}
${|\calC_\ell|}/{|\bbar\rho_{A,\ell}(\Gal_K)|} =1 +O_g(\ell)$.
\end{alphenum}
\end{lemma}
\begin{proof}
Take any prime $\ell\nmid nD$ satisfying $\ell \geq c \cdot \max(\{[K:\QQ],h(A)\})^\gamma$.   From above, we know that $\calG_{A,\ell}$ is reductive and $[\calG_{A,\ell}(\ZZ_\ell): \rho_{A,\ell}(\Gal_K)] \ll_{g} 1$.  Let $\calB_\ell$ be the set of elements in $\calG_{A,\ell}(\FF_\ell)$ as in Lemma~\ref{L:GRH Bell}.  Define $\calC_\ell:=  \bbar\rho_{A,\ell}(\Gal_K) \cap \calB_\ell$; it is stable under conjugation by $\bbar\rho_{A,\ell}(\Gal_K)$.  Take any prime ideal $\p\nmid \ell$ of $\OO_K$ for which $A$ has good reduction and $\bbar\rho_{A,\ell}(\Frob_p)\in \calC_\ell$.  Since $\bbar\rho_{A,\ell}(\Frob_p)\in \calB_\ell$, the group $\Phi_{A,\p}$ is free abelian of rank $r$.  This proves part (\ref{L:existence of ell a}).

We have  $\bbar\rho_{A,\ell}(\Gal_K) - \calC_\ell \subseteq \calG_{A,\ell}(\FF_\ell)- \calB_\ell$ and hence
\[
|\bbar\rho_{A,\ell}(\Gal_K) - \calC_\ell| \leq |\calG_{A,\ell}(\FF_\ell)- \calB_\ell| = |\calG_{A,\ell}(\FF_\ell)| (1 - |\calB_\ell|/|\calG_{A,\ell}(\FF_\ell)|) \ll_g |\calG_{A,\ell}(\FF_\ell)|/\ell,
\]
where the last inequality uses that $|\calB_\ell|/|\calG_{A,\ell}(\FF_\ell)| = 1+ O_g(1/\ell)$.   Using that $[\calG_{A,\ell}(\FF_\ell): \bbar\rho_{A,\ell}(\Gal_K)]\leq [\calG_{A,\ell}(\ZZ_\ell): \rho_{A,\ell}(\Gal_K)] \ll_{g} 1$, we deduce that
\[
|\bbar\rho_{A,\ell}(\Gal_K)| - |\calC_\ell| = |\bbar\rho_{A,\ell}(\Gal_K) - \calC_\ell|\ll_g |\bbar\rho_{A,\ell}(\Gal_K)|/\ell.
\]
Part (\ref{L:existence of ell b}) follows by dividing by $|\bbar\rho_{A,\ell}(\Gal_K)|$.
\end{proof}

\begin{prop} \label{P:need GRH}
Take any prime $\ell\nmid nD$ satisfying $\ell\geq c \cdot \max(\{[K:\QQ],h(A)\})^\gamma$.   After possibly increasing the constant $c$, that depends only on $g$,  we have $N(\q)\ll_g (\max\{\ell, [K:\QQ], \log D\})^e,$ where $e\geq 1$ depends only on $g$.
\end{prop}
\begin{proof}
Take any prime $\ell\nmid nD$ satisfying $\ell\geq c \cdot \max(\{[K:\QQ],h(A)\})^\gamma$.   Define the group $G:=\bbar\rho_{A,\ell}(\Gal_K)$ and the field $L:=K(A[\ell])$. Note that $L$ is the subfield of $\Kbar$ fixed by $\ker \bbar\rho_{A,\ell}$.    
Let $\calC_\ell\subseteq G$ be the set from Lemma~\ref{L:existence of ell}.  By increasing the constant $c$, that only depends on $g$, we may assume that $|\calC_\ell|/|G| \geq 1/2$.

Let $\pi_{\calC_\ell}(x)$ be the set of non-zero prime ideals $\p$ of $\OO_K$ that are unramified in $L$ and satisfy $\bbar\rho_{A,\ell}(\Frob_\p)\in \calC_\ell$.   An effective version of the Chebotarev density theorem (Th\'eor\`eme~4 and Remark ($20_R$) of \cite{MR644559} along with the trivial bound $|\calC_\ell|\leq [L:K]$) implies that
\[
\Big|\pi_{\calC_\ell}(x) -  \frac{|\calC_\ell|}{|G|}  \operatorname{Li}(x) \Big| \ll [L:K] x^{1/2} \Big( \log x + \log [L:\QQ] + [K:\QQ]^{-1} \log d_K + \sum_{p\in P(L/K)} \log p\Big),
\]
where $\operatorname{Li}(x)=\int^x_2 (\log t)^{-1}\, dt$, $d_K$ is the absolute value of the discriminant of $K$, and $P(L/K)$ is the set of primes $p$ that are divisible by some prime ideal of $\OO_K$ that ramifies in $L$.   Note that this version of the Chebotarev density theorem uses our GRH assumption.
By \cite{MR644559}*{Proposition~6}, we have 
\[
[K:\QQ]^{-1}\log d_K \leq \sum_{p\in P(K)} \log p + |P(K)| \log [K:\QQ] \ll (\log [K:\QQ]+1) \sum_{p\in P(K)} \log p, 
\]
where $P(K)$ is the set of primes $p$ that ramify in $K$.   Since $P(K)\cup P(L/K) \subseteq V \cup\{\ell\}$, we have
\[
\Big|\pi_{\calC_\ell}(x) -  \frac{|\calC_\ell|}{|G|}  \operatorname{Li}(x) \Big| \ll [L:\QQ] x^{1/2} \Big( \log x + \log [L:\QQ] + (\log [K:\QQ]+1) ({\sum}_{p\in V} \log p + \log \ell)\Big).
\]
Since $[L:K] \leq |\GL_{2g}(\FF_\ell)| \leq \ell^{4g^2}$, we find that 
\[
\Big|\pi_{\calC_\ell}(x) -  \frac{|\calC_\ell|}{|G|}  \operatorname{Li}(x) \Big| \ll_g \ell^{4g^2+1}[K:\QQ]^2 x^{1/2} \Big( \log x +  {\sum}_{p\in V} \log p \Big)
\]
and hence
\[
\pi_{\calC_\ell}(x) \geq  \frac{1}{2} \operatorname{Li}(x) +O_g\Big( \ell^{4g^2+1}[K:\QQ]^2 x^{1/2} (\log x+\log D)  \Big).
\]
So there is an $e\geq 2$, depending only on $g$, such that if $\max\{\ell, [K:\QQ], \log D\} \ll x^{1/e}$, then $\pi_{\calC_\ell}(x) \geq  \frac{1}{2} \operatorname{Li}(x) +O_g(x^{9/10})$.   Thus for $x\gg_g (\max\{\ell, [K:\QQ], \log D\})^e$, we will have $\pi_{\calC_\ell}(x) \geq \frac{1}{4} \operatorname{Li}(x)$ and also $\frac{1}{4} \operatorname{Li}(x)\geq 2 [K:\QQ](\log D+1)+1$ after possibly increasing $e$ (which depends only on $g$).   Therefore, for $x\gg_g (\max\{\ell, [K:\QQ], \log D\})^e$, we have $\pi_{\calC_\ell}(x)\geq 2 [K:\QQ](\log D+1)+1$ and hence $\pi_{\calC_\ell}(x)$ is strictly larger than the number of prime ideals $\p$ of $\OO_K$ dividing $D\ell$.  So there is a non-zero prime ideal $\p\nmid D\ell$ with $N(\p)\ll_g (\max\{\ell, [K:\QQ], \log D\})^e$ for which $\bbar\rho_{A,\ell}(\Frob_\p) \in \calC_{\ell}$.   The abelian variety $A$ has good reduction at $\p$ since $\p\nmid D$.   By Lemma~\ref{L:existence of ell}, $\Phi_{A,\p}$ is a free abelian group of rank $r$.     By the minimality of our choice of $\q$, we have $N(\q) \leq N(\p) \ll_g (\max\{\ell, [K:\QQ], \log D\})^e$.
\end{proof}

We have
\[
\sum_{\ell | nD} \log \ell  \leq\log n + \log D \leq \log (c N(\q)^{\gamma}) + \log D =\gamma \log N(\q) +\log c + \log D
\]

Define $C:= c \cdot \max(\{[K:\QQ],h(A)\})^\gamma$.  By the prime number theorem, there is an absolute constant $m\geq 2$ such that for all $Q\geq mC$, we have
\[
\sum_{C\leq \ell \leq Q} \log\ell \geq Q/2
\]
So with $Q:= 4 \max\{mC, \gamma \log N(\q) +\log c + \log D\}$, we have 
\[
\sum_{C\leq \ell \leq Q} \log\ell \geq 2 \max\{mC, \gamma \log N(\q) +\log c + \log D\} > \gamma \log N(\q) +\log c + \log D \geq \sum_{\ell | nD} \log \ell.
\]
From $\sum_{C\leq \ell \leq Q} \log\ell > \sum_{\ell | nD} \log \ell$, we deduce that there is a prime $C\leq \ell \leq Q$ with $\ell \nmid nD$.  By Proposition~\ref{P:need GRH}, we have
\begin{align*}
N(\q)&\ll_g (\max\{Q, [K:\QQ], \log D\})^e \\
& \ll_g   (\max\{C,\log N(\q), [K:\QQ], \log D\})^e\\
& \ll_g   (\max\{\log N(\q), [K:\QQ], h(A), \log D\})^f,
\end{align*}
where $e\geq 1$ and $f\geq 1$ depend only on $g$.   If $\log N(\q) \leq \max\{ [K:\QQ], h(A), \log D\}$, then
\begin{align} \label{E:Nq bound}
N(\q) \ll_g (\max\{[K:\QQ], h(A), \log D\})^f.
\end{align}
Now suppose that $\log N(\q) > \max\{ [K:\QQ], h(A), \log D\}$ and hence $N(\q) \ll_g (\log N(\q))^f$.  Since $f$ depends only on $g$, we have $N(\q) \ll_g 1$ and hence (\ref{E:Nq bound}) holds as well.

Theorem~\ref{T:GRH bound for Nq} is now a direct consequence of Theorem~\ref{T:main new} and the upper bound (\ref{E:Nq bound}) for $N(\q)$.

\section{Proof of Corollary~\ref{C:maximal monodromy}} 
\label{S:proof of maximal monodromy}
 
Take any prime $\ell\geq c \cdot \max(\{[K:\QQ],h(A), N(\q)\})^\gamma$ that is unramified in $K$, where $c$ and $\gamma$ are constants as in Theorem~\ref{T:main new}.   After possibly increasing the constants $c$ and $\gamma$, that depend only on $g$, Th\'eor\`eme~1.1 of \cite{MR3263028} implies that $A$ has a polarization defined over $K$ whose degree is not divisible by $\ell$.   This polarization gives rise to an isogeny $\varphi\colon A\to A^\vee$ whose degree is not divisible by $\ell$, where $A^\vee$ is the dual abelian variety of $A$.   Combining the Weil pairing of $A$ with $\varphi$ gives rise to a non-degenerate skew-symmetric form of $\ZZ_\ell$-modules
\[
e_\ell\colon T_\ell (A) \times T_\ell (A)  \xrightarrow{\text{id}\times \varphi } T_\ell (A) \times T_\ell (A^\vee)\to \ZZ_\ell(1)\cong \ZZ_\ell
\]
such that $e_\ell(\sigma(P),\sigma(Q))=\chi_\ell(\sigma) e_\ell(P,Q)$ for all $P,Q\in T_\ell(A)$ and $\sigma\in \Gal_K$, where $\chi_\ell\colon \Gal_K \to \ZZ_\ell^\times$ is the $\ell$-adic cyclotomic character.  We thus have
\begin{align} \label{E:GSp}
\rho_{A,\ell}\colon \Gal_K \to \GSp(T_\ell(A),e_\ell)\cong \GSp_{2g}(\ZZ_\ell),
\end{align}
where the last isomorphism depends on a suitable choice of a $\ZZ_\ell$-basis of $T_\ell(A)$.    We have $\chi_\ell(\Gal_K)=\ZZ_\ell^\times$ since $\ell$ is unramified in $K$.  So to prove that $\rho_{A,\ell}(\Gal_K)=\GSp_{2g}(\ZZ_\ell)$ it suffices to show that $\rho_{A,\ell}(\Gal_K) \supseteq \Sp_{2g}(\ZZ_\ell)$.

From (\ref{E:GSp}), we may identify $G_{A,\ell}$ with a closed subgroup of $\GSp_{2g,\QQ_\ell}$.

\begin{lemma} \label{L:GSp monodromy}
We have $G_{A,\ell}=\GSp_{2g,\QQ_\ell}$.
\end{lemma}
\begin{proof}
We have $G_{A,\ell}\subseteq \GSp_{2g,\QQ_\ell}$ and hence the rank $r$ of $G_{A,\ell}^\circ$ is at most $g+1$, i.e., the rank of $\GSp_{2g,\QQ_\ell}$.  By assumption, we have a prime ideal $\q \subseteq \OO_K$ for which $A$ has good reduction and for which the group $\Phi_{A,\q}$ is free abelian of rank $g+1$.     By Lemma~\ref{L:pre common rank}(\ref{L:pre common rank i}), we have $g+1 \leq r$.    Therefore, $r=g+1$.

Our assumption $\End(A_{\Kbar})=\ZZ$ and Proposition~\ref{P:conn facts}(\ref{P:conn facts iii}) implies that the commutant of $G_{A,\ell}^\circ$ in $\End_{\QQ_\ell}(V_\ell(A))$ agrees with the scalar endomorphisms $\QQ_\ell$.   The commutant of $\GSp_{2g,\QQ_\ell}$ in $\End_{\QQ_\ell}(V_\ell(A))$ is also $\QQ_\ell$.  By Lemma~\ref{L:reductive inclusion}, we deduce that $G_{A,\ell}^\circ=\GSp_{2g,\QQ_\ell}$ and hence $G_{A,\ell}=\GSp_{2g,\QQ_\ell}$
\end{proof}

By Lemma~\ref{L:GSp monodromy} and (\ref{E:GSp}), we have $\calG_{A,\ell}=\GSp_{2g,\ZZ_\ell}$.    By Theorem~\ref{T:main new}(\ref{T:main new d}), we have 
\[
\rho_{A,\ell}(\Gal_K) \supseteq \GSp_{2g}(\ZZ_\ell)' \supseteq \Sp_{2g}(\ZZ_\ell)'.
\]
It remains to prove that $\Sp_{2g}(\ZZ_\ell)'=\Sp_{2g}(\ZZ_\ell)$.

With notation as in \S\ref{SS:derived calG}, we have $\calS_{A,\ell}=\Sp_{2g,\ZZ_\ell}$.  By Proposition~\ref{P:calS}(\ref{P:calS i}) and Lemma~\ref{L:calS mod ell basics},  with appropriate $c$ and $\gamma$, it suffices to show that $\calS_{A,\ell}(\FF_\ell)=\Sp_{2g}(\FF_\ell)$ is generated by elements of order $\ell$.  This is indeed true; moreover, $\Sp_{2g}(\FF_\ell)$ is generated by symplectic transvections.


\begin{bibdiv}
\begin{biblist}

\bib{MR574307}{article}{
      author={Bogomolov, Fedor~Alekseivich},
       title={Sur l'alg\'ebricit\'e des repr\'esentations {$l$}-adiques},
        date={1980},
        ISSN={0151-0509},
     journal={C. R. Acad. Sci. Paris S\'er. A-B},
      volume={290},
      number={15},
       pages={A701\ndash A703},
      review={\MR{MR574307 (81c:14025)}},
}


\bib{MR2372809}{article}{
   author={Caruso, Xavier},
   title={Conjecture de l'inertie mod\'{e}r\'{e}e de Serre},
   language={French, with French summary},
   journal={Invent. Math.},
   volume={171},
   date={2008},
   number={3},
   pages={629--699},
   issn={0020-9910},
   review={\MR{2372809}},
   doi={10.1007/s00222-007-0091-9},
}

\bib{MR861978}{article}{
   author={Chai, Ching-Li},
   title={Siegel moduli schemes and their compactifications over ${\bf C}$},
   conference={
      title={Arithmetic geometry},
      address={Storrs, Conn.},
      date={1984},
   },
   book={
      publisher={Springer, New York},
   },
   date={1986},
   pages={231--251},
   review={\MR{861978}},
}


\bib{MR3362641}{article}{
   author={Conrad, Brian},
   title={Reductive group schemes},
   language={English, with English and French summaries},
   conference={
      title={Autour des sch\'{e}mas en groupes. Vol. I},
   },
   book={
      series={Panor. Synth\`eses},
      volume={42/43},
      publisher={Soc. Math. France, Paris},
   },
   date={2014},
   pages={93--444},
   review={\MR{3362641}},
}

\bib{MR632548}{book}{
   author={Curtis, Charles W.},
   author={Reiner, Irving},
   title={Methods of representation theory. Vol. I},
   note={With applications to finite groups and orders;
   Pure and Applied Mathematics;
   A Wiley-Interscience Publication},
   publisher={John Wiley \&\ Sons, Inc., New York},
   date={1981},
   pages={xxi+819},
   isbn={0-471-18994-4},
   review={\MR{632548}},
}


\bib{MR654325}{book}{
   author={Deligne, Pierre},
   author={Milne, James S.},
   author={Ogus, Arthur},
   author={Shih, Kuang-yen},
   title={Hodge cycles, motives, and Shimura varieties},
   series={Lecture Notes in Mathematics},
   volume={900},
   publisher={Springer-Verlag, Berlin-New York},
   date={1982},
   pages={ii+414},
   isbn={3-540-11174-3},
   review={\MR{654325}},
}

\bib{1008.3675}{article}{
   author={Ellenberg, Jordan S.},
   author={Hall, Chris},
   author={Kowalski, Emmanuel},
   title={Expander graphs, gonality, and variation of Galois
   representations},
   journal={Duke Math. J.},
   volume={161},
   date={2012},
   number={7},
   pages={1233--1275},
   issn={0012-7094},
   review={\MR{2922374}},
}

\bib{MR861971}{incollection}{
      author={Faltings, Gerd},
       title={Finiteness theorems for abelian varieties over number fields},
        date={1986},
   booktitle={Arithmetic geometry ({S}torrs, {C}onn., 1984)},
   publisher={Springer},
     address={New York},
       pages={9\ndash 27},
        note={Translated from the German original [Invent. Math. {{\bf{7}}3}
  (1983), no. 3, 349--366; ibid. {{\bf{7}}5} (1984), no. 2, 381; MR
  85g:11026ab] by Edward Shipz},
      review={\MR{MR861971}},
}

\bib{MR1175627}{book}{
   author={Faltings, Gerd},
   author={W\"ustholz, Gisbert},
   author={Grunewald, Fritz},
   author={Schappacher, Norbert},
   author={Stuhler, Ulrich},
   title={Rational points},
   series={Aspects of Mathematics, E6},
   edition={3},
   note={Papers from the seminar held at the Max-Planck-Institut f\"ur
   Mathematik, Bonn/Wuppertal, 1983/1984;
   With an appendix by W\"ustholz},
   publisher={Friedr. Vieweg \& Sohn, Braunschweig},
   date={1992},
   pages={x+311},
   isbn={3-528-28593-1},
   review={\MR{1175627}},
   doi={10.1007/978-3-322-80340-5},
}


\bib{MR3263028}{article}{
   author={Gaudron, \'{E}ric},
   author={R\'{e}mond, Ga\"{e}l},
   title={Polarisations et isog\'{e}nies},
   language={French, with English and French summaries},
   journal={Duke Math. J.},
   volume={163},
   date={2014},
   number={11},
   pages={2057--2108},
   issn={0012-7094},
}


\bib{MR693314}{incollection}{
      author={Henniart, Guy},
       title={Repr\'esentations {$l$}-adiques ab\'eliennes},
        date={1982},
   booktitle={Seminar on {N}umber {T}heory, {P}aris 1980-81 ({P}aris,
  1980/1981)},
      series={Progr. Math.},
      volume={22},
   publisher={Birkh\"auser Boston, Boston, MA},
       pages={107\ndash 126},
      review={\MR{693314 (85d:11070)}},
}

\bib{MR3038553}{article}{
   author={Jouve, F.},
   author={Kowalski, E.},
   author={Zywina, D.},
   title={Splitting fields of characteristic polynomials of random elements
   in arithmetic groups},
   journal={Israel J. Math.},
   volume={193},
   date={2013},
   number={1},
   pages={263--307},
   issn={0021-2172},
   review={\MR{3038553}},
   doi={10.1007/s11856-012-0117-x},
}

\bib{MR2832632}{article}{
      author={Larsen, Michael},
       title={Exponential generation and largeness for compact {$p$}-adic {L}ie
  groups},
        date={2010},
        ISSN={1937-0652},
     journal={Algebra Number Theory},
      volume={4},
      number={8},
       pages={1029\ndash 1038},
         url={http://dx.doi.org/10.2140/ant.2010.4.1029},
      review={\MR{2832632}},
}


\bib{MR1441234}{article}{
      author={Larsen, Michael},
      author={Pink, Richard},
       title={A connectedness criterion for {$l$}-adic {G}alois
  representations},
        date={1997},
        ISSN={0021-2172},
     journal={Israel J. Math.},
      volume={97},
       pages={1\ndash 10},
      review={\MR{MR1441234 (98k:11066)}},
}

\bib{larsen-pink-finite_groups}{article}{
      author={Larsen, Michael},
      author={Pink, Richard},
       title={Finite subgroups of algebraic groups},
        date={2011},
     journal={J. Amer. Math. Soc.},
      volume={24},
       pages={1105\ndash 1158},
}



\bib{explicit}{article}{
      author={Lombardo, Davide},
       title={Explicit open image theorems for abelian varieties with trivial endomorphism ring},
        date={2015},
      eprint={arXiv:1508.01293},
        note={arXiv:1508.01293},
}


\bib{MR1336608}{incollection}{
      author={Masser, D.~W.},
      author={W{\"u}stholz, G.},
       title={Refinements of the {T}ate conjecture for abelian varieties},
        date={1995},
   booktitle={Abelian varieties ({E}gloffstein, 1993)},
   publisher={de Gruyter},
     address={Berlin},
       pages={211\ndash 223},
      review={\MR{MR1336608 (97a:11092)}},
}


\bib{MR2544610}{article}{
   author={Miller, Alexander},
   author={Reiner, Victor},
   title={Differential posets and Smith normal forms},
   journal={Order},
   volume={26},
   date={2009},
   number={3},
   pages={197--228},
   issn={0167-8094},
   review={\MR{2544610}},
}

\bib{MR880952}{article}{
      author={Nori, Madhav~V.},
       title={On subgroups of {${\rm GL}_n({\bf F}_p)$}},
        date={1987},
        ISSN={0020-9910},
     journal={Invent. Math.},
      volume={88},
      number={2},
       pages={257\ndash 275},
}

\bib{MR3566639}{article}{
   author={Petersen, Sebastian},
   title={Group-theoretical independence of $\ell$-adic Galois
   representations},
   journal={Acta Arith.},
   volume={176},
   date={2016},
   number={2},
   pages={161--176},
   issn={0065-1036},
   review={\MR{3566639}},
}

\bib{MR0457455}{article}{
      author={Ribet, Kenneth~A.},
       title={Galois action on division points of {A}belian varieties with real
  multiplications},
        date={1976},
        ISSN={0002-9327},
     journal={Amer. J. Math.},
      volume={98},
      number={3},
       pages={751\ndash 804},
      review={\MR{MR0457455 (56 \#15660)}},
}




\bib{MR0387283}{article}{
   author={Serre, Jean-Pierre},
   title={Propri\'{e}t\'{e}s galoisiennes des points d'ordre fini des courbes
   elliptiques},
   language={French},
   journal={Invent. Math.},
   volume={15},
   date={1972},
   number={4},
   pages={259--331},
   issn={0020-9910},
   review={\MR{0387283}},
   doi={10.1007/BF01405086},
}

\bib{MR0476753}{incollection}{
      author={Serre, Jean-Pierre},
       title={Repr\'esentations {$l$}-adiques},
        date={1977},
   booktitle={Algebraic number theory ({K}yoto {I}nternat. {S}ympos., {R}es.
  {I}nst. {M}ath. {S}ci., {U}niv. {K}yoto, {K}yoto, 1976)},
   publisher={Japan Soc. Promotion Sci.},
     address={Tokyo},
       pages={177\ndash 193},
      review={\MR{MR0476753 (57 \#16310)}},
}



\bib{MR644559}{article}{
   author={Serre, Jean-Pierre},
   title={Quelques applications du th\'{e}or\`eme de densit\'{e} de Chebotarev},
   language={French},
   journal={Inst. Hautes \'{E}tudes Sci. Publ. Math.},
   number={54},
   date={1981},
   pages={323--401},
   issn={0073-8301},
   review={\MR{644559}},
}


\bib{MR1484415}{book}{
      author={Serre, Jean-Pierre},
       title={Abelian {$l$}-adic representations and elliptic curves},
      series={Research Notes in Mathematics},
   publisher={A K Peters Ltd.},
     address={Wellesley, MA},
        date={1998},
      volume={7},
        ISBN={1-56881-077-6},
        note={With the collaboration of Willem Kuyk and John Labute, Revised
  reprint of the 1968 original},
      review={\MR{MR1484415 (98g:11066)}},
}

\bib{MR1730973}{book}{
      author={Serre, Jean-Pierre},
       title={{\OE}uvres. {C}ollected papers. {IV}},
   publisher={Springer-Verlag},
     address={Berlin},
        date={2000},
        ISBN={3-540-65683-9},
        note={1985--1998},
      review={\MR{MR1730973 (2001e:01037)}},
}

\bib{MR0354656}{book}{
   title={Groupes de monodromie en g\'eom\'etrie alg\'ebrique. I},
   language={French},
   series={Lecture Notes in Mathematics, Vol. 288},
   note={S\'eminaire de G\'eom\'etrie Alg\'ebrique du Bois-Marie 1967--1969 (SGA 7
   I);
   Dirig\'e par A. Grothendieck. Avec la collaboration de M. Raynaud et D. S.
   Rim},
   publisher={Springer-Verlag, Berlin-New York},
   date={1972},
   pages={viii+523},
   review={\MR{0354656}},
   label={SGA7 I},
}



\bib{MR3117743}{article}{
   author={Ullmo, Emmanuel},
   author={Yafaev, Andrei},
   title={Mumford-Tate and generalised Shafarevich conjectures},
   language={English, with English and French summaries},
   journal={Ann. Math. Qu\'e.},
   volume={37},
   date={2013},
   number={2},
   pages={255--284},
   issn={2195-4755},
   review={\MR{3117743}},
   doi={10.1007/s40316-013-0009-4},
}

\bib{MR2400251}{article}{
   author={Vasiu, Adrian},
   title={Some cases of the Mumford-Tate conjecture and Shimura varieties},
   journal={Indiana Univ. Math. J.},
   volume={57},
   date={2008},
   number={1},
   pages={1--75},
   issn={0022-2518},
   review={\MR{2400251}},
   doi={10.1512/iumj.2008.57.3513},
}



\bib{MR1944805}{article}{
      author={Wintenberger, J.-P.},
       title={D\'emonstration d'une conjecture de {L}ang dans des cas
  particuliers},
        date={2002},
        ISSN={0075-4102},
     journal={J. Reine Angew. Math.},
      volume={553},
       pages={1\ndash 16},
      review={\MR{MR1944805 (2003i:11075)}},
}

\bib{Zywina-Large families}{article}{
	author={Zywina, David},
	title={Families of abelian varieties and large Galois images (preprint)},
	date={2019},
}

\end{biblist}
\end{bibdiv}


\end{document}